\documentclass{amsart}%
\usepackage{listings}
\usepackage{color}
\usepackage[title,toc,page,header]{appendix}
\usepackage{amsfonts}
\usepackage{amsmath}
\usepackage{geometry}
\usepackage{amssymb}
\usepackage{subfig}
\usepackage{graphicx}
\providecommand{\U}[1]{\protect\rule{.1in}{.1in}}

\newtheorem{theorem}{Theorem}[section]

\newtheorem{proposition}[theorem]{Proposition}
\newtheorem{corollary}[theorem]{Corollary}
\theoremstyle{definition}

\theoremstyle{remark}
\newtheorem{remark}[theorem]{Remark}
\numberwithin{equation}{section}

\DeclareMathOperator{\cond}{cond}

\begin{document}

\title{Numerical differentiation on scattered data through multivariate 
polynomial interpolation}

\author[F. Dell'Accio]{F. Dell'Accio}
\address{Department of Mathematics and Computer Science, University of Calabria, Italy}
\email{fdellacc@unical.it}
\author[F. Di Tommaso]{F. Di Tommaso}
\address{Department of Mathematics and Computer Science, University of Calabria, Italy}
\email{filomena.ditommaso@unical.it}
\author[N. Siar]{N. Siar}
\address{Department of Mathematics and Computer Science, University of Calabria, Italy\\Department of Mathematics, Ibn Tofail University, Kenitra, Morocco}
\email{najoua.siar@uit.ac.ma}
\author[M. Vianello]{M. Vianello}
\address{University of Padova, Italy}
\email{marcov@math.unipd.it}

\begin{abstract}
We discuss a pointwise numerical differentiation formula
on multivariate scattered data, based on the coefficients of
local polynomial interpolation at Discrete Leja Points, written in Taylor's formula monomial basis. Error bounds for the
approximation of
partial derivatives of any order compatible with the function
regularity are provided, as well as sensitivity estimates to functional
perturbations, in terms of the inverse 
Vandermonde coefficients 
that are active in the differentiation process.
Several numerical tests are presented showing the accuracy of
the approximation.
\end{abstract}


\maketitle

\section{Introduction}

\label{section1}Let $\Pi _{d}\left( \mathbb{R}^{s}\right) $ be the space of
polynomials of total degree at most $d$ in the variable $\mathbf{x}=\left(
\xi _{1},\dots ,\xi _{s}\right) $. A basis for this space, in the
multi-index notation, is given by the monomials $\mathbf{x}^{\alpha }\mathbf{%
=}\xi _{1}^{\alpha _{1}}\dots \xi _{s}^{\alpha _{s}}$, where $\alpha =\left(
\alpha _{1},\dots ,\alpha _{s}\right) \in 
\mathbb{N}
_{0}^{s}$, $\left\vert \alpha \right\vert =\alpha _{1}+\dots +\alpha
_{s}\leq d$ and therefore $\dim \Pi _{d}\left( \mathbb{R}^{s}\right) =\binom{%
d+s}{s}$. We introduce a total order in the set of all multi-indices $\alpha 
$. More precisely, we assume $\alpha <\beta $ if $\left\vert \alpha
\right\vert <\left\vert \beta \right\vert $, otherwise, if $\left\vert
\alpha \right\vert =\left\vert \beta \right\vert $ we follow the
lexicographic order of the dictionary of words of $\left\vert \alpha
\right\vert $ letters from the ordered alphabet $\left\{ \xi _{1},\dots ,\xi
_{s}\right\} $ with the possibility to repeat each letter $\xi _{i}$ only
consecutively many times. For instance, if $\left\vert \alpha \right\vert =3$%
, we have $\left( 3,0,0\right) <\left( 2,1,0\right) <\left( 2,0,1\right)
<\left( 1,2,0\right) <\left( 1,1,1\right) <\left( 1,0,2\right) <\left(
0,3,0\right) <\left( 0,2,1\right) <\left( 0,1,2\right) <\left( 0,0,3\right) $%
. Further details on multivariate polynomials and related multi-index
notations can be found in \cite[Ch. $4$]{cheney2009course}.

Let us consider a set 
\begin{equation}
\sigma =\left\{ \mathbf{x}_{1},\dots ,\mathbf{x}_{m}\right\}
\end{equation}
of $m=\binom{d+s}{s}$ pairwise distinct points in $%
\mathbb{R}
^{s}$ and let us assume that they are unisolvent for Lagrange interpolation
in $\Pi _{d}\left( \mathbb{R}^{s}\right) $, that is for any choice of $%
y_{1},\dots ,y_{m}\in \mathbb{R}$, there exists and it is unique $p\in \Pi
_{d}\left( \mathbb{R}^{s}\right) $ satisfying 
\begin{equation}
p(\mathbf{x}_{i})=y_{i},\text{ }i=1,\dots ,m.  \label{interpolation_cond}
\end{equation}%
An equivalent result \cite[Ch. 1]{cheney2009course} is the non singularity
of the Vandermonde matrix 
\begin{equation*}
V\left( \sigma \right) =\left[ \mathbf{x}_{i}^{\alpha }\right] _{\substack{ %
i=1,\dots ,m  \\ \left\vert \alpha \right\vert \leq d}},
\end{equation*}%
where the index $i$, related to the points, varies along the rows while the
index $\alpha $, related to the powers, increases with the column index by
following the above introduced order. By denoting with $\overline{\mathbf{x}}
$ any point in $\mathbb{R}^{s}$ and by fixing the basis 
\begin{equation}
\left( \mathbf{x}-\overline{\mathbf{x}}\right) ^{\alpha
}:=\prod_{i=1}^{s}\left( \xi _{i}-\overline{x}_{i}\right) ^{\alpha
_{i}}=\left( \xi _{1}-\overline{x}_{1}\right) ^{\alpha _{1}}\left( \xi _{2}-%
\overline{x}_{2}\right) ^{\alpha _{2}}\dots \left( \xi _{s}-\overline{x}%
_{s}\right) ^{\alpha _{s}},  \label{Taylor_basis}
\end{equation}%
the Vandermonde matrix centered at $\overline{\mathbf{x}}$ 
\begin{equation}
V_{\overline{\mathbf{x}}}\left( \sigma \right) =\left[ \left( \mathbf{x}_{i}-%
\overline{\mathbf{x}}\right) ^{\alpha }\right] _{\substack{ i=1,\dots ,m  \\ %
\left\vert \alpha \right\vert \leq d}}  \label{Vandermonde_matrix_bar}
\end{equation}%
is non singular as well \cite[Theorem $3$, Ch. $5$]{cheney2009course}.
Therefore, for any choice of the vector 
\begin{equation}
y=\left[ f(\mathbf{x}_i)\right] _{i=1,\dots ,m}\in 
\mathbb{R}
^{m},
\end{equation}
the solution $p\left[y,\sigma \right] \left( \mathbf{x}\right) $ of the
interpolation problem \eqref{interpolation_cond} in the basis %
\eqref{Taylor_basis}, can be obtained by solving the linear system%
\begin{equation}
V_{\overline{\mathbf{x}}}\left( \sigma \right) c=y  \label{linear_system}
\end{equation}%
and by setting, using matrix notation, 
\begin{equation}
p\left[ y,\sigma \right] \left( \mathbf{x}\right) =\left[ \left( \mathbf{x}-%
\overline{\mathbf{x}}\right) ^{\alpha }\right] _{\left\vert \alpha
\right\vert \leq d}c,  \label{solution_mat}
\end{equation}%
where $c=\left[ c_{\alpha }\right] _{\left\vert \alpha \right\vert \leq
d}^{T}\in \mathbb{R}^{m}$ is the solution of the system (\ref{linear_system}%
). This approach, for $s=2$ and $\overline{\mathbf{x}}$ the barycenter of
the node set $\sigma $, has been recently proposed in \cite%
{dell2020numerical} in connection with the use of the $PA=LU$ factorization
of the matrix $V_{\overline{\mathbf{x}}}\left( \sigma \right) $.

The main goal of the paper is to provide a pointwise numerical
differentiation method of a target function $f$ sampled at scattered points,
by locally  using the interpolation formula (\ref{solution_mat}). The key
tools are the connection to Taylor's formula via the shifted monomial basis (%
\ref{Taylor_basis}), suitably preconditioned by local scaling to reduce the
conditioning of the Vandermonde matrix, together with the extraction of
Leja-like local interpolation subsets from the scattered sampling set via
basic numerical linear algebra. Our approach is complementary to other
existing techniques, based on least-square approximation or on different function spaces, see for example \cite%
{davydov2016error,davydov2018minimal,belward2008derivative,ling2018meshfree} with the references therein. 
In Section $\ref{section3}$ we provide error bounds in approximating
function and derivative values at a given point $\overline{\mathbf{x}}$, as well as sensitivity estimates to perturbations of the function values, and
in Section $\ref{section4}$ we conduct some numerical experiments to show
the accuracy of the proposed method. 

\section{\label{section3}Error bounds and sensitivity estimates}

In the following we assume that $\Omega \subset 
\mathbb{R}
^{s}$ is a convex body containing $\sigma $ and that the sampled function $%
f:\Omega $ $\rightarrow 
\mathbb{R}
$ is of class $C^{d,1}\left( \Omega \right) $, that is $f\in C^{d}\left(
\Omega \right) $ and all its partial derivatives of order $d$ 
\begin{equation*}
D^{\alpha }f=\prod\limits_{i=1}^{s}\frac{\partial ^{\alpha _{i}}f}{\partial
\xi _{i}^{\alpha _{i}}}=\frac{\partial ^{\left\vert \alpha \right\vert }f}{%
\partial \xi _{1}^{\alpha _{1}}\partial \xi _{2}^{\alpha _{2}}\dots \partial
\xi _{s}^{\alpha _{s}}},\text{ }\left\vert \alpha \right\vert =d,
\end{equation*}%
are Lipschitz continuous in $\Omega $. Let $K\subseteq \Omega$ compact convex: we equip the space $C^{d,1}\left(
K \right) $ with the semi-norm \cite{farwig1986rate} 
\begin{equation}
\left\Vert f\right\Vert _{d,1}^{K}=\sup \left\{ \frac{\left\vert D^{\alpha }f(%
\mathbf{u})-D^{\alpha }f(\mathbf{v})\right\vert}{\left\Vert \mathbf{u}-%
\mathbf{v}\right\Vert _{2}}:\mathbf{u},\mathbf{v}\in K ,\text{ }\mathbf{%
u}\neq \mathbf{v},\text{ }\left\vert \alpha \right\vert =d\right\}
\label{seminorm}
\end{equation}%
and we denote by $T_{d}\left[ f,\overline{\mathbf{x}}\right] \left( \mathbf{x%
}\right) $ the truncated Taylor expansion of $f$ of order $d$ centered at $%
\overline{\mathbf{x}}\in \Omega $ 
\begin{equation}
T_{d}\left[ f,\overline{\mathbf{x}}\right] \left( \mathbf{x}\right)
=\sum_{l=0}^{d}\frac{D_{\mathbf{x}-\overline{\mathbf{x}}}^{l}f\left( 
\overline{\mathbf{x}}\right) }{l!}  \label{Taylor_form}
\end{equation}%
and by $R_{T}\left[ f,\overline{\mathbf{x}}\right] \left( \mathbf{x}\right) $
the corresponding remainder term in integral form \cite%
{waldron1998multipoint} 
\begin{equation}
R_{T}\left[ f,\overline{\mathbf{x}}\right] \left( \mathbf{x}\right)
=\int_{0}^{1}\frac{D_{\mathbf{x-}\overline{\mathbf{x}}}^{d+1}f\left( 
\overline{\mathbf{x}}+t\left( \mathbf{x-}\overline{\mathbf{x}}\right)
\right) }{d!}\left( 1-t\right) ^{d}dt,  \label{integral_remainder}
\end{equation}%
where \cite[Ch. $4$]{cheney2009course} 
\begin{equation}
D_{\mathbf{x-}\overline{\mathbf{x}}}^{l}f\left( \mathbf{\cdot }\right)
:=\left( \left[ D^{\beta }f\left( \mathbf{\cdot }\right) \right]
_{\left\vert \beta \right\vert =1}\cdot \left( \mathbf{x-}\overline{\mathbf{x%
}}\right) \right) ^{l}=\sum\limits_{\left\vert \beta \right\vert =l}\frac{l!%
}{\beta !}D^{\beta }f\left( \mathbf{\cdot }\right) \left( \mathbf{x-}%
\overline{\mathbf{x}}\right) ^{\beta },\text{ }l\in 
\mathbb{N}
_{0},  \label{directional_derivatives}
\end{equation}%
with the multi-indices $\beta $ following the order specified in Section \ref%
{section1}.\\
Let us denote by $\ell _{i}\left( \mathbf{x}\right) $ the $i^{th}$ bivariate
fundamental Lagrange polynomial. Since 
\begin{equation*}
\ell _{i}\left( \mathbf{x}_{j}\right) =\delta _{ij}=\left\{ 
\begin{array}{cc}
1, & i=j, \\ 
0, & \text{otherwise,}%
\end{array}%
\right.
\end{equation*}%
by setting for each $i=1,\dots ,m,$%
\begin{equation*}
\begin{array}{c}
\delta ^{i}=\left[ 
\begin{array}{lllllll}
0 & \dots & 0 & 1 & 0 & \dots & 0%
\end{array}%
\right] ^{T} \\ 
\quad \quad \overset{\uparrow }{i^{th}~\text{column}}\label{b_i_coefficients}%
\end{array}%
\end{equation*}%
and by solving the linear system $V_{\overline{\mathbf{x}}}\left( \sigma
\right) a^{i}=\delta ^{i}$, we get the expression of $\ell _{i}\left( 
\mathbf{x}\right) $ in the translated canonical basis (\ref{Taylor_basis}),
that is 
\begin{equation}
\ell _{i}\left( \mathbf{x}\right) =\sum_{\left\vert \alpha \right\vert \leq
d}a_{\alpha }^{i}\left( \mathbf{x}-\overline{\mathbf{x}}\right) ^{\alpha },
\label{lagrange_poly}
\end{equation}%
where $a^{i}=\left[ a_{\alpha }^{i}\right] _{\left\vert \alpha \right\vert
\leq d}^{T}$.

\begin{remark}
Denoting by 
\begin{equation}
h=\max\limits_{i=1,\dots ,m}\left\Vert \mathbf{x}_{i}-\overline{\mathbf{x}}%
\right\Vert _{2},  \label{h_scaling}
\end{equation}%
in order to control the conditioning, it is useful to consider
the scaled canonical polynomial basis centered at $\overline{\mathbf{x}}$ 
\begin{equation}
\left[ \left( \frac{\mathbf{x}-\overline{\mathbf{x}}}{h}\right) ^{\alpha }%
\right] _{\left\vert \alpha \right\vert \leq d},  \label{scaled_basis}
\end{equation}%
in such a way that $\left(\mathbf{x}-\overline{\mathbf{x}}\right)/h$ belongs
to the unit disk (cf. \cite{dell2020numerical}), and to rewrite the
expression (\ref{lagrange_poly}) in the scaled basis (\ref{scaled_basis}) as 
\begin{equation}
\ell _{i}\left( \mathbf{x}\right) =\sum_{\left\vert \alpha \right\vert \leq
d}a_{\alpha ,h}^{i}\left( \frac{\mathbf{x}-\overline{\mathbf{x}}}{h}\right)
^{\alpha },  \label{lagrange_poly_scaled}
\end{equation}%
where $a_{\alpha ,h}^{i}=a_{\alpha }^{i}h^{\left\vert \alpha \right\vert }$.
The interpolation polynomial $p\left[ y,\sigma \right] $ (\ref{solution_mat}) can also be expressed in the basis (\ref{scaled_basis}) as 
\begin{equation}
p\left[ y,\sigma \right] \left( \mathbf{x}\right) =\left[ \left( \frac{%
\mathbf{x}-\overline{\mathbf{x}}}{h}\right) ^{\alpha }\right] _{\left\vert
\alpha \right\vert \leq d}c_{h},  \label{interpolation_poly_scaled}
\end{equation}%
where $c_{h}=\left[ c_{\alpha ,h}\right] _{\left\vert \alpha \right\vert
\leq d}^{T}$, with $c_{\alpha ,h}=c_{\alpha }h^{\left\vert \alpha
\right\vert }$.
\end{remark}

In the following we denote by $V_{\overline{\mathbf{x}},h}(\sigma )$ the
Vandermonde matrix in the scaled basis (\ref{scaled_basis}) 
\begin{equation}
V_{\overline{\mathbf{x}},h}(\sigma )=\left[ \left( \frac{\mathbf{x}_{i}-%
\overline{\mathbf{x}}}{h}\right) ^{\alpha }\right] _{\substack{ i=1,\dots ,m 
\\ \left\vert \alpha \right\vert \leq d}}, \label{scaled_vandermonde_matrix}
\end{equation}
and by $B_{h}\left( \mathbf{\overline{\mathbf{x}}}\right)$
the ball of radius $h$ centered at $\mathbf{\overline{\mathbf{x}}}$.

\begin{proposition}
\label{theo:theoretical_bound1}Let $\overline{\mathbf{x}}\in \Omega$ and $%
f\in C^{d,1}\left( \Omega \right) $. Then for any $\mathbf{x}\in K=B_h(\overline{\mathbf{x}})\cap\Omega$
and for any $\nu \in 
\mathbb{N}
_{0}^{s}$ such that $\left\vert \nu \right\vert \leq d$, we have%
\begin{equation}
\left\vert \left( D^{\nu }f-D^{\nu }p\left[ y,\sigma \right] \right) \left( 
\mathbf{x}\right) \right\vert \leq k_{d-\left\vert \nu \right\vert
}\left\Vert \mathbf{x}-\overline{\mathbf{x}}\right\Vert _{2}^{d-\left\vert
\nu \right\vert +1}\left\Vert f\right\Vert _{d,1}^K+\left\vert D^{\nu }\left(
\sum\limits_{i=1}^{m}\ell _{i}\left( \mathbf{x}\right) R_{T}\left[ f,%
\overline{\mathbf{x}}\right] \left( \mathbf{x}_{i}\right) \right)
\right\vert ,  \label{Dnu_error_bound1}
\end{equation}%
where $k_{j}=\frac{s^{j}}{\left( j-1\right) !}$ for $j>0$, $k_{0}=1$. In
particular, for $\mathbf{x}=\overline{\mathbf{x}}$, we have 
\begin{equation*}
\left\vert \left( D^{\nu }f-D^{\nu }p\left[ y,\sigma \right] \right) \left( 
\overline{\mathbf{x}}\right) \right\vert \leq \left\vert D^{\nu }\left(
\sum\limits_{i=1}^{m}\ell _{i}\left( \mathbf{x}\right) R_{T}\left[ f,%
\overline{\mathbf{x}}\right] \left( \mathbf{x}_{i}\right) \right)
\right\vert .
\end{equation*}
\end{proposition}

\begin{proof}
Since%
\begin{equation}
p\left[ y,\sigma \right] \left( \mathbf{x}\right) =\sum\limits_{i=1}^{m}\ell
_{i}\left( \mathbf{x}\right) y_{i},  \label{inter_poly_lagrange}
\end{equation}%
by (\ref{lagrange_poly_scaled}), we have%
\begin{equation*}
f\left( \mathbf{x}\right) -p\left[ y,\sigma \right] \left( \mathbf{x}\right)
=f\left( \mathbf{x}\right) -\sum\limits_{i=1}^{m}\ell _{i}\left( \mathbf{x}%
\right) y_{i}.
\end{equation*}%
By representing $f\left( \mathbf{x}\right) $ and $f\left( \mathbf{x}%
_{i}\right)=y_{i} $ in truncated Taylor series of order $d$ centered at $%
\overline{\mathbf{x}}$ (\ref{Taylor_form}) with integral remainder (\ref%
{integral_remainder}), we obtain%
\begin{eqnarray}
f\left( \mathbf{x}\right) -p\left[ y,\sigma \right] \left( \mathbf{x}\right)
&=&\sum_{l=0}^{d}\frac{1}{l!}D_{\mathbf{x}-\overline{\mathbf{x}}}^{l}f\left( 
\overline{\mathbf{x}}\right) +R_{T}\left[ f,\overline{\mathbf{x}}\right]
\left( \mathbf{x}\right) -\sum\limits_{i=1}^{m}\ell _{i}\left( \mathbf{x}%
\right)  \label{error_of_interpolation} \\
&&\times \left( \sum_{l=0}^{d}\frac{1}{l!}D_{\mathbf{x}_{i}-\overline{%
\mathbf{x}}}^{l}f\left( \overline{\mathbf{x}}\right) +R_{T}\left[ f,%
\overline{\mathbf{x}}\right] \left( \mathbf{x}_{i}\right) \right) .  \notag
\end{eqnarray}%
On the other hand 
\begin{equation*}
\sum\limits_{i=1}^{m}\ell _{i}\left( \mathbf{x}\right) \left( \mathbf{x}_{i}-%
\overline{\mathbf{x}}\right) ^{\beta }=\left( \mathbf{x}-\overline{\mathbf{x}%
}\right) ^{\beta },\text{ }\left\vert \beta \right\vert \leq d,
\end{equation*}%
since interpolation of degree $d$ at the nodes in $\sigma$ reproduces
exactly polynomials of total degree less than or equal to $d$. Therefore by (%
\ref{directional_derivatives}) we have%
\begin{eqnarray}
\sum\limits_{i=1}^{m}\ell _{i}\left( \mathbf{x}\right) \sum\limits_{l=0}^{d}%
\frac{1}{l!}D_{\mathbf{x}_{i}-\overline{\mathbf{x}}}^{l}f\left( \overline{%
\mathbf{x}}\right) &=&\sum\limits_{i=1}^{m}\ell _{i}\left( \mathbf{x}\right)
\sum\limits_{l=0}^{d}\frac{1}{l!}\sum\limits_{\left\vert \beta \right\vert
=l}\frac{l!}{\beta !}D^{\beta }f\left( \overline{\mathbf{x}}\right) \left( 
\mathbf{x}_{i}\mathbf{-}\overline{\mathbf{x}}\right) ^{\beta }  \notag \\
&=&\sum\limits_{l=0}^{d}\frac{1}{l!}\sum\limits_{\left\vert \beta
\right\vert =l}\frac{l!}{\beta !}D^{\beta }f\left( \overline{\mathbf{x}}%
\right) \sum\limits_{i=1}^{m}\ell _{i}\left( \mathbf{x}\right) \left( 
\mathbf{x}_{i}\mathbf{-}\overline{\mathbf{x}}\right) ^{\beta }  \notag \\
&=&\sum\limits_{l=0}^{d}\frac{1}{l!}\sum\limits_{\left\vert \beta
\right\vert =l}\frac{l!}{\beta !}D^{\beta }f\left( \overline{\mathbf{x}}%
\right) \left( \mathbf{x}-\overline{\mathbf{x}}\right) ^{\beta }  \notag \\
&=&\sum\limits_{l=0}^{d}\frac{1}{l!}D_{\mathbf{x-}\overline{\mathbf{x}}%
}^{l}f\left( \overline{\mathbf{x}}\right) .  \label{reproduction_property}
\end{eqnarray}%
Consequently, by substituting (\ref{reproduction_property}) in (\ref%
{error_of_interpolation}), we get 
\begin{equation}
f\left( \mathbf{x}\right) -p\left[ y,\sigma \right] \left( \mathbf{x}\right)
=R_{T}\left[ f,\overline{\mathbf{x}}\right] \left( \mathbf{x}\right)
-\sum\limits_{i=1}^{m}\ell _{i}\left( \mathbf{x}\right) R_{T}\left[ f,%
\overline{\mathbf{x}}\right] \left( \mathbf{x}_{i}\right).
\label{difference_errors}
\end{equation}%
By applying the differentiation operator $D^{\nu }$ to the expression (\ref%
{difference_errors}) and by using the triangular inequality, we obtain%
\begin{equation*}
\left\vert \left( D^{\nu }f-D^{\nu }p\left[ y,\sigma \right] \right) \left( 
\mathbf{x}\right) \right\vert \leq \left\vert D^{\nu }R_{T}\left[ f,%
\overline{\mathbf{x}}\right] \left( \mathbf{x}\right) \right\vert
+\left\vert D^{\nu }\left( \sum\limits_{i=1}^{m}\ell _{i}\left( \mathbf{x}%
\right) R_{T}\left[ f,\overline{\mathbf{x}}\right] \left( \mathbf{x}%
_{i}\right) \right) \right\vert ,
\end{equation*}%
where \cite[Lemma 2.1]{farwig1986rate} 
\begin{equation}
\left\vert D^{\nu }R_{T}\left[ f,\overline{\mathbf{x}}\right] \left( \mathbf{%
x}\right) \right\vert \leq k_{d-\left\vert \nu \right\vert }\left\Vert 
\mathbf{x}-\overline{\mathbf{x}}\right\Vert _{2}^{d-\left\vert \nu
\right\vert +1}\left\Vert f\right\Vert _{d,1}^K.  \label{Dnu_remainder}
\end{equation}%
Consequently, 
\begin{equation*}
\left\vert \left( D^{\nu }f-D^{\nu }p\left[ y,\sigma \right] \right) \left( 
\mathbf{x}\right) \right\vert \leq k_{d-\left\vert \nu \right\vert
}\left\Vert \mathbf{x}-\overline{\mathbf{x}}\right\Vert _{2}^{d-\left\vert
\nu \right\vert +1}\left\Vert f\right\Vert _{d,1}^K+\left\vert D^{\nu }\left(
\sum\limits_{i=1}^{m}\ell _{i}\left( \mathbf{x}\right) R_{T}\left[ f,%
\overline{\mathbf{x}}\right] \left( \mathbf{x}_{i}\right) \right)
\right\vert .
\end{equation*}
\begin{flushright}
$\blacksquare$
\end{flushright}
\end{proof}

The estimates in Theorem \ref{theo:theoretical_bound1} can be written in
terms of the Lebesgue constant of interpolation at the node set $\sigma $,
defined by 
\begin{equation*}
\Lambda _{d}\left( \sigma \right) =\max\limits_{\mathbf{x}\in\Omega
}\sum\limits_{i=1}^{m}\left\vert \ell _{i}\left( \mathbf{x}\right)
\right\vert .
\end{equation*}

\begin{proposition}
\label{theo:theoretical_bound_Lebesgue}Let $f\in C^{d,1}\left( \Omega \right)
$ and $B_{h}\left( \mathbf{\overline{\mathbf{x}}}\right) \subset \Omega$.
Then for any $\mathbf{x}\in K=B_{h}\left( \mathbf{\overline{\mathbf{x}}}%
\right) $ and for any $\nu \in 
\mathbb{N}
_{0}^{s}$ such that $\left\vert \nu \right\vert \leq d$, we have 
\begin{equation}
\left\vert \left( D^{\nu }f-D^{\nu }p\left[ y,\sigma \right] \right) \left( 
\mathbf{x}\right) \right\vert \leq \left\Vert f\right\Vert _{d,1}^K\left(
k_{d-\left\vert \nu \right\vert }+k_{d}\,M_{d,\nu }\,\Lambda _{d}\left(
\sigma \right) \right) \,h^{d-\left\vert \nu \right\vert +1},
\label{bound_dnu_lebesgue}
\end{equation}%
where $k_{j}=\frac{s^{j}}{\left( j-1\right) !}$ for $j>0$, $k_{0}=1$, and 
\begin{equation}
M_{d,\nu }=\prod_{j=0}^{\left\vert \nu \right\vert -1}\left( d-j\right) ^{2}.
\end{equation}%
In particular, for $\mathbf{x}=\overline{\mathbf{x}}$, the following
inequality holds 
\begin{equation}
\left\vert \left( D^{\nu }f-D^{\nu }p\left[ y,\sigma \right] \right) \left( 
\overline{\mathbf{x}}\right) \right\vert \leq \left\Vert f\right\Vert
_{d,1}^K\,k_{d}\,M_{d,\nu }\,\Lambda _{d}(\sigma )\,h^{d-\left\vert \nu
\right\vert +1}.  \label{bound_dnu_Lebesgue_xbar}
\end{equation}
\end{proposition}

\begin{proof}
Since $\sum\limits_{i=1}^{m}R_{T}\left[ f,\overline{\mathbf{x}}%
\right] \left( \mathbf{x}_{i}\right) \ell _{i}\left( \mathbf{x}\right) $ is
a polynomial of total degree less than or equal to $d$, by repeatedly
applying the Markov inequality \cite{wilhelmsen1974markov} for a ball with
radius $h$ in the form 
\begin{equation*}
\left\Vert \frac{\partial q}{\partial \xi _{i}}\right\Vert _{\infty }\leq 
\frac{n^{2}}{h}\,\Vert q\Vert _{\infty },\;\;\forall q\in \Pi _{n}(\mathbb{R}%
^{s}),\;n=d, d-1,\dots ,d-\left\vert \nu \right\vert +1,
\end{equation*}%
and recalling that each partial derivative lowers the degree by one, we
easily obtain 
\begin{eqnarray*}
\left\vert D^{\nu }\left( \sum\limits_{i=1}^{m}R_{T}\left[ f,\overline{%
\mathbf{x}}\right] \left( \mathbf{x}_{i}\right) \ell _{i}\left( \mathbf{x}%
\right) \right) \right\vert &\leq &\frac{M_{d,\nu }}{h^{\left\vert \nu
\right\vert }}\left\Vert \sum\limits_{i=1}^{m}R_{T}\left[ f,\overline{%
\mathbf{x}}\right] \left( \mathbf{x}_{i}\right) \ell _{i}\right\Vert
_{\infty } \\
&\leq &\frac{M_{d,\nu }}{h^{\left\vert \nu \right\vert }}\max_{\mathbf{x}\in 
B_{h}\left( \mathbf{\overline{\mathbf{x}}}\right)
}\sum\limits_{i=1}^{m}\left\vert R_{T}\left[ f,\overline{\mathbf{x}}\right]
\left( \mathbf{x}_{i}\right) \right\vert \left\vert \ell _{i}\left( \mathbf{x%
}\right) \right\vert \\
&\leq &\frac{M_{d,\nu }}{h^{\left\vert \nu \right\vert }}\max_{i=1,\dots
,m}\left\vert R_{T}\left[ f,\overline{\mathbf{x}}\right] \left( \mathbf{x}%
_{i}\right) \right\vert \max_{\mathbf{x}\in B_{h}\left( \mathbf{\overline{%
\mathbf{x}}}\right) }\sum\limits_{i=1}^{m}\left\vert \ell _{i}\left( \mathbf{%
x}\right) \right\vert \\
&\leq &\frac{M_{d,\nu }}{h^{\left\vert \nu \right\vert }}\max_{i=1,\dots
,m}\left\vert R_{T}\left[ f,\overline{\mathbf{x}}\right] \left( \mathbf{x}%
_{i}\right) \right\vert \Lambda _{d}\left( \sigma \right) ,
\end{eqnarray*}%
where \cite[Lemma 2.1]{farwig1986rate}%
\begin{equation}
\begin{array}{cc}
\left\vert R_{T}\left[ f,\overline{\mathbf{x}}\right] \left( \mathbf{x}%
_{i}\right) \right\vert \leq k_{d}\left\Vert \mathbf{x}_{i}-\overline{%
\mathbf{x}}\right\Vert _{2}^{d+1}\left\Vert f\right\Vert _{d,1}^K, & \text{for
all }i=1,\dots ,m.
\end{array}
\label{bound_remainder_xi}
\end{equation}%
Consequently%
\begin{equation}
\left\vert D^{\nu }\left( \sum\limits_{i=1}^{m}\ell _{i}\left( \mathbf{x}%
\right) R_{T}\left[ f,\overline{\mathbf{x}}\right] \left( \mathbf{x}%
_{i}\right) \right) \right\vert \leq k_{d}\,M_{d,\nu }%
\,\left\Vert f\right\Vert _{d,1}^Kh^{d-\left\vert \nu \right\vert +1}\Lambda
_{d}\left( \sigma \right) .  \label{Dnu_sum_li_remainder}
\end{equation}%
Based on the inequality (\ref{Dnu_error_bound1}) in Theorem \ref%
{theo:theoretical_bound1}, we have%
\begin{equation}
\left\vert \left( D^{\nu }f-D^{\nu }p\left[ y,\sigma \right] \right) \left( 
\mathbf{x}\right) \right\vert \leq k_{d-\left\vert \nu \right\vert
}\left\Vert \mathbf{x}-\overline{\mathbf{x}}\right\Vert _{2}^{d-\left\vert
\nu \right\vert +1}\left\Vert f\right\Vert _{d,1}^K+k_{d}\,
M_{d,\nu }\left\Vert f\right\Vert _{d,1}^K h^{d-\left\vert \nu \right\vert
+1}\Lambda _{d}\left( \sigma \right) ,  \label{Dnu_error_bound2}
\end{equation}%
and since $\mathbf{x}\in B_{h}\left( \mathbf{\overline{\mathbf{x}}}\right) $%
, it follows that 
\begin{equation}
\left\vert \left( D^{\nu }f-D^{\nu }p\left[ y,\sigma \right] \right) \left( 
\mathbf{x}\right) \right\vert \leq \left\Vert f\right\Vert
_{d,1}^K h^{d-\left\vert \nu \right\vert +1}\left( k_{d-\left\vert \nu
\right\vert }+k_{d}\,M_{d,\nu }\,\Lambda _{d}\left( \sigma
\right) \right) ,  \notag
\end{equation}%
while (\ref{bound_dnu_Lebesgue_xbar}) follows easily by evaluating (\ref%
{Dnu_error_bound2}) at $\overline{\mathbf{x}}$. \hfill $\blacksquare$

\end{proof}

\begin{proposition}
\label{theo:theoretical_bound_3}Let $\overline{\mathbf{x}}\in \Omega$ and $f\in C^{d,1}\left( \Omega \right) $. Then for any $\mathbf{%
x}\in  K=B_h(\overline{\mathbf{x}})\cap\Omega$ and for
any $\nu \in 
\mathbb{N}
_{0}^{s}$ such that $\left\vert \nu \right\vert \leq d$, we have%
\begin{equation}
\left\vert \left( D^{\nu }f-D^{\nu }p\left[ y,\sigma \right] \right) \left( 
\mathbf{x}\right) \right\vert \leq \left\Vert f\right\Vert
_{d,1}^K h^{d+1}\left( k_{d-\left\vert \nu \right\vert }h^{-\left\vert \nu
\right\vert }+k_{d}\sum\limits_{i=1}^{m}\left\vert \sum_{\substack{ %
\left\vert \alpha \right\vert \leq d  \\ \alpha \geq \nu }}a_{\alpha
,h}^{i}h^{-\left\vert \alpha \right\vert }\frac{\alpha !}{\left( \alpha -\nu
\right) !}\left( \mathbf{x}-\overline{\mathbf{x}}\right) ^{\alpha -\nu
}\right\vert \right) ,  \label{bound_der_nu_x3}
\end{equation}%
where $k_{j}=\frac{s^{j}}{\left( j-1\right) !}$ for $j>0$, $k_{0}=1$. The
inequalities between multi-indices are interpreted componentwise, that is $%
\alpha \leq \beta $ if and only if $\alpha _{i}\leq \beta _{i}$, $i=1,\dots
,s$. In particular, for $\mathbf{x}=\overline{\mathbf{x}}$, the following
inequality holds 
\begin{equation}
\left\vert \left( D^{\nu }f-D^{\nu }p\left[ y,\sigma \right] \right) \left( 
\overline{\mathbf{x}}\right) \right\vert \leq \nu !k_{d}\left\Vert
f\right\Vert _{d,1}^K h^{d-\left\vert \nu \right\vert
+1}\sum\limits_{i=1}^{m}\left\vert a_{\nu ,h}^{i}\right\vert .
\label{bound_der_nu_xbar3}
\end{equation}
\end{proposition}

\begin{proof}
By using the expression\ of the fundamental Lagrange polynomial $\ell
_{i}\left( \mathbf{x}\right) $ in the scaled basis (\ref%
{lagrange_poly_scaled}) and by applying the differentiation operator $D^{\nu
}$ to the expression (\ref{difference_errors}), we obtain%
\begin{equation}
\left( D^{\nu }f-D^{\nu }p\left[ y,\sigma \right] \right) \left( \mathbf{x}%
\right) =D^{\nu }R_{T}\left[ f,\overline{\mathbf{x}}\right] \left( \mathbf{x}%
\right) -\sum\limits_{i=1}^{m}\left( \sum_{\left\vert \alpha \right\vert
\leq d}a_{\alpha ,h}^{i}h^{-\left\vert \alpha \right\vert }D^{\nu }\left( 
\mathbf{x}-\overline{\mathbf{x}}\right) ^{\alpha }\right) R_{T}\left[ f,%
\overline{\mathbf{x}}\right] \left( \mathbf{x}_{i}\right) ,
\label{derivative_difference}
\end{equation}%
where%
\begin{equation}
D^{\nu }\left( \mathbf{x}-\overline{\mathbf{x}}\right) ^{\alpha }=\left\{ 
\begin{array}{cc}
\frac{\alpha !}{\left( \alpha -\nu \right) !}\left( \mathbf{x}-\overline{%
\mathbf{x}}\right) ^{\alpha -\nu } & \text{if }\nu \leq \alpha \text{,} \\ 
0 & \text{otherwise.}%
\end{array}%
\right.  \label{derivative_x-xbar}
\end{equation}%
By taking the modulus of both sides of (\ref{derivative_difference}) and by
using the triangular inequality, we get%
\begin{equation*}
\left\vert \left( D^{\nu }f-D^{\nu }p\left[ y,\sigma \right] \right) \left( 
\mathbf{x}\right) \right\vert \leq \left\vert D^{\nu }R_{T}\left[ f,%
\overline{\mathbf{x}}\right] \left( \mathbf{x}\right) \right\vert
+\sum\limits_{i=1}^{m}\left\vert \sum_{\left\vert \alpha \right\vert \leq
d}a_{\alpha ,h}^{i}h^{-\left\vert \alpha \right\vert }D^{\nu }\left( \mathbf{%
x}-\overline{\mathbf{x}}\right) ^{\alpha }\right\vert \left\vert R_{T}\left[
f,\overline{\mathbf{x}}\right] \left( \mathbf{x}_{i}\right) \right\vert ,
\end{equation*}%
Therefore, (\ref{Dnu_remainder}), (\ref{derivative_x-xbar}) and (\ref%
{bound_remainder_xi}) imply%
\begin{eqnarray}
\left\vert \left( D^{\nu }f-D^{\nu }p\left[ y,\sigma \right] \right) \left( 
\mathbf{x}\right) \right\vert &\leq &k_{d-\left\vert \nu \right\vert
}\left\Vert f\right\Vert _{d,1}^K\left\Vert \mathbf{x}-\overline{\mathbf{x}}%
\right\Vert _{2}^{d-\left\vert \nu \right\vert +1}  \label{Dnu_x_error} \\
&&+k_{d}\left\Vert f\right\Vert _{d,1}^K\sum\limits_{i=1}^{m}\left\vert \sum 
_{\substack{ \left\vert \alpha \right\vert \leq d  \\ \alpha \geq \nu }}%
a_{\alpha ,h}^{i}h^{-\left\vert \alpha \right\vert }\frac{\alpha !}{\left(
\alpha -\nu \right) !}\left( \mathbf{x}-\overline{\mathbf{x}}\right)
^{\alpha -\nu }\right\vert \left\Vert \mathbf{x}_{i}-\overline{\mathbf{x}}%
\right\Vert _{2}^{d+1}.  \notag
\end{eqnarray}%
Since $\mathbf{x}\in B_{h}\left( \mathbf{\overline{\mathbf{x}}}\right) $, we
get%
\begin{equation*}
\left\vert \left( D^{\nu }f-D^{\nu }p\left[ y,\sigma \right] \right) \left( 
\mathbf{x}\right) \right\vert \leq \left\Vert f\right\Vert
_{d,1}^K h^{d+1}\left( k_{d-\left\vert \nu \right\vert }h^{-\left\vert \nu
\right\vert }+k_{d}\sum\limits_{i=1}^{m}\left\vert \sum_{\substack{ %
\left\vert \alpha \right\vert \leq d  \\ \alpha \geq \nu }}a_{\alpha
,h}^{i}h^{-\left\vert \alpha \right\vert }\frac{\alpha !}{\left( \alpha -\nu
\right) !}\left( \mathbf{x}-\overline{\mathbf{x}}\right) ^{\alpha -\nu
}\right\vert \right) .
\end{equation*}%
Evaluating (\ref{Dnu_x_error}) at $\overline{\mathbf{x}}$ yields to (\ref%
{bound_der_nu_xbar3}).  \hfill $\blacksquare$

\end{proof}

In line with \cite{dell2020numerical} we can write the bounds in Proposition %
\ref{theo:theoretical_bound_3} in terms of the $1$-norm condition number, say $\cond_h\left(\sigma\right)$, of
the Vandermonde matrix $V_{\overline{\mathbf{x}},h}\left( \sigma \right) $.

\begin{corollary}
\label{cor:theoretical_bound_kappa}Let  $\overline{\mathbf{x}}\in \Omega$ and $f\in C^{d,1}\left( \Omega \right) $.
Then for any $\mathbf{x}\in K=B_{h}\left( \mathbf{\overline{\mathbf{x}}}%
\right)\cap \Omega$ and for any $\nu \in 
\mathbb{N}
_{0}^{s}$ such that $\left\vert \nu \right\vert \leq d$, we have%
\begin{equation}
\left\vert \left( D^{\nu }f-D^{\nu }p\left[ y,\sigma \right] \right) \left( 
\mathbf{x}\right) \right\vert \leq \left\Vert f\right\Vert
_{d,1}^K h^{d-\left\vert \nu \right\vert +1}\left( k_{d-\left\vert \nu
\right\vert }+k_{d}\max_{\substack{ \left\vert \alpha \right\vert \leq d  \\ %
\alpha \geq \nu }}\left( \frac{\alpha !}{\left( \alpha -\nu \right) !}%
\right) \cond_{h}\left( \sigma \right) \right) ,  \label{bound_der_nu_kappa}
\end{equation}%
where $k_{j}=\frac{s^{j}}{\left( j-1\right) !}$ for $j>0$, $k_{0}=1$. The
inequalities between multi-indices are interpreted componentwise, that is $%
\alpha \leq \beta $ if and only if $\alpha _{i}\leq \beta _{i}$, $i=1,\dots
,s$. In particular, for $\mathbf{x}=\overline{\mathbf{x}}$, the following
inequality holds 
\begin{equation}
\left\vert \left( D^{\nu }f-D^{\nu }p\left[ y,\sigma \right] \right) \left( 
\overline{\mathbf{x}}\right) \right\vert \leq \nu !k_{d}\left\Vert
f\right\Vert _{d,1}^K h^{d-\left\vert \nu \right\vert +1}\cond_{h}\left( \sigma
\right) .  \label{bound_der_nu_xbar_kappa}
\end{equation}
\end{corollary}

\begin{proof}
By applying the triangular inequality to (\ref{Dnu_x_error}), we get%
\begin{eqnarray*}
\left\vert \left( D^{\nu }f-D^{\nu }p\left[ y,\sigma \right] \right) \left( 
\mathbf{x}\right) \right\vert &\leq &k_{d-\left\vert \nu \right\vert
}\left\Vert f\right\Vert _{d,1}^K\left\Vert \mathbf{x}-\overline{\mathbf{x}}%
\right\Vert _{2}^{d-\left\vert \nu \right\vert +1} \\
&&+k_{d}\left\Vert f\right\Vert _{d,1}^K h^{d+1}\sum\limits_{i=1}^{m}\sum 
_{\substack{ \left\vert \alpha \right\vert \leq d  \\ \alpha \geq \nu }}%
\left\vert a_{\alpha ,h}^{i}\right\vert h^{-\left\vert \alpha \right\vert }%
\frac{\alpha !}{\left( \alpha -\nu \right) !}\left\vert \left( \mathbf{x}-%
\overline{\mathbf{x}}\right) ^{\alpha -\nu }\right\vert \\
&\leq &k_{d-\left\vert \nu \right\vert }\left\Vert f\right\Vert
_{d,1}^K \left\Vert \mathbf{x}-\overline{\mathbf{x}}\right\Vert
_{2}^{d-\left\vert \nu \right\vert +1} \\
&&+k_{d}\left\Vert f\right\Vert _{d,1}^K \max_{\substack{ \left\vert \alpha
\right\vert \leq d  \\ \alpha \geq \nu }}\left( h^{-\left\vert \alpha
\right\vert }\frac{\alpha !}{\left( \alpha -\nu \right) !}\left\vert \left( 
\mathbf{x}-\overline{\mathbf{x}}\right) ^{\alpha -\nu }\right\vert \right)
h^{d+1}\sum\limits_{i=1}^{m}\sum_{\substack{ \left\vert \alpha \right\vert
\leq d  \\ \alpha \geq \nu }}\left\vert a_{\alpha ,h}^{i}\right\vert \\
&\leq &k_{d-\left\vert \nu \right\vert }\left\Vert f\right\Vert
_{d,1}^K \left\Vert \mathbf{x}-\overline{\mathbf{x}}\right\Vert
_{2}^{d-\left\vert \nu \right\vert +1} \\
&&+k_{d}\left\Vert f\right\Vert _{d,1}^K \max_{\substack{ \left\vert \alpha
\right\vert \leq d  \\ \alpha \geq \nu }}\left( h^{-\left\vert \alpha
\right\vert }\frac{\alpha !}{\left( \alpha -\nu \right) !}\left\vert \left( 
\mathbf{x}-\overline{\mathbf{x}}\right) ^{\alpha -\nu }\right\vert \right)
h^{d+1}\sum\limits_{i=1}^{m}\left\Vert a_{h}^{i}\right\Vert _{1},
\end{eqnarray*}%
where $a_{h}^{i}=\left[ a_{\alpha ,h}^{i}\right] _{\left\vert \alpha
\right\vert \leq d}^{T}$. Based on the expression of $\ell _{i}\left( 
\mathbf{x}\right) $ (\ref{lagrange_poly_scaled}), we have \cite%
{dell2020numerical} 
\begin{equation*}
V_{\overline{\mathbf{x}},h}^{-1}(\sigma )=\left[ 
\begin{tabular}{l|l|l|l|l}
$a_{h}^{1}$ & $a_{h}^{2}$ & $\cdots $ & $a_{h}^{m-1}$ & $a_{h}^{m}$%
\end{tabular}%
\right] ,
\end{equation*}%
and then 
\begin{equation*}
\max\limits_{i=1,\dots m}\left\Vert a_{h}^{i}\right\Vert _{1}=\left\Vert V_{%
\overline{\mathbf{x}},h}^{-1}\left( \sigma \right) \right\Vert _{1}.
\end{equation*}%
Moreover, since $\left\Vert V_{\overline{\mathbf{x}},h}\left( \sigma \right)
\right\Vert _{1}=m$ then $\sum\limits_{i=1}^{m}\left\Vert
a_{h}^{i}\right\Vert _{1}\leq \cond_{h}\left( \sigma \right) $ and therefore%
\begin{eqnarray}
\left\vert \left( D^{\nu }f-D^{\nu }p\left[ y,\sigma \right] \right) \left( 
\mathbf{x}\right) \right\vert &\leq &k_{d-\left\vert \nu \right\vert
}\left\Vert f\right\Vert _{d,1}^K \left\Vert \mathbf{x}-\overline{\mathbf{x}}%
\right\Vert _{2}^{d-\left\vert \nu \right\vert +1}  \label{bound} \\
&&+k_{d}\left\Vert f\right\Vert _{d,1}^K \max_{\substack{ \left\vert \alpha
\right\vert \leq d  \\ \alpha \geq \nu }}\left( h^{-\left\vert \alpha
\right\vert }\frac{\alpha !}{\left( \alpha -\nu \right) !}\left\vert \left( 
\mathbf{x}-\overline{\mathbf{x}}\right) ^{\alpha -\nu }\right\vert \right)
h^{d+1}\cond_{h}\left( \sigma \right)  \notag \\
&\leq &k_{d-\left\vert \nu \right\vert }\left\Vert f\right\Vert
_{d,1}^K \left\Vert \mathbf{x}-\overline{\mathbf{x}}\right\Vert
_{2}^{d-\left\vert \nu \right\vert +1}  \notag \\
&&+k_{d}\left\Vert f\right\Vert _{d,1}^K \max_{\substack{ \left\vert \alpha
\right\vert \leq d  \\ \alpha \geq \nu }}\left( h^{-\left\vert \alpha
\right\vert }\frac{\alpha !}{\left( \alpha -\nu \right) !}\left\Vert \mathbf{%
x}-\overline{\mathbf{x}}\right\Vert _{2}^{\left\vert \alpha \right\vert
-\left\vert \nu \right\vert }\right) h^{d+1}\cond_{h}\left( \sigma \right) .
\notag
\end{eqnarray}%
Being $\mathbf{x}\in B_{h}\left( \mathbf{\overline{\mathbf{x}}}\right) $, it
follows that 
\begin{equation*}
\left\vert \left( D^{\nu }f-D^{\nu }p\left[ y,\sigma \right] \right) \left( 
\mathbf{x}\right) \right\vert \leq \left\Vert f\right\Vert
_{d,1}^K h^{d-\left\vert \nu \right\vert +1}\left( k_{d-\left\vert \nu
\right\vert }+k_{d}\max_{\substack{ \left\vert \alpha \right\vert \leq d  \\ %
\alpha \geq \nu }}\left( \frac{\alpha !}{\left( \alpha -\nu \right) !}%
\right) \cond_{h}\left( \sigma \right) \right) .
\end{equation*}%
By evaluating (\ref{bound}) at\textbf{\ }$\overline{\mathbf{x}}$, we obtain (%
\ref{bound_der_nu_xbar_kappa}). \hfill $\blacksquare$
\end{proof}

The results of Table \ref{tab:estimates} in Section \ref{section4} show that
the bounds  \eqref{bound_dnu_Lebesgue_xbar} and %
\eqref{bound_der_nu_xbar_kappa} are much larger than %
\eqref{bound_der_nu_xbar3}, which is only based on the ``active
coefficients'' in the differentiation process. Therefore, in the analysis of
the sensitivity to the perturbation of the function values, we use only the
``active coefficients''.

\begin{proposition}
Let $\widetilde{y}=y+\Delta y$, where $\Delta y=\left[ \Delta y_{i}\right]
_{i=1,\dots ,m}$ corresponds to the perturbation on the function values $y=%
\left[ y_{i}\right] _{i=1,\dots ,m}$. Then for any $\mathbf{x}\in 
K=B_{h}\left( \mathbf{\overline{\mathbf{x}}}\right)\cap \Omega$, for any $\nu
\in 
\mathbb{N}
_{0}^{s}$ such that $\left\vert \nu \right\vert \leq d$ and for any $\Delta
y $ with $\left\vert \Delta y\right\vert \leq \varepsilon $, we have%
\begin{equation*}
\left\vert \left( D^{\nu }p\left[ y,\sigma \right] -D^{\nu }p\left[ 
\widetilde{y},\sigma \right] \right) \left( \mathbf{x}\right) \right\vert
\leq \varepsilon \sum\limits_{i=1}^{m}\left\vert \sum_{\substack{ \left\vert
\alpha \right\vert \leq d  \\ \alpha \geq \nu }}a_{\alpha
,h}^{i}h^{-\left\vert \alpha \right\vert }\frac{\alpha !}{\left( \alpha -\nu
\right) !}\left( \mathbf{x}-\overline{\mathbf{x}}\right) ^{\alpha -\nu
}\right\vert ,
\end{equation*}%
where $k_{j}=\frac{s^{j}}{\left( j-1\right) !}$ for $j>0$, $k_{0}=1$. The
inequalities between multi-indices are interpreted componentwise, that is $%
\alpha \leq \beta $ if and only if $\alpha _{i}\leq \beta _{i}$, $i=1,\dots
,s$. In particular, for $\mathbf{x}=\overline{\mathbf{x}}$, the following
inequality holds%
\begin{equation}  \label{stabest}
\left\vert \left( D^{\nu }p\left[ y,\sigma \right] -D^{\nu }p\left[ 
\widetilde{y},\sigma \right] \right) \left( \overline{\mathbf{x}}\right)
\right\vert \leq \varepsilon\,\sum_{i=1}^m{|D^{\nu }\ell_i(%
\overline{\mathbf{x}})|}=\,\varepsilon\, \nu !h^{-\left\vert \nu
\right\vert }\sum\limits_{i=1}^{m}\left\vert a_{\nu ,h}^{i}\right\vert .
\end{equation}
\end{proposition}

\begin{proof}
Denoting by $\widetilde{y}=y+\Delta y$, where $\Delta y=\left[ \Delta y_{i}%
\right] _{i=1,\dots ,m}$ corresponds to the perturbation on the function
values $y=\left[ y_{i}\right] _{i=1,\dots ,m}$, by (\ref{inter_poly_lagrange}%
) we get 
\begin{eqnarray*}
p\left[ \widetilde{y},\sigma \right] \left( \mathbf{x}\right)
&=&\sum\limits_{i=1}^{m}\ell _{i}\left( \mathbf{x}\right) \widetilde{y_{i}}
\\
&=&\sum\limits_{i=1}^{m}\ell _{i}\left( \mathbf{x}\right) \left(
y_{i}+\Delta y_{i}\right) \\
&=&p\left[ y,\sigma \right] \left( \mathbf{x}\right)
+\sum\limits_{i=1}^{m}\ell _{i}\left( \mathbf{x}\right) \Delta y_{i}.
\end{eqnarray*}%
Then for any $\Delta y$ such that $\left\vert \Delta y\right\vert \leq
\varepsilon $, 
\begin{eqnarray*}
\left\vert \left( D^{\nu }p\left[ y,\sigma \right] -D^{\nu }p\left[ 
\widetilde{y},\sigma \right] \right) \left( \mathbf{x}\right) \right\vert
&=&\left\vert D^{\nu }\left( \sum\limits_{i=1}^{m}\ell _{i}\left( \mathbf{x}%
\right) \Delta y_{i}\right) \right\vert \\
&=&\left\vert \sum\limits_{i=1}^{m}D^{\nu } \ell _{i}\left( \mathbf{x}%
\right) \Delta y_{i}\right\vert \\
&\leq &\varepsilon \sum\limits_{i=1}^{m}\left\vert D^{\nu }\ell _{i}\left( 
\mathbf{x}\right) \right\vert \\
&\leq &\varepsilon \sum\limits_{i=1}^{m}\left\vert D^{\nu }\left(
\sum_{\left\vert \alpha \right\vert \leq d}a_{\alpha ,h}^{i}\left( \frac{%
\mathbf{x}-\overline{\mathbf{x}}}{h}\right) ^{\alpha }\right) \right\vert .
\end{eqnarray*}%
Since, using (\ref{derivative_x-xbar}), we have 
\begin{equation*}
D^{\nu }\left( \sum_{\left\vert \alpha \right\vert \leq d}a_{\alpha
,h}^{i}\left( \frac{\mathbf{x}-\overline{\mathbf{x}}}{h}\right) ^{\alpha
}\right) =\sum_{\substack{ \left\vert \alpha \right\vert \leq d  \\ \alpha
\geq \nu }}a_{\alpha ,h}^{i}h^{-\left\vert \alpha \right\vert }\frac{\alpha !%
}{\left( \alpha -\nu \right) !}\left( \mathbf{x}-\overline{\mathbf{x}}%
\right) ^{\alpha -\nu },
\end{equation*}%
then 
\begin{equation}
\left\vert \left( D^{\nu }p\left[ y,\sigma \right] -D^{\nu }p\left[ 
\widetilde{y},\sigma \right] \right) \left( \mathbf{x}\right) \right\vert
\leq \varepsilon \sum\limits_{i=1}^{m}\left\vert \sum_{\substack{ \left\vert
\alpha \right\vert \leq d  \\ \alpha \geq \nu }}a_{\alpha
,h}^{i}h^{-\left\vert \alpha \right\vert }\frac{\alpha !}{\left( \alpha -\nu
\right) !}\left( \mathbf{x}-\overline{\mathbf{x}}\right) ^{\alpha -\nu
}\right\vert .  \label{bound_diff_perturb}
\end{equation}%
For $\mathbf{x}=\overline{\mathbf{x}}$, it follows that%
\begin{equation*}
\left\vert \left( D^{\nu }p\left[ y,\sigma \right] -D^{\nu }p\left[ 
\widetilde{y},\sigma \right] \right) \left( \overline{\mathbf{x}}\right)
\right\vert \leq \varepsilon \nu !h^{-\left\vert \nu \right\vert
}\sum\limits_{i=1}^{m}\left\vert a_{\nu ,h}^{i}\right\vert .
\end{equation*}
\begin{flushright}
$\blacksquare$
\end{flushright}
\end{proof}

\begin{remark}
It is worth observing that the quantity 
\begin{equation}  \label{stabconst}
\sum_{i=1}^m{|D^{\nu }\ell_i(\overline{\mathbf{x}})|}= \nu !h^{-\left\vert
\nu \right\vert }\sum\limits_{i=1}^{m}\left\vert a_{\nu ,h}^{i}\right\vert
\end{equation}
is the ``stability constant'' of pointwise differentiation via local
polynomial interpolation, namely the value at the center of the ``stability
function'' for the ball, that is 
\begin{equation} \label{stability_constant}
\sum_{i=1}^m{|D^{\nu }\ell_i(\mathbf{x})|} = \sum\limits_{i=1}^{m}\left\vert
\sum_{\substack{ \left\vert \alpha \right\vert \leq d  \\ \alpha \geq \nu }}%
a_{\alpha ,h}^{i}h^{-\left\vert \alpha \right\vert }\frac{\alpha !}{\left(
\alpha -\nu \right) !}\left( \mathbf{x}-\overline{\mathbf{x}}\right)
^{\alpha -\nu }\right\vert\;.
\end{equation}
Notice also that in view of (\ref{bound_der_nu_xbar3}) and (\ref{stabconst})
the overall numerical differentiation error, in the presence of
perturbations on the the function values of size not exceeding $\varepsilon$%
, can be estimated as 
\begin{equation}  \label{overallest }
\left\vert \left( D^{\nu }f-D^{\nu }p\left[ \widetilde{y},\sigma \right]%
\right) \left( \overline{\mathbf{x}}\right) \right\vert \leq
\left(k_d\|f\|_{d,1}^K h^{d+1}+\varepsilon\right)\,\nu !h^{-\left\vert \nu
\right\vert }\sum\limits_{i=1}^{m}\left\vert a_{\nu ,h}^{i}\right\vert \;.
\end{equation}
For the purpose of illustration, in Table \ref{tab:estimates_pertub} and Figure \ref{fig:stability_constant_plot} of
Section \ref{section4}  we show the magnitude of the stability constant %
\eqref{stabconst} relative to some numerical tests.
\end{remark}

\begin{remark}
The previous results are useful to estimate the error of approximation in
several processes of scattered data interpolation which use polynomials as
local interpolants, like for example triangular Shepard \cite%
{dell2016approximation,cavoretto2019fast}, hexagonal Shepard \cite%
{dell2020hexagonal} and tetrahedral Shepard methods \cite%
{cavoretto2020efficient}. They are also crucial to realize extensions of
those methods to higher dimensions \cite{dell2019rate}. 
\end{remark}

\section{\label{section4}Numerical experiments}

In this section we provide some numerical tests to support the above
theoretical results in approximating function, gradient and second order
derivative values. We fix $s=2$, $\Omega=[0,1]^2$ and we take different
positions for the point $\overline{\mathbf{x}}$ in $\Omega$: at the center,
and near/on a side and a corner of the square. We use different
distributions of scattered points in $\Omega$, namely Halton points \cite{kocis1997computational} and
uniform random points. We focus on the scattered points in the ball $B_r(%
\overline{\mathbf{x}})$ centered at $\overline{\mathbf{x}}$ for different
radii, from which we extract an interpolation subset $\sigma$ of $m=\binom{%
d+2}{2}$ Discrete Leja Points computed through the algorithm proposed in 
\cite{bos2010computing} (see Figure \ref{fig:Halton_balls}).

The reason for adopting Discrete Leja Points is twofold. We recall that they
are extracted from a finite set of points (in this case the scattered points
in the ball) by LU factorization with row pivoting of the corresponding
rectangular Vandermonde matrix. Indeed, Gaussian elimination with row
pivoting performs a sort of greedy optimization of the Vandermonde
determinant, by searching iteratively the new row (that is selecting the new
interpolation point) in such a way that the modulus of the augmented
determinant is maximized. In addition, if the polynomial basis is
lexicographically ordered, the Discrete Leja Points form a sequence, that is
the first ${\binom{k+s }{s}}$ are the Discrete Leja Points for interpolation
of degree $k$ in $s$ variables, $1\leq k\leq d$; see \cite{bos2010computing}
for a comprehensive discussion.

Then, on one hand Discrete Leja Points provide, with a low computational
cost, a unisolvent interpolation set, since a nonzero Vandermonde
determinant is automatically seeked. On the other hand, since they are
computed by a greedy maximization, one can expect, as a qualitative
guideline, that the elements of the corresponding inverse Vandermonde matrix
(that are cofactors divided by the Vandermonde determinant), and thus also
the relevant sum in the error bound \eqref{bound_der_nu_xbar3} as well as
the condition number, are not allowed to increase rapidly. These results are
in line with those shown in \cite[Table 1]{dell2020numerical}. In addition,
using Discrete Leja Points has also the effect of trying to minimize the
sup-norm of the fundamental Lagrange polynomials $\ell_i$ (which, as it is
well-known, can be written as ratio of determinants, cf. \cite%
{bos2010computing}) and thus the Lebesgue constant, which is relevant to
estimate \eqref{bound_dnu_Lebesgue_xbar}. Nevertheless, it is clear from
Table 1 that the bounds involving the Lebesgue constant and the condition
number are much larger than \eqref{bound_der_nu_xbar3} which rests only on
the ``active coefficients'' in the differentiation process. Further
numerical experiments show that, while decreasing $r$, for each value of $%
\left \vert\nu\right \vert$, the first and third rows in Table \ref{tab:estimates} remain of the same
order of magnitude thanks to the scaling of the basis, for the feasible degrees (since unisolvence of interpolation of degree $d$ is possible until there are enough scattered points in
the ball).

\begin{table}[tbp]
\begin{center}
{\scriptsize \ 
\begin{tabular}{c@{\hspace{.4cm}}c|@{\hspace{.4cm}}c|@{\hspace{.4cm}}c|@{\hspace{.4cm}}c|@{\hspace{.4cm}}c|@{\hspace{.4cm}}c|@{\hspace{.4cm}}c}
&  &  & $d=5$ & $d=10$ & $d=15$ & $d=20$ & $d=25$ \\ \hline
&  & $\nu!\sum\limits_{i=1}^{m}\left\vert a_{\nu ,h}^{i}\right\vert$ & 2.31
& 2.43 & 6.69 & 2.41e+1 & 3.51e+1 \\ 
&  &  &  &  &  &  &  \\ 
& $\left \vert \nu\right \vert =0$ & $M_{d,\nu}\Lambda _{d}\left(
\sigma\right)$ & 6.22 & 6.91 & 2.50e+1 & 6.07e+1 & 1.21e+3 \\ 
&  &  &  &  &  &  &  \\ 
&  & $\nu! \cond _{h}\left( \sigma \right) $ & 1.96e+3 & 1.25e+6 & 8.89e+8 & 
3.38e+11 & 2.05e+14 \\ \hline
&  & $\nu!\sum\limits_{i=1}^{m}\left\vert a_{\nu ,h}^{i}\right\vert$ & 
1.32e+1 & 3.63e+1 & 2.27e+2 & 4.53e+2 & 3.87e+2 \\ 
&  &  &  &  &  &  &  \\ 
& $\left \vert \nu\right \vert =1$ & $M_{d,\nu}\Lambda _{d}\left(
\sigma\right)$ & 1.55e+2 & 6.91e+2 & 5.63e+3 & 2.43e+4 & 7.57e+5 \\ 
&  &  &  &  &  &  &  \\ 
&  & $\nu! \cond _{h}\left( \sigma \right) $ & 1.96e+3 & 1.25e+6 & 8.89e+8 & 
3.38e+11 & 2.05e+14 \\ \hline
&  & $\nu!\sum\limits_{i=1}^{m}\left\vert a_{\nu ,h}^{i}\right\vert$ & 
2.48e+1 & 3.53e+2 & 8.25e+2 & 4.55e+3 & 7.62e+3 \\ 
&  &  &  &  &  &  &  \\ 
& $\left \vert \nu\right \vert =2$ & $M_{d,\nu}\Lambda _{d}\left(
\sigma\right)$ & 2.49e+3 & 5.60e+4 & 1.10e+6 & 8.76e+6 & 4.36e+8 \\ 
&  &  &  &  &  &  &  \\ 
&  & $\nu! \cond _{h}\left( \sigma \right) $ & 3.27e+3 & 2.08e+6 & 1.48e+9 & 
5.63e+11 & 3.42e+14 \\ 
&  &  &  &  &  &  & 
\end{tabular}%
}
\end{center}
\caption{Numerical comparison among the mean value of the uncommon terms of
the estimates (\protect\ref{bound_der_nu_xbar3}), (\protect\ref%
{bound_dnu_Lebesgue_xbar}) and (\protect\ref{bound_der_nu_xbar_kappa}) with $%
\left\vert \protect\nu \right\vert \leq 2$, for interpolation at a sequence
of degrees $d$ on Discrete Leja Points extracted from the subset of $1000$
Halton points in $[0,1]^2$ contained in the ball of radius $r=1/2$ centered
at $\overline{\mathbf{x}}=(0.5,0.5)$.}
\label{tab:estimates}
\end{table}

For simplicity, from now on we set 
\begin{equation*}
p:=p[y,\sigma ]
\end{equation*}%
and, to measure the error of approximation, we compute the relative errors 
\begin{equation}
fe=\frac{\left\vert f(\overline{\mathbf{x}})-p(\overline{\mathbf{x}}%
)\right\vert }{\left\vert f(\overline{\mathbf{x}})\right\vert },
\label{function_error}
\end{equation}%
\begin{equation}
ge=\frac{\left\Vert \nabla{f}(\overline{\mathbf{x}}) - \nabla{p}(\overline{%
\mathbf{x}}) \right\Vert _{2}}{\left\Vert \nabla{f}(\overline{\mathbf{x}})
\right\Vert _{2}},  \label{gradient_error}
\end{equation}%
and%
\begin{equation}
sde=\frac{\left\Vert \left( f_{xx}(\overline{\mathbf{x}}),f_{xy}(\overline{%
\mathbf{x}}),f_{yy}(\overline{\mathbf{x}})\right) -\left( p_{xx}(\overline{%
\mathbf{x}}),p_{xy}(\overline{\mathbf{x}}),p_{yy}(\overline{\mathbf{x}}%
)\right) \right\Vert _{2}}{\left\Vert \left( f_{xx}(\overline{\mathbf{x}}%
),f_{xy}(\overline{\mathbf{x}}),f_{yy}(\overline{\mathbf{x}})\right)
\right\Vert _{2}},  \label{second_order_error}
\end{equation}%
using the following bivariate test functions 
\begin{equation*}
\begin{array}{ll}
f_{1}\left( x,y\right) & =0.75\exp \Bigl(-\dfrac{(9x-2)^{2}+(9y-2)^{2}}{4}%
\Bigr)+0.50\exp \Bigl(-\dfrac{\left( 9x-7\right) ^{2}+\left( 9y-3\right) ^{2}%
}{4}\Bigr) \\ 
& +0.75\exp \Bigl(-\dfrac{(9x+1)^{2}}{49}-\dfrac{(9y+1)^{2}}{10}\Bigr)%
-0.20\exp \Bigl(-(9x-4)^{2}-(9y-7)^{2}\Bigr), \\ 
&  \\ 
f_{2}\left( x,y\right) & =e^{x+y}, \\ 
&  \\ 
f_{3}\left( x,y\right) & =2\cos (10x)\sin (10y)+\sin (10xy),%
\end{array}%
\end{equation*}%
where $f_{1}$ is the well known Franke's function and $f_{3}$ is an
oscillating function (see Figure \ref{fig:cosine_function}) both in Renka's
test set \cite{renka1999algorithm}, whereas $f_{2}$ is obtained by a
superposition of the univariate exponential with an inner product and then
is constant on the parallel hyperplanes $x+y=q$, $q\in \mathbb{R}$ (ridge
function). For each test function we approximate $D^{\nu }f\left( \overline{%
\mathbf{x}}\right) $ by 
\begin{equation}
D^{\nu }f\left( \overline{\mathbf{x}}\right) \approx \nu !\dfrac{c_{\nu ,h}}{%
h^{\left\vert \nu \right\vert }},\qquad \left\vert \nu \right\vert \leq 2,
\label{Dnu2}
\end{equation}%
where $c_{\nu ,h}$, with $h\leq r$ defined in \eqref{h_scaling}, are the
coefficients of the interpolating polynomial (\ref{interpolation_poly_scaled}%
) at the point $\overline{\mathbf{x}}$. Interpolation is made at Discrete
Leja Points in $B_r(\overline{\mathbf{x}})$ for $r=\frac{1}{2},\frac{3}{8},%
\frac{1}{4},\frac{1}{8}$ at a sequence of degrees $d$. We stress that for a
fixed radius $r$, unisolvence of interpolation is possible only for a
finite number of degrees, that is until there are enough scattered points in
the ball.

In the first experiment we start from $1000$, $2000$ and $4000$ Halton and
uniform random points and we set $\overline{\mathbf{x}}=\left(0.5,0.5\right)$
(see Figure \ref{fig:Halton_balls}). For the test function $f_{1}$, the
numerical results are displayed in Figures \ref{fig:Franke_Halton}-\ref%
{fig:Franke_Random}.

\begin{figure}[t]
{\small \centering 
\parbox{.32\linewidth}{\centering
    \includegraphics[width=1.1\linewidth]{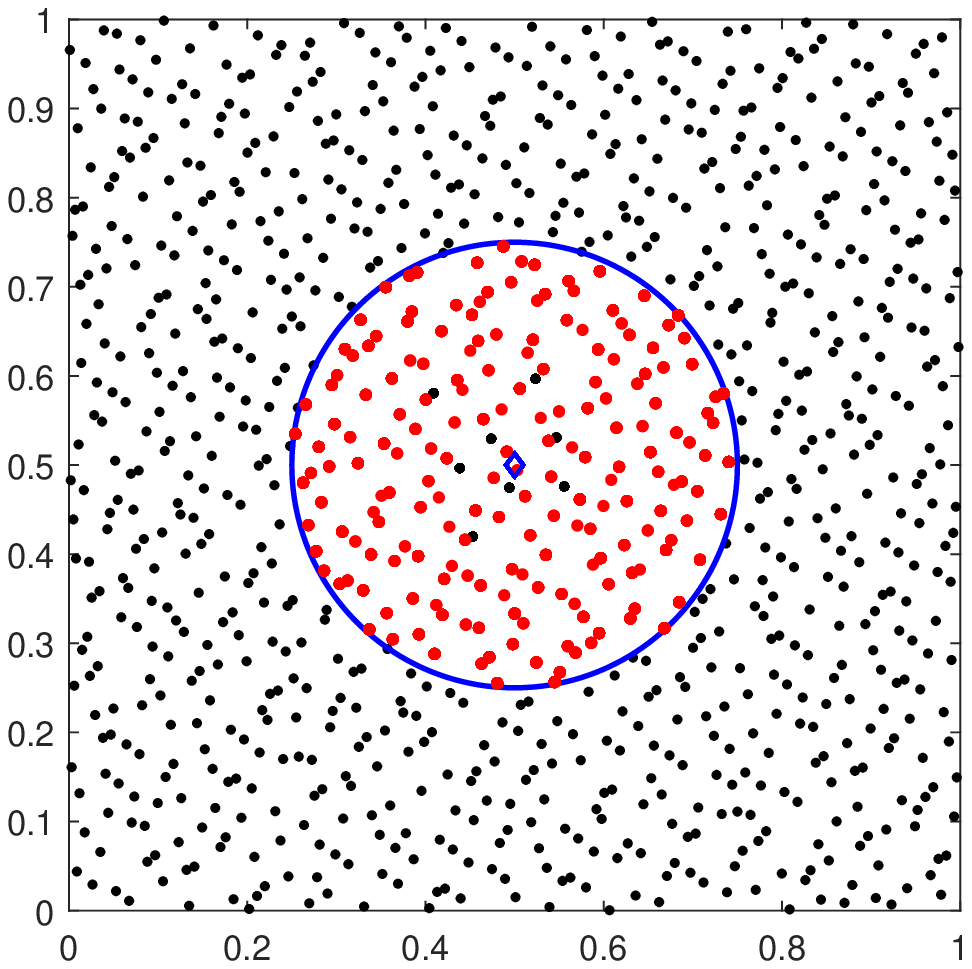}
} 
\parbox{.32\linewidth}{\centering
    \includegraphics[width=1.1\linewidth]{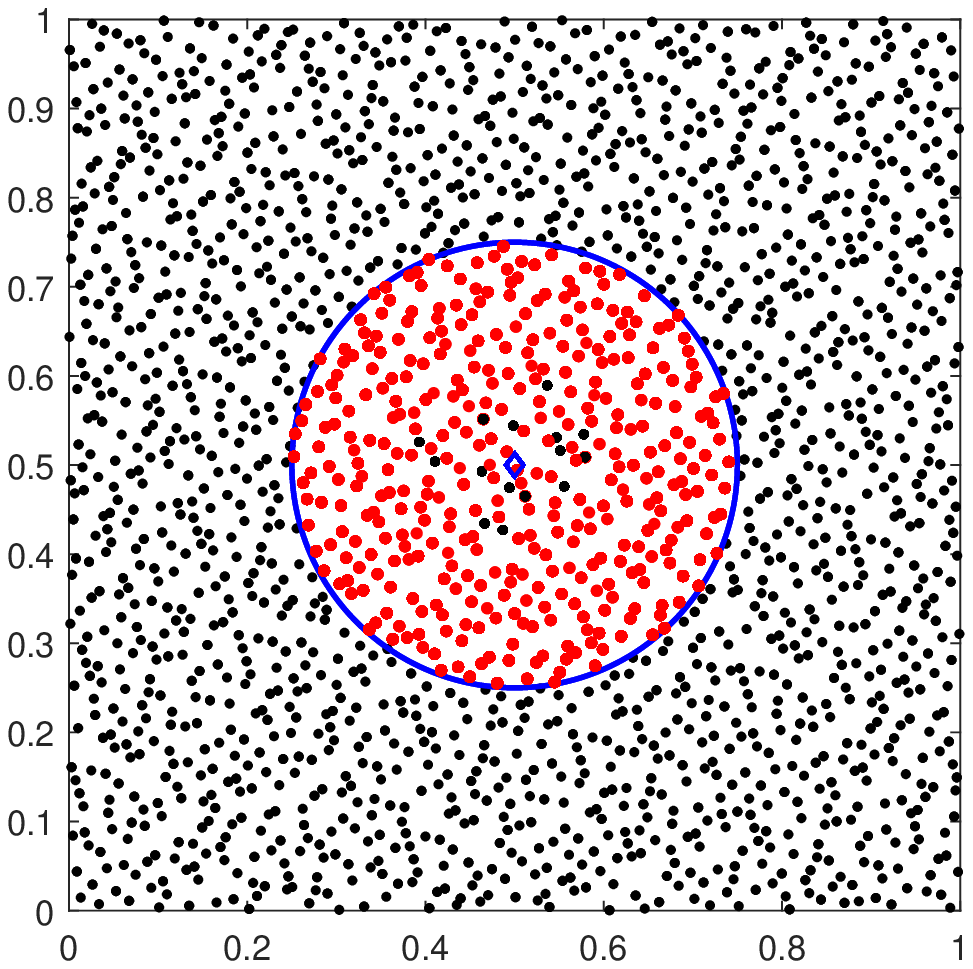}
} 
\parbox{.32\linewidth}{\centering
    \includegraphics[width=1.1\linewidth]{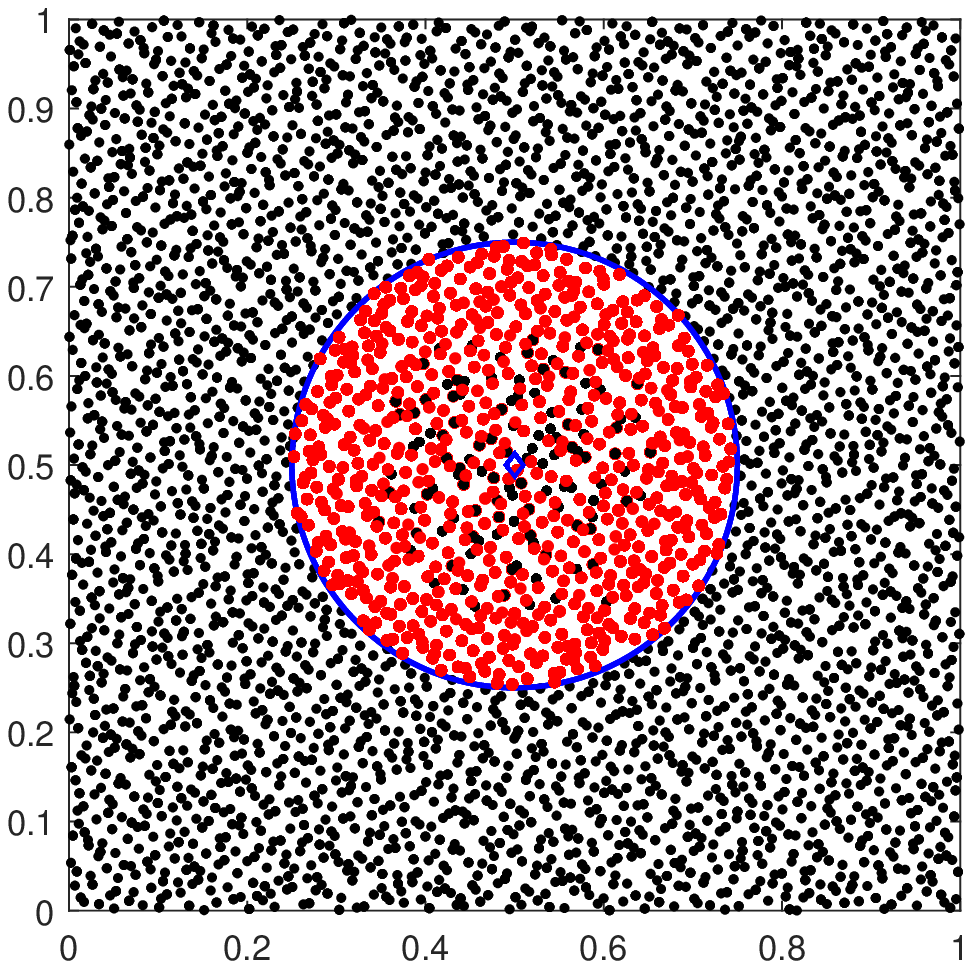}
}}
\caption{The interpolation points (in red) in the ball of radius $r=1/4$
centered at $(0.5,0.5)$ selected from $1000$ (left), $2000$ (center) and $4000$
(right) Halton points at degrees $d=17,25,35$, respectively, corresponding
to minimal errors in Figure \protect\ref{fig:Franke_Halton}.}
\label{fig:Halton_balls}
\end{figure}

\begin{figure}[t]
{\small \centering 
\parbox{.32\linewidth}{\centering
    \includegraphics[width=1.1\linewidth]{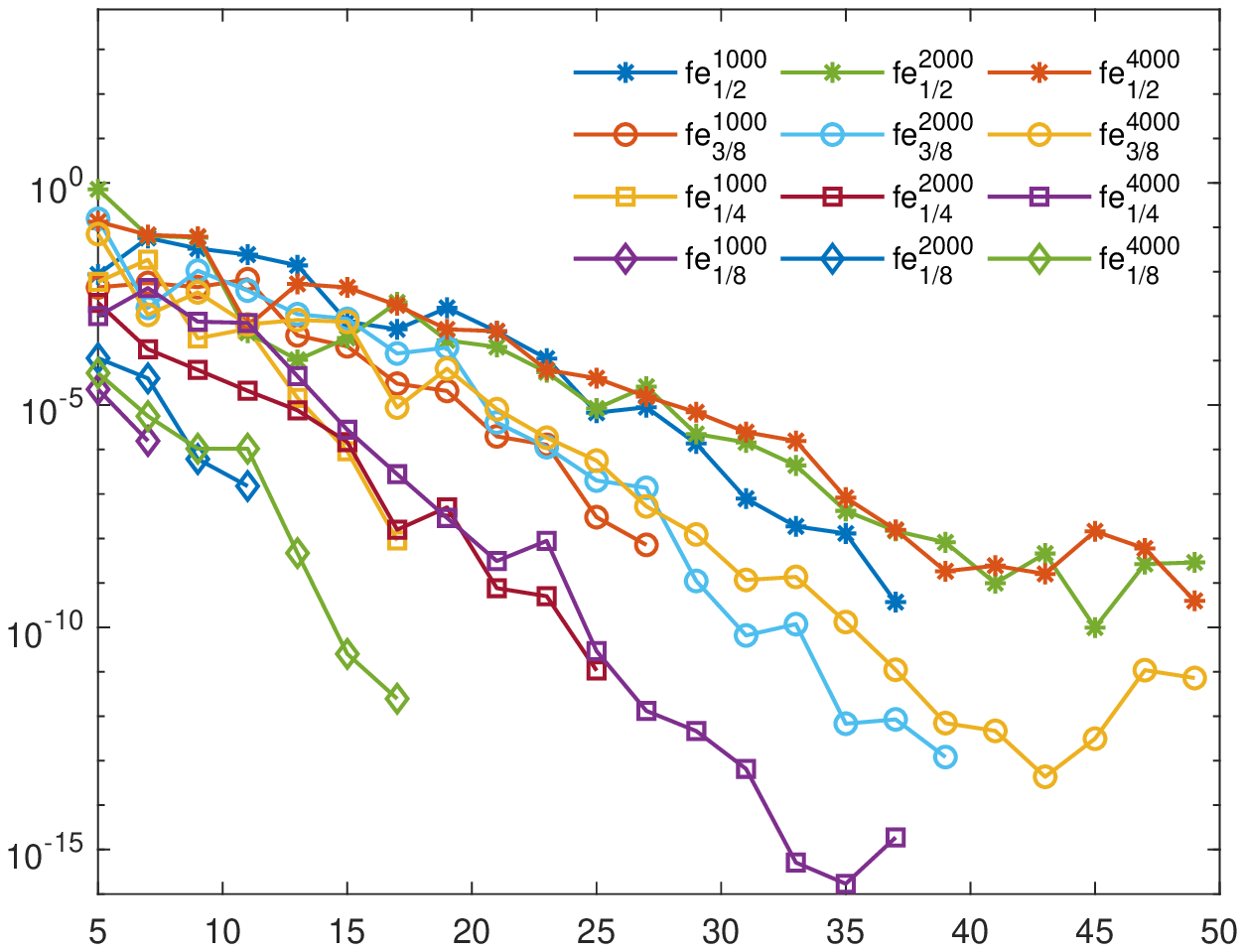}
} 
\parbox{.32\linewidth}{\centering
    \includegraphics[width=1.1\linewidth]{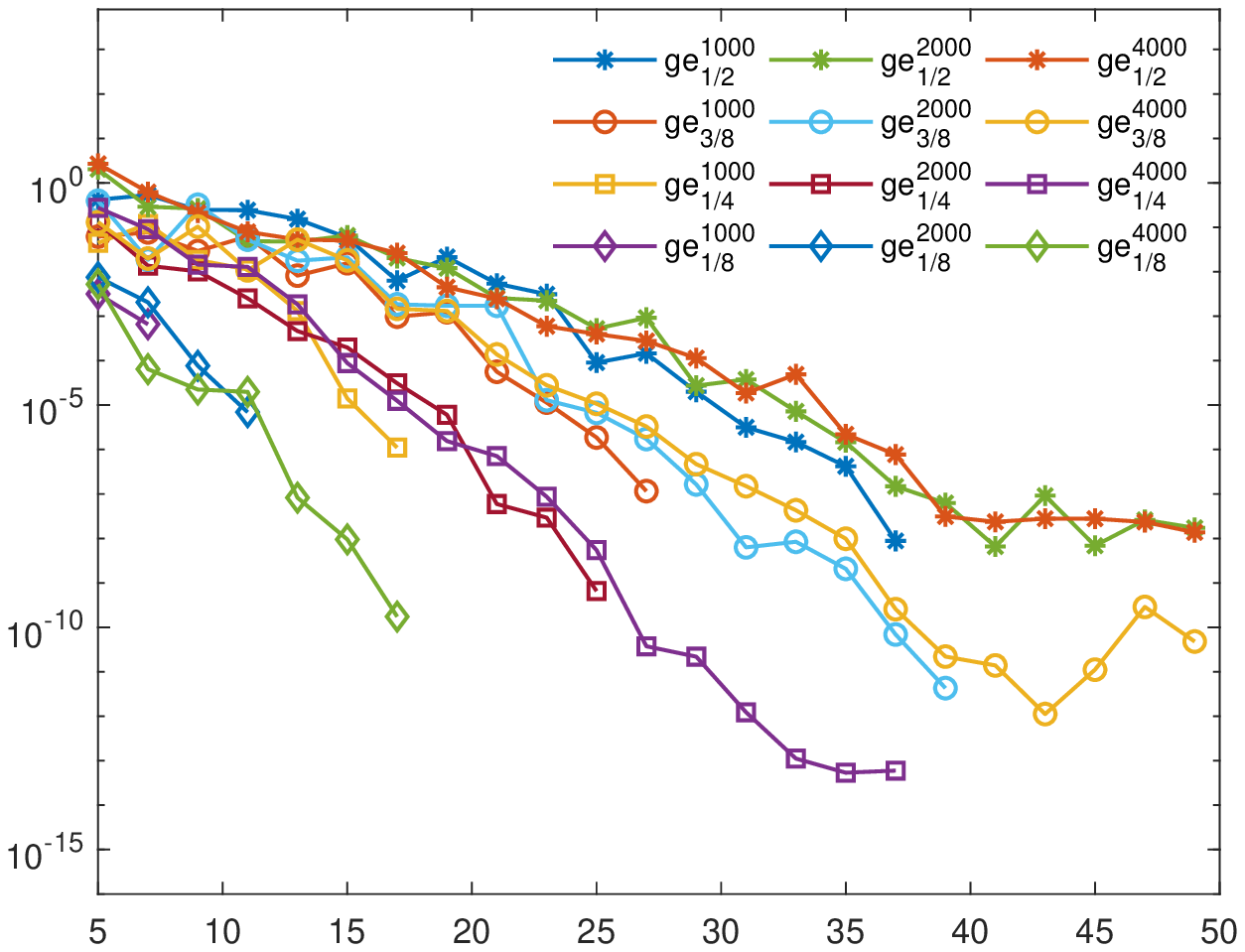}
} 
\parbox{.32\linewidth}{\centering
    \includegraphics[width=1.1\linewidth]{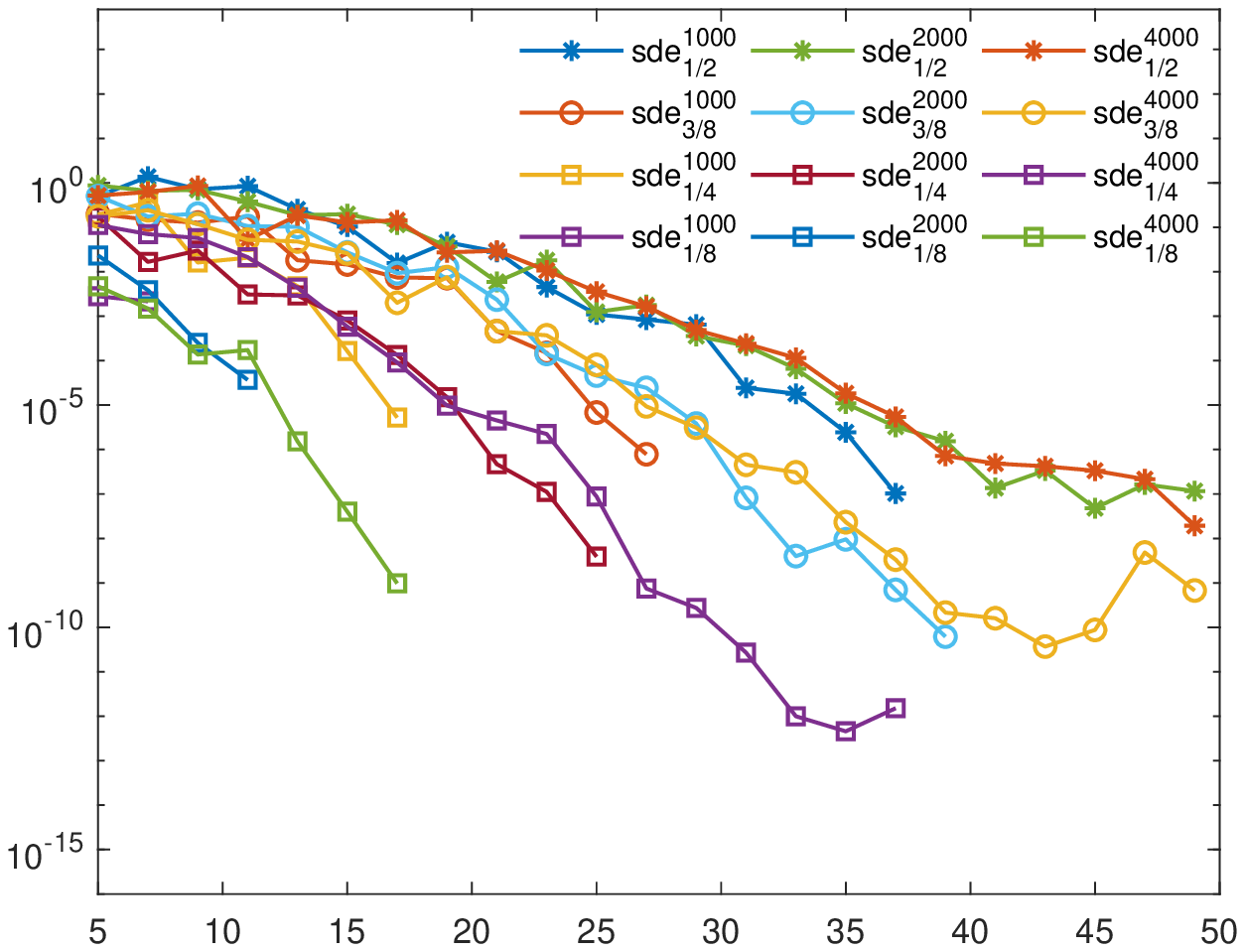}
}}
\caption{Relative errors \eqref{function_error}-\eqref{second_order_error}
for function $f_1$ by using the subsets of $1000$, $2000$ and $4000$ Halton
points intersecting $B_r\left(\overline{\mathbf{x}}\right)$, where $%
\overline{\mathbf{x}}=\left(0.5,0.5\right)$ and $r=\frac{1}{2},\frac{3}{8},%
\frac{1}{4},\frac{1}{8}$ on a sequence of degrees. Note that shorter
sequences (missing marks) are due to a lack of points for interpolation of
higher degree.}
\label{fig:Franke_Halton}
\end{figure}

\begin{figure}[t]
{\small \centering 
\parbox{.32\linewidth}{\centering
    \includegraphics[width=1.1\linewidth]{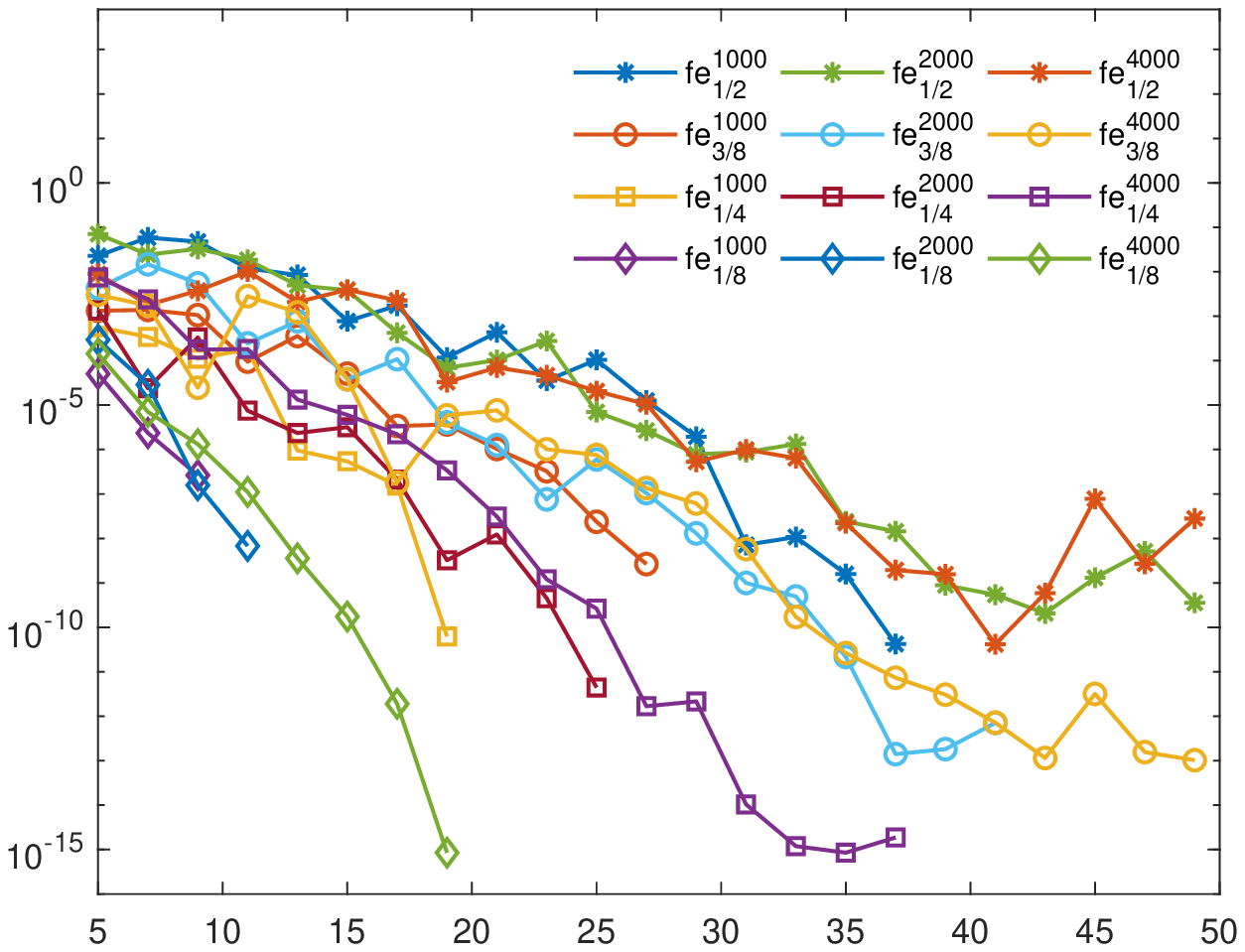}
} 
\parbox{.32\linewidth}{\centering
    \includegraphics[width=1.1\linewidth]{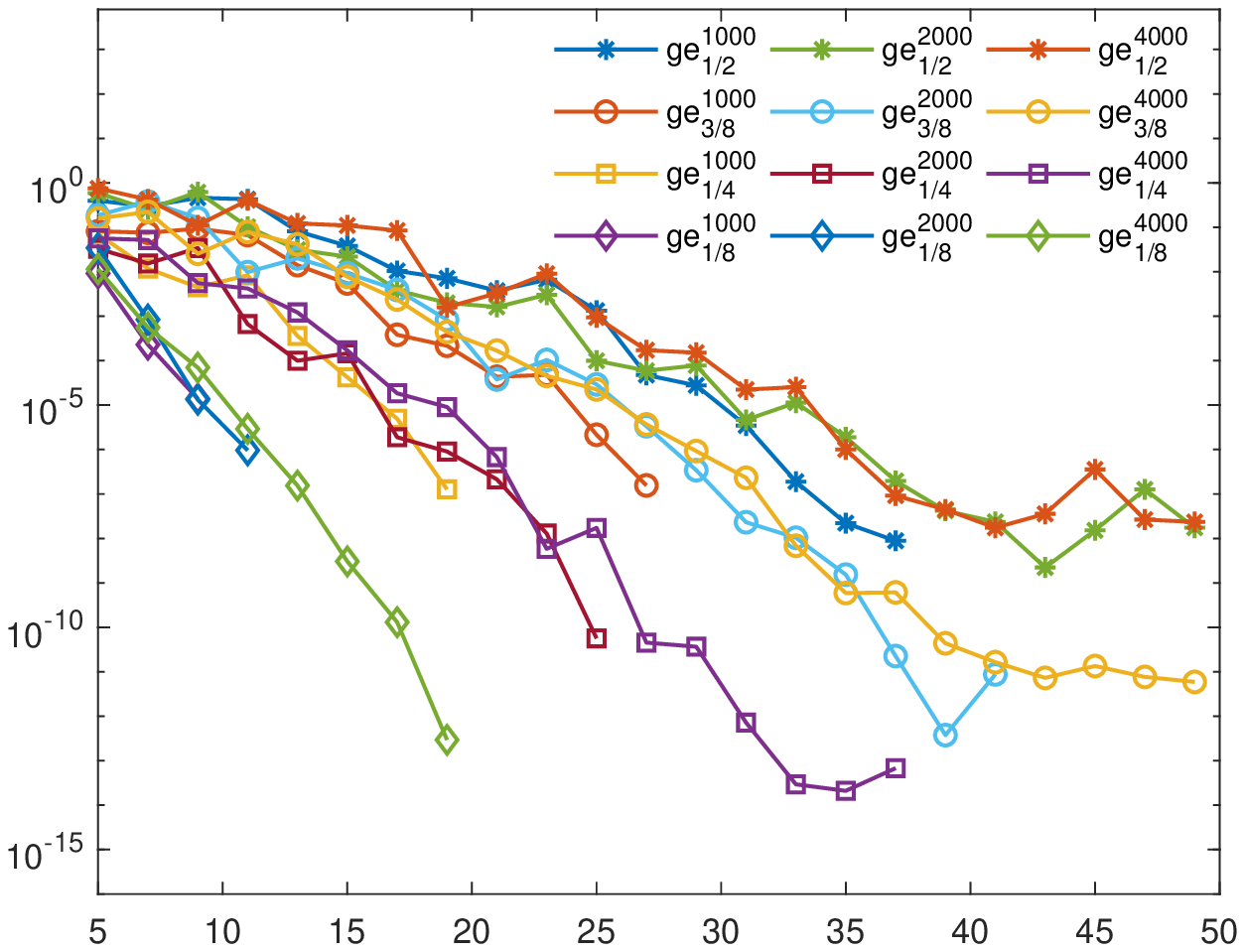}
} 
\parbox{.32\linewidth}{\centering
    \includegraphics[width=1.1\linewidth]{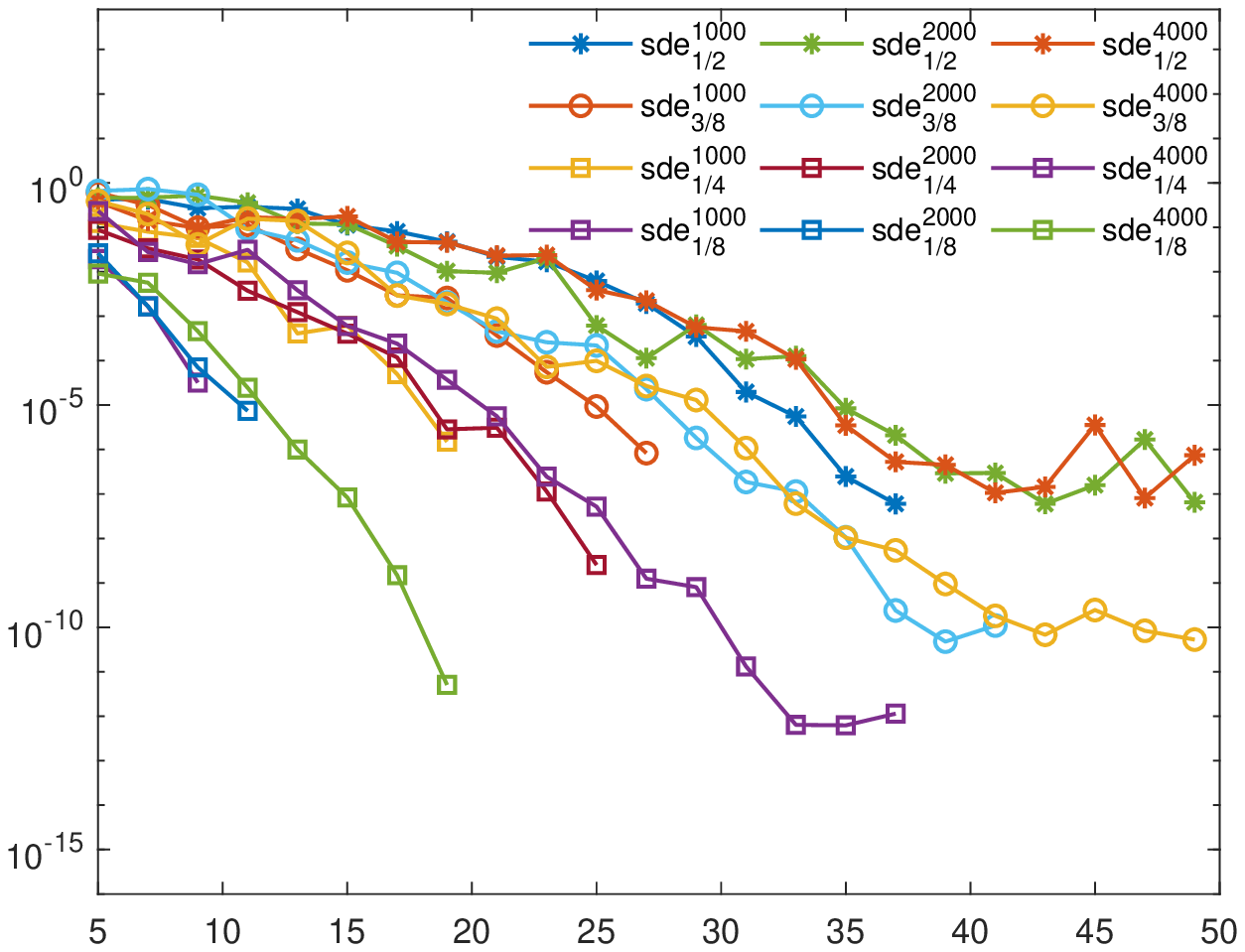}
}}
\caption{As in Figure \protect\ref{fig:Franke_Halton} starting from uniform
random points.}
\label{fig:Franke_Random}
\end{figure}

In the second experiment, we start from $2000$ Halton points for the test
function $f_2$, again with $\overline{\mathbf{x}}=\left(0.5,0.5\right)$. The
numerical results are displayed in Figure \ref{fig:Exponential_Halton}.

\begin{figure}[t]
{\small \centering 
\parbox{.32\linewidth}{\centering
    \includegraphics[width=1.1\linewidth]{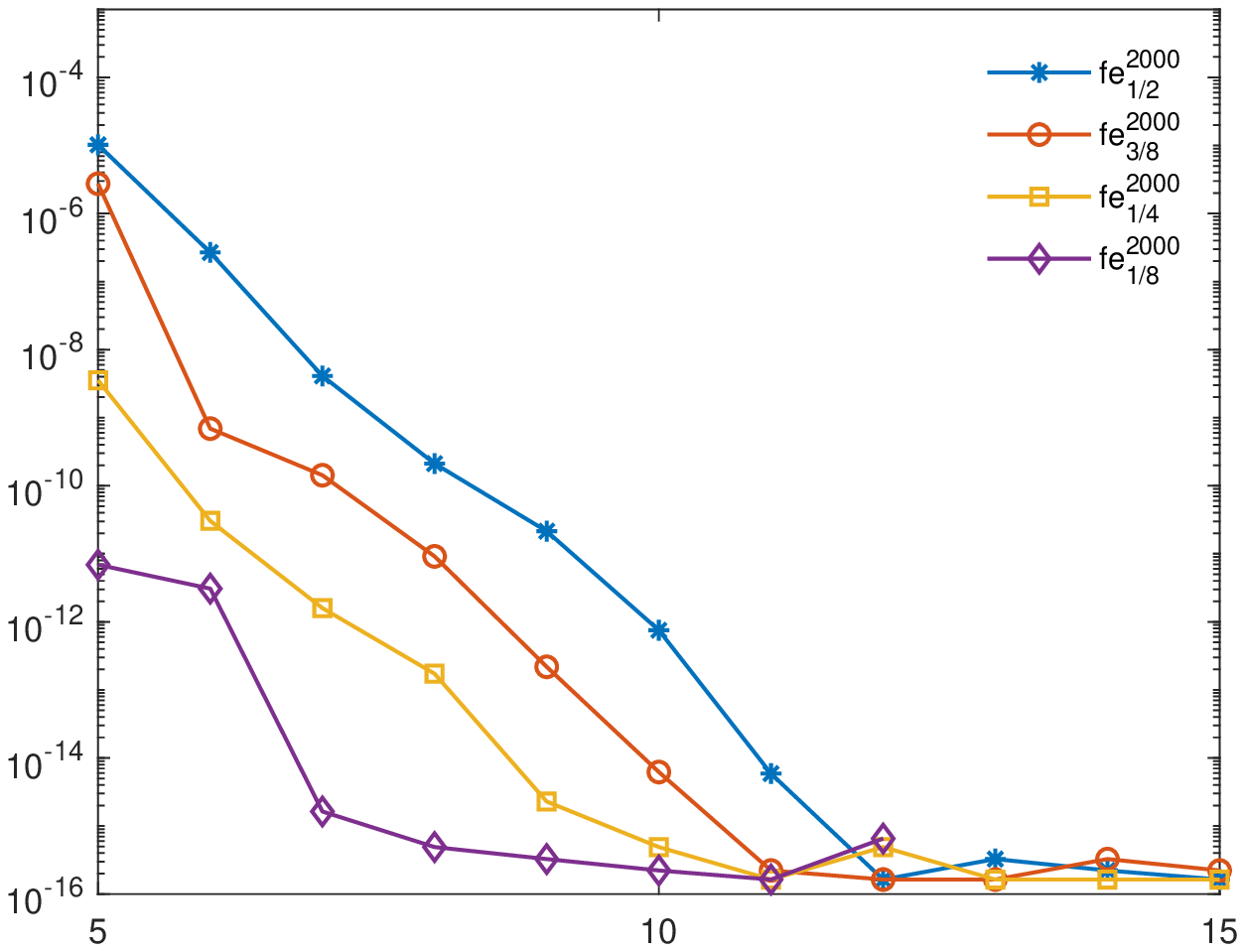}
} 
\parbox{.32\linewidth}{\centering
    \includegraphics[width=1.1\linewidth]{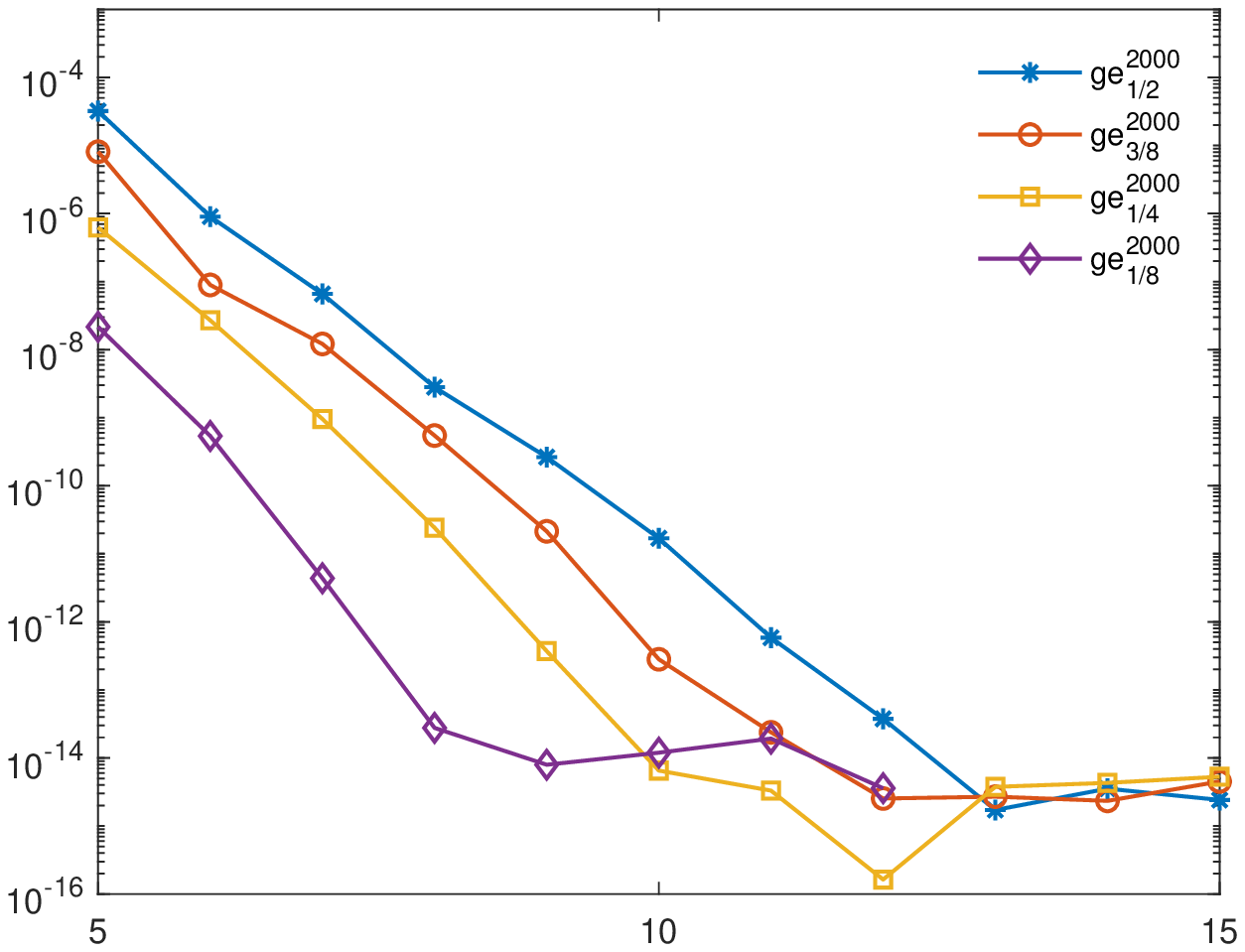}
} 
\parbox{.32\linewidth}{\centering
    \includegraphics[width=1.1\linewidth]{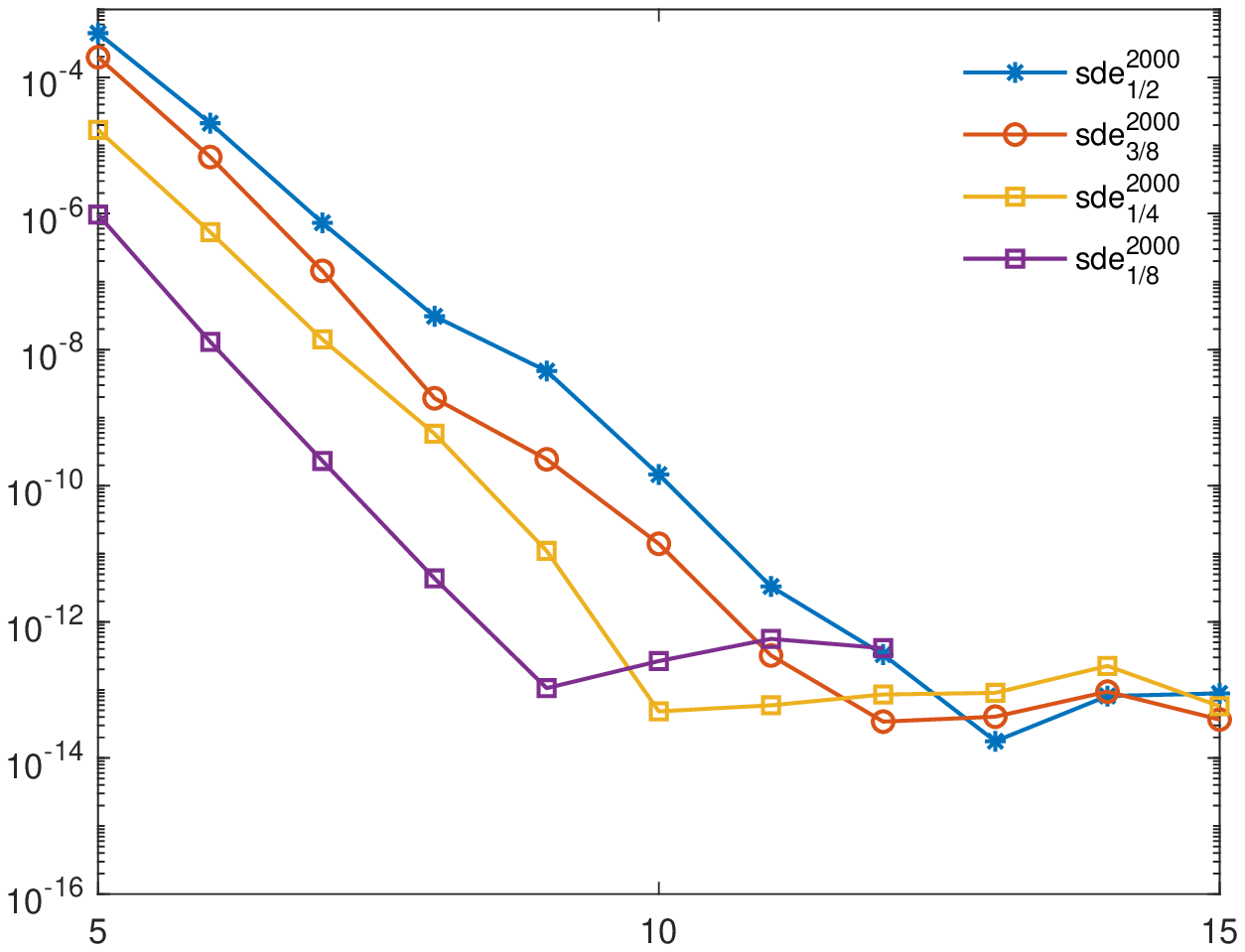}
}}
\caption{Relative errors \eqref{function_error}-\eqref{second_order_error}
for function $f_2$ using $2000$ Halton points with $\overline{\mathbf{x}}%
=\left(0.5,0.5\right)$.}
\label{fig:Exponential_Halton}
\end{figure}

In the third experiment, for the test function $f_3$, we start
from $4000$ Halton points choosing $\overline{\mathbf{x}}$ at the center,
then close to the right side and finally close to the north-east corner (see
Figures \ref{fig:points_in_the_ball}-\ref{fig:cosine_function}). The
numerical results are displayed in Figures \ref{fig:Cosine_Halton_center}, \ref{fig:Cosine_Halton_side} and
\ref{fig:Cosine_Halton_vertex}. We repeat the same experiments choosing $%
\overline{\mathbf{x}}$ on the right side and at the north-east corner and we
report the results in Figures \ref{fig:Cosine_4000_Halton_exact_side} and \ref%
{fig:Cosine_4000_Halton_exact_vertex}.

\begin{figure}[t]
{\small \centering 
\parbox{.32\linewidth}{\centering
    \includegraphics[width=1.2\linewidth]{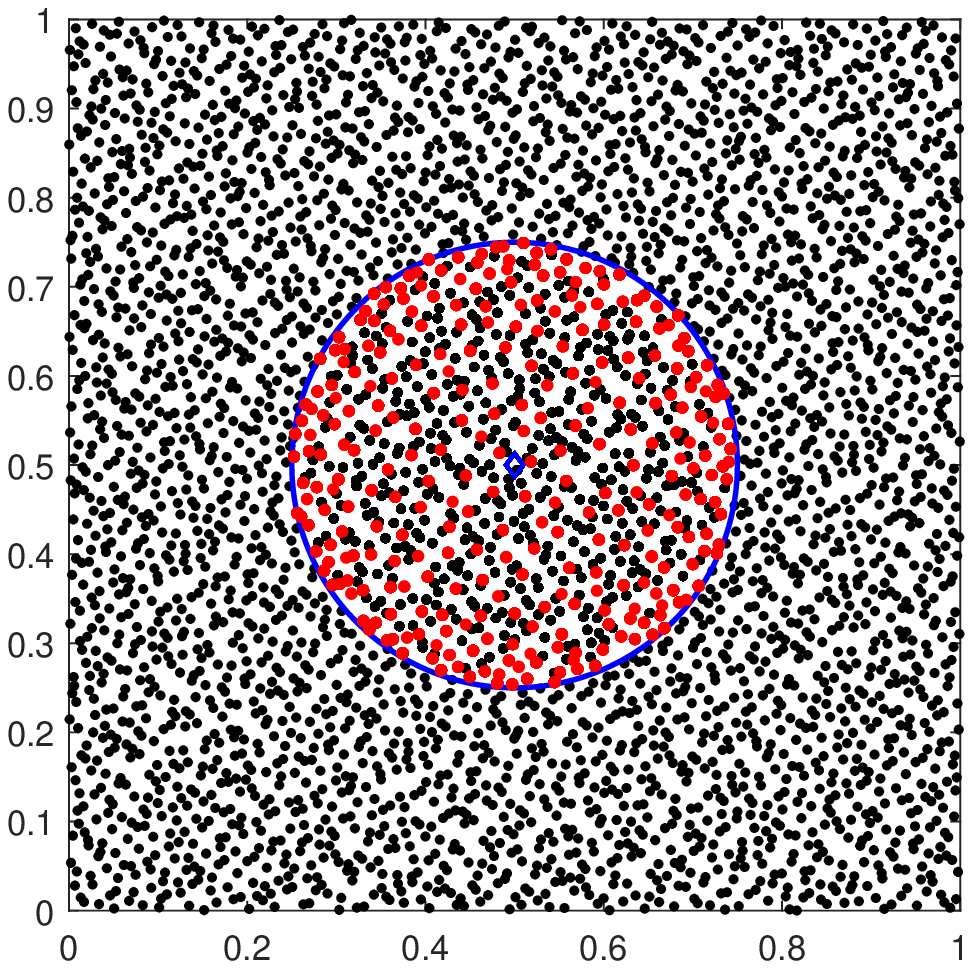}
} 
\parbox{.32\linewidth}{\centering
    \includegraphics[width=1.2\linewidth]{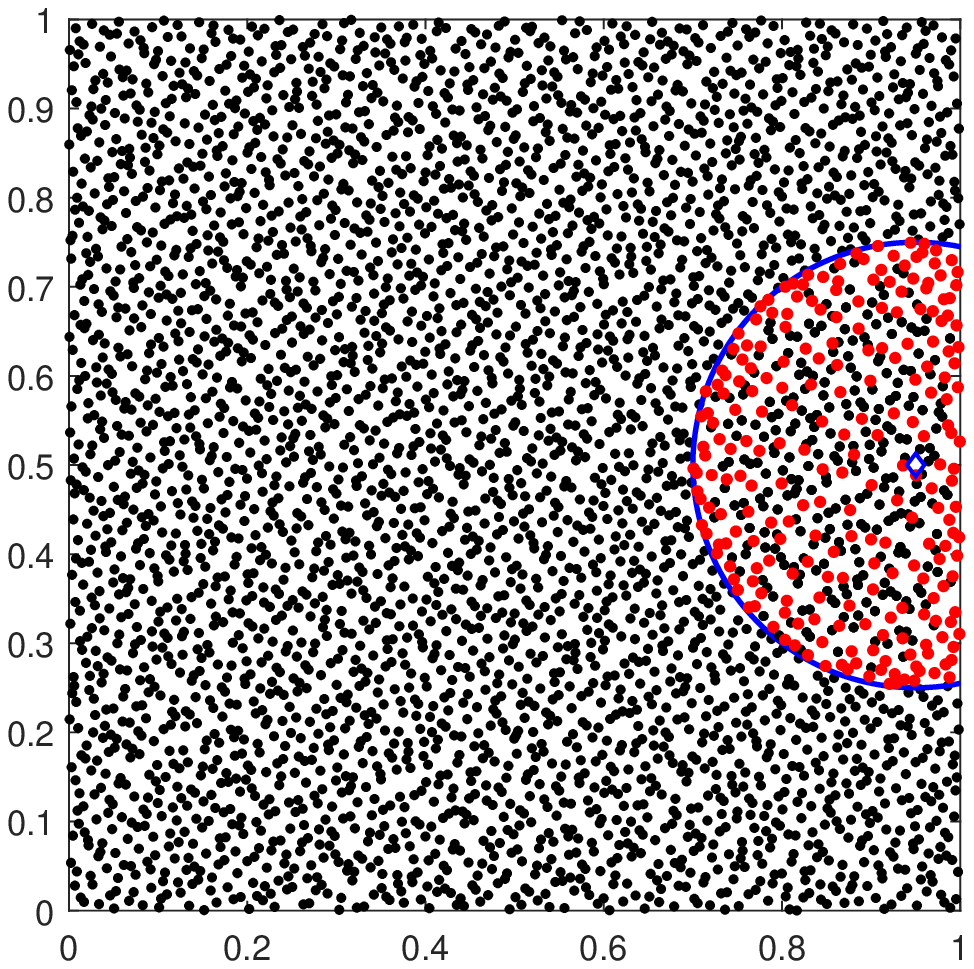}
} 
\parbox{.32\linewidth}{\centering
    \includegraphics[width=1.2\linewidth]{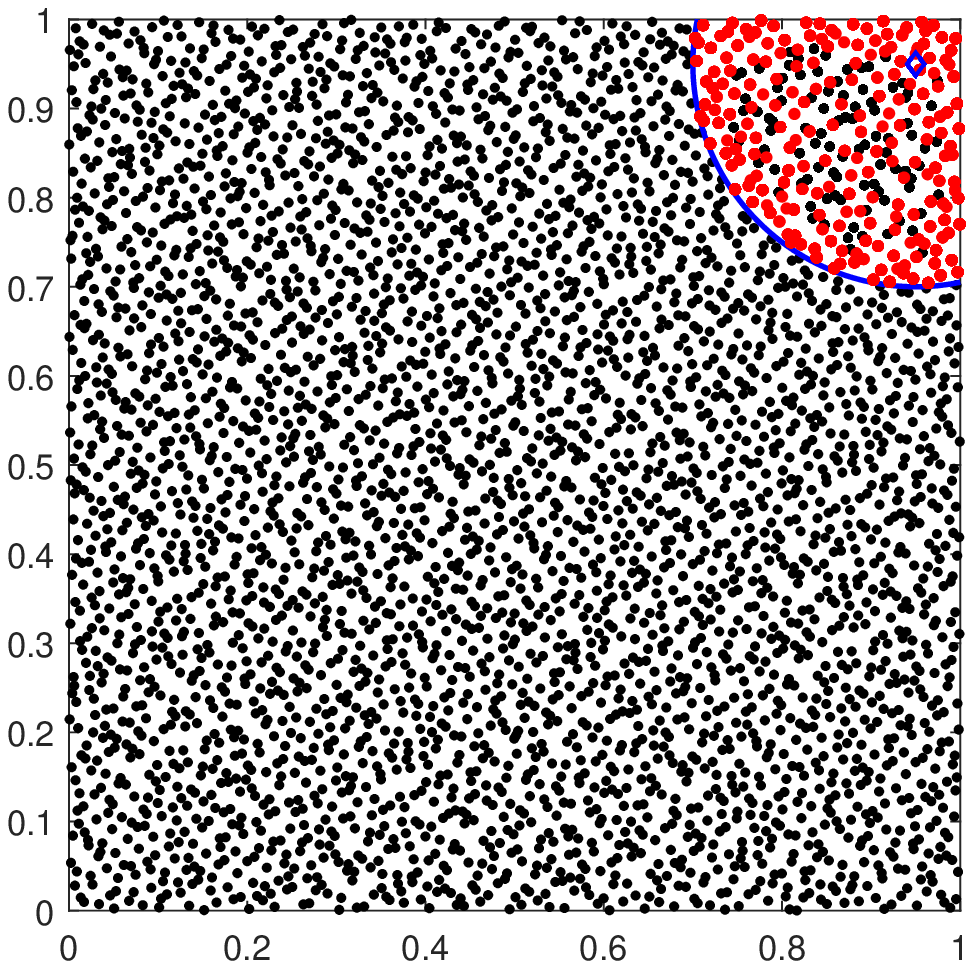}
}}
\caption{The interpolation points (in red) in the ball of radius 
$r=1/4$ centered at $\overline{\mathbf{x}}=(0.5,0.5)$ (left), $(0.95,0.5)$
(center), $(0.95,0.95)$ (right) selected from $4000$ Halton points for $d=20$.
}
\label{fig:points_in_the_ball}
\end{figure}

\begin{figure}[t]
{\small \centering 
\parbox{.32\linewidth}{\centering
    \includegraphics[width=1.1\linewidth]{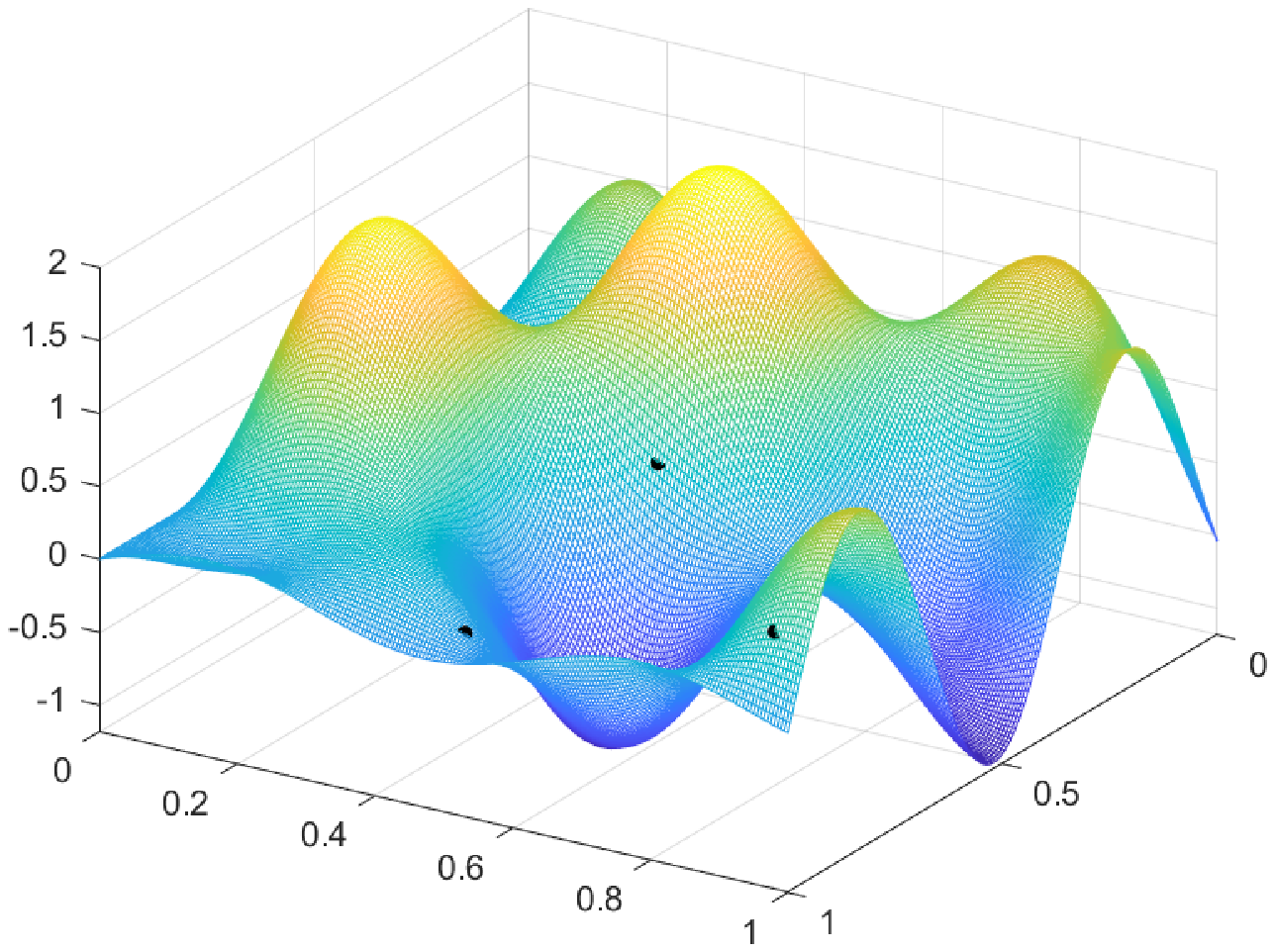}
} 
\parbox{.32\linewidth}{\centering
    \includegraphics[width=1.1\linewidth]{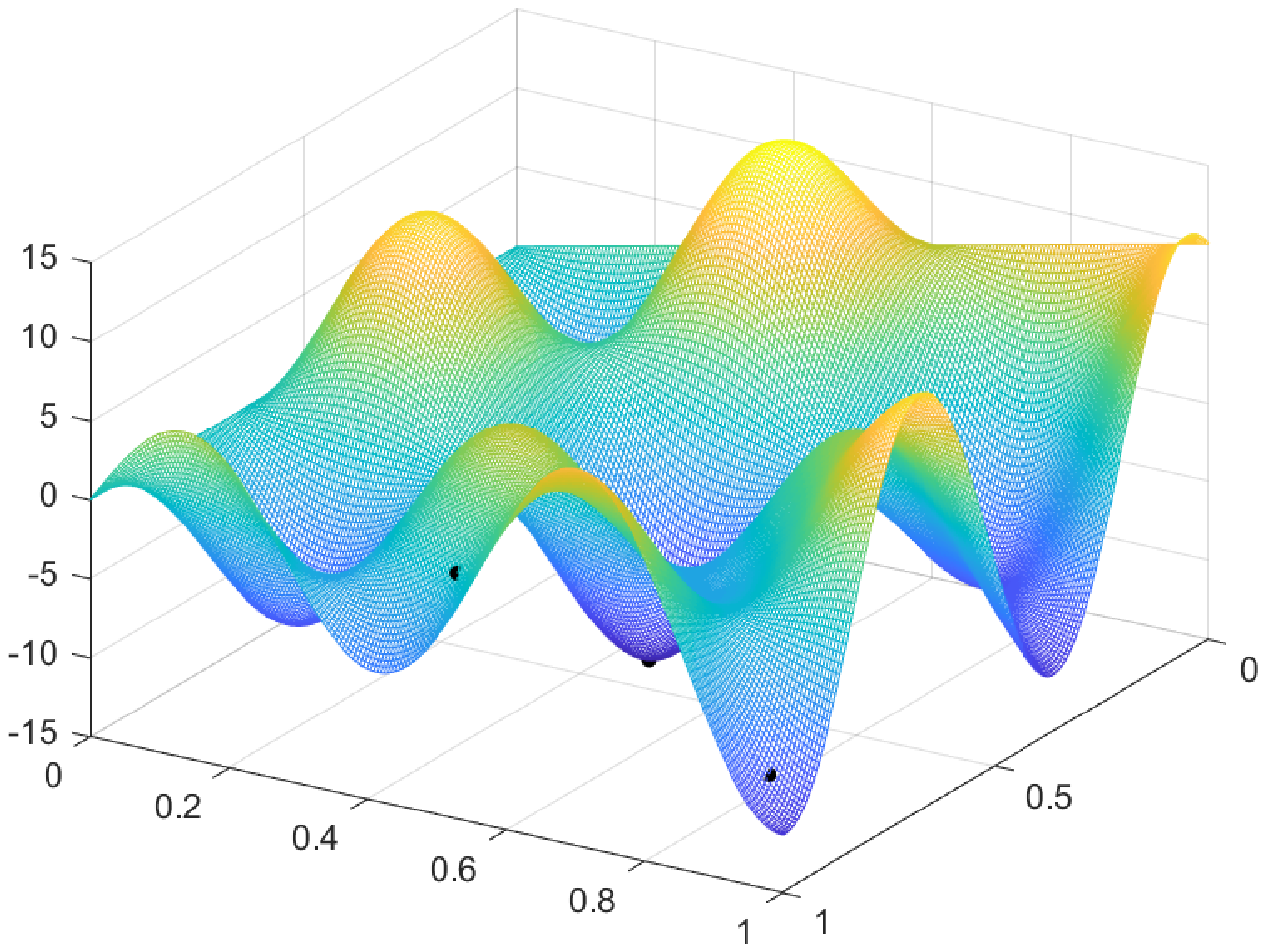}
} 
\parbox{.32\linewidth}{\centering
    \includegraphics[width=1.1\linewidth]{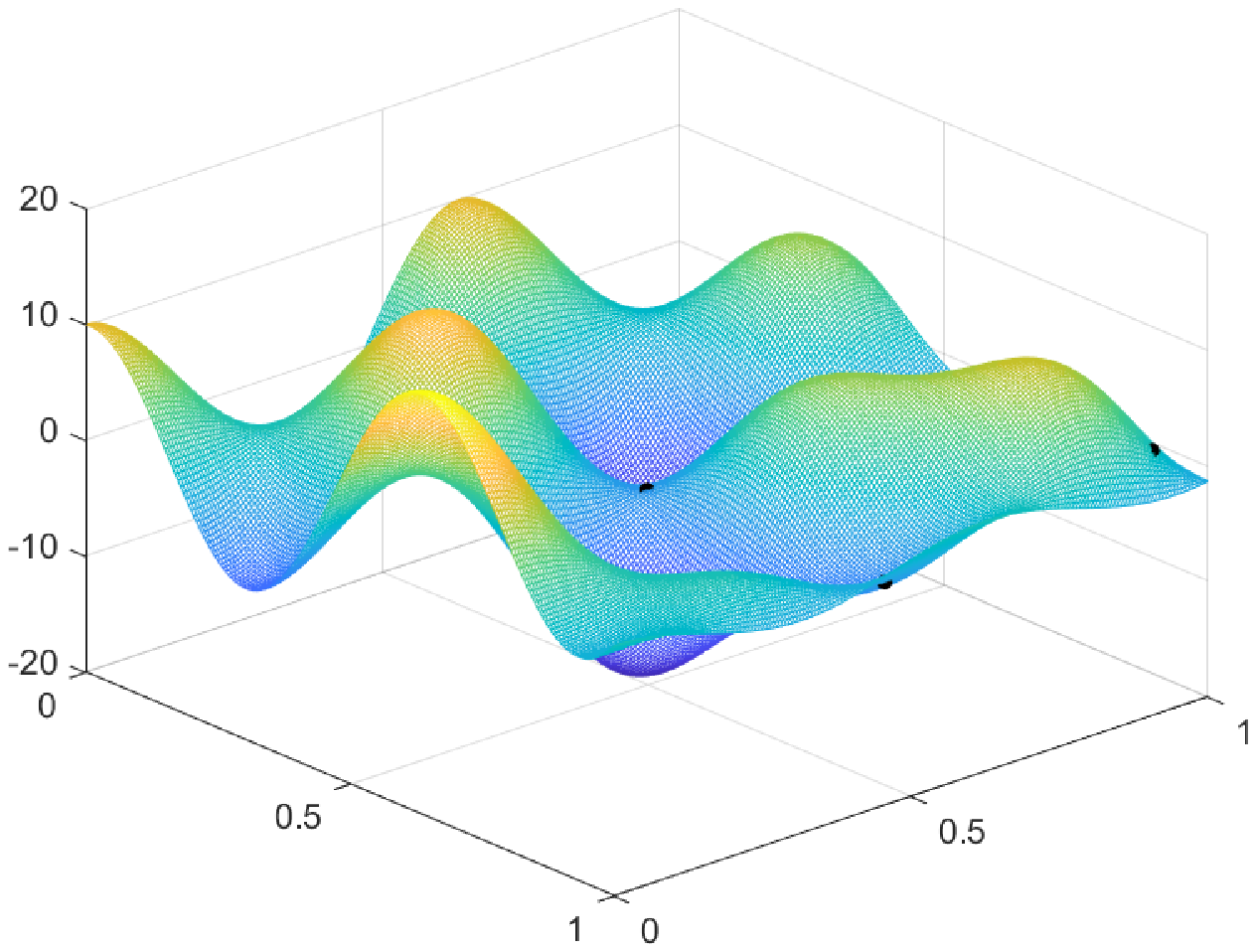}
}}
\caption{Plot of the function $f_3$ (left), $\frac{\partial f_3}{%
\partial x} $ (center), $\frac{\partial f_3}{\partial y} $ (right); the
surface points corresponding to $\overline{\mathbf{x}}=(0.5,0.5)$, $%
(0.95,0.5)$, and $(0.95,0.95)$ are displayed with black circles. }
\label{fig:cosine_function}
\end{figure}

\begin{figure}[t]
{\small \centering 
\parbox{.32\linewidth}{\centering
    \includegraphics[width=1.1\linewidth]{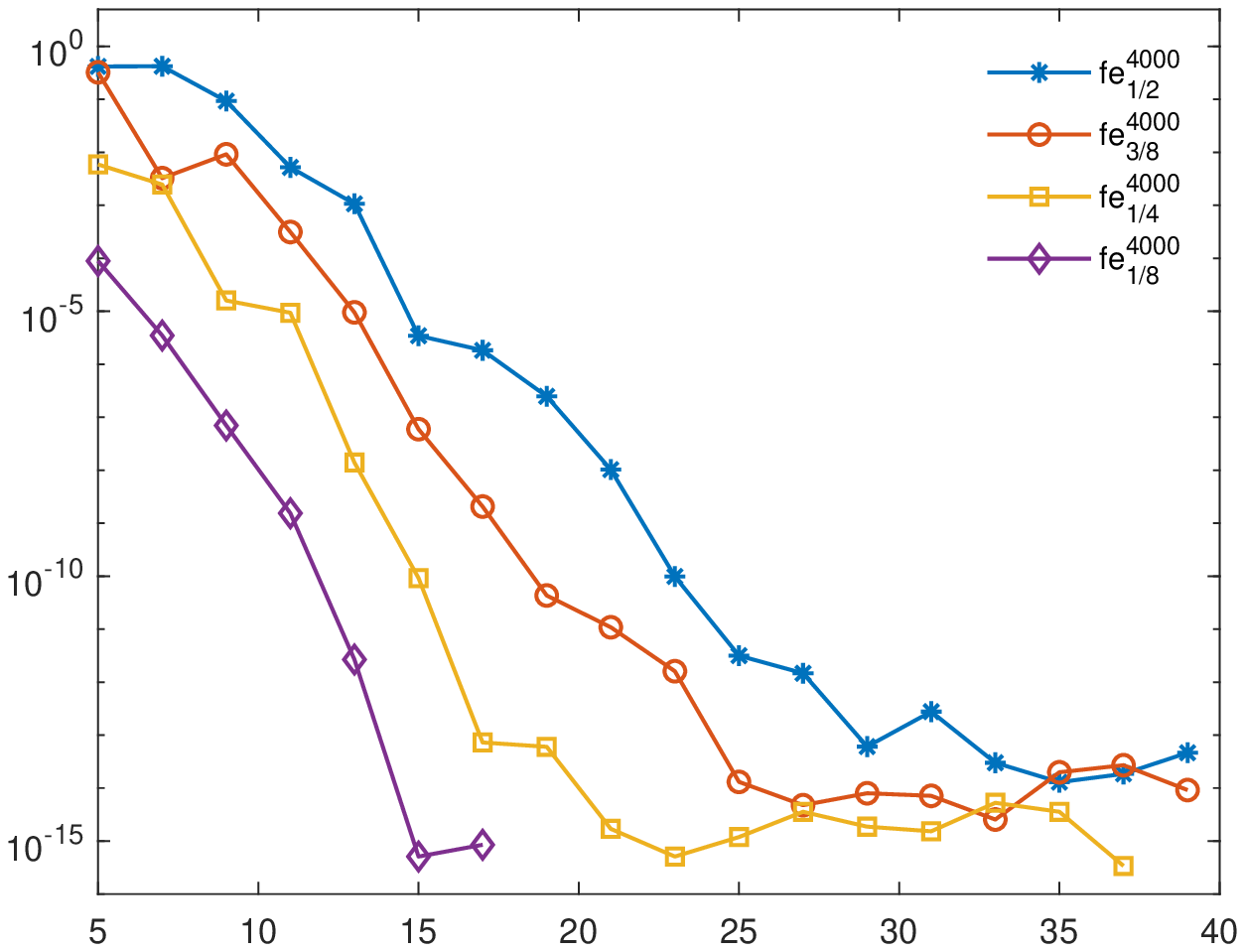}
} 
\parbox{.32\linewidth}{\centering
    \includegraphics[width=1.1\linewidth]{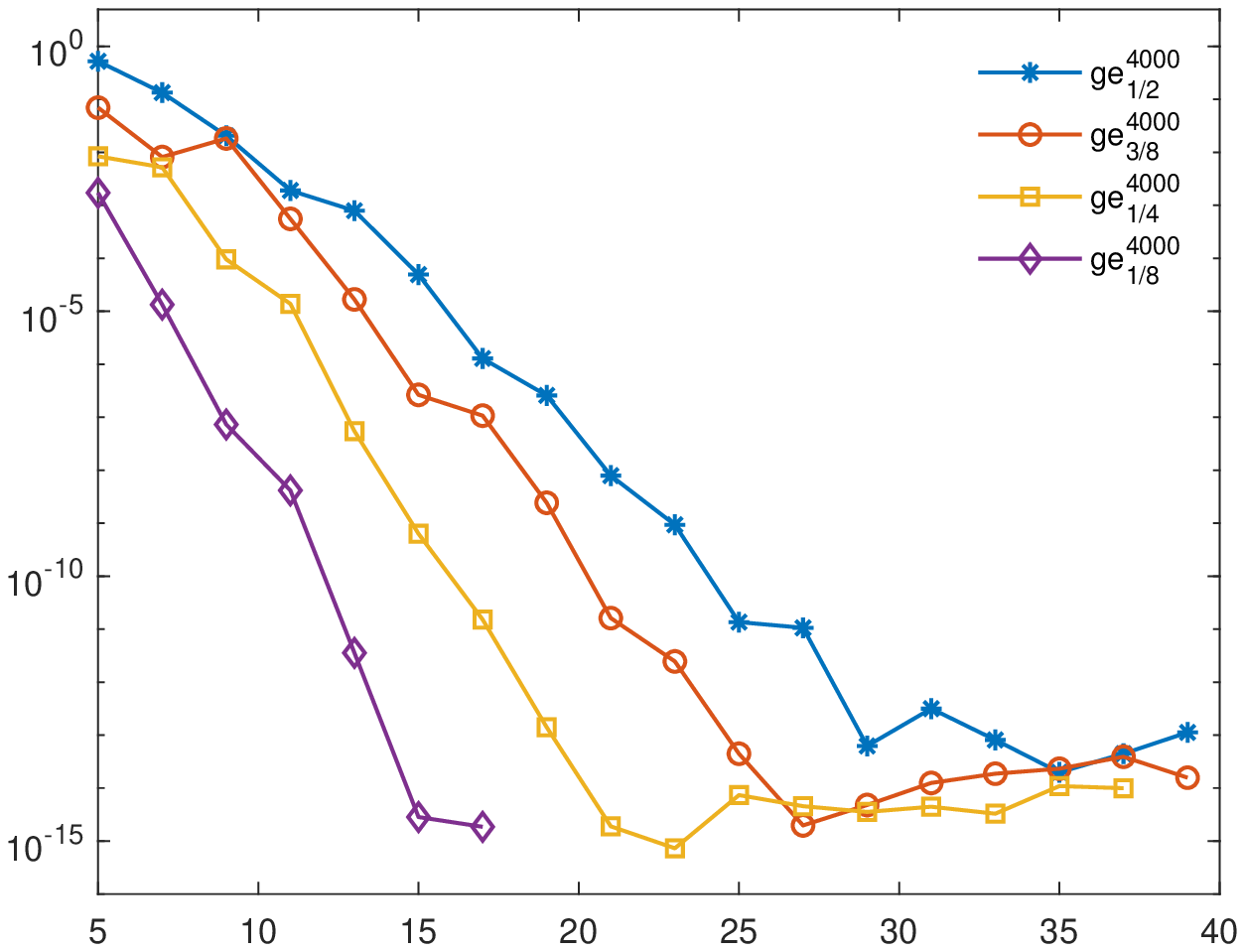}
} 
\parbox{.32\linewidth}{\centering
    \includegraphics[width=1.1\linewidth]{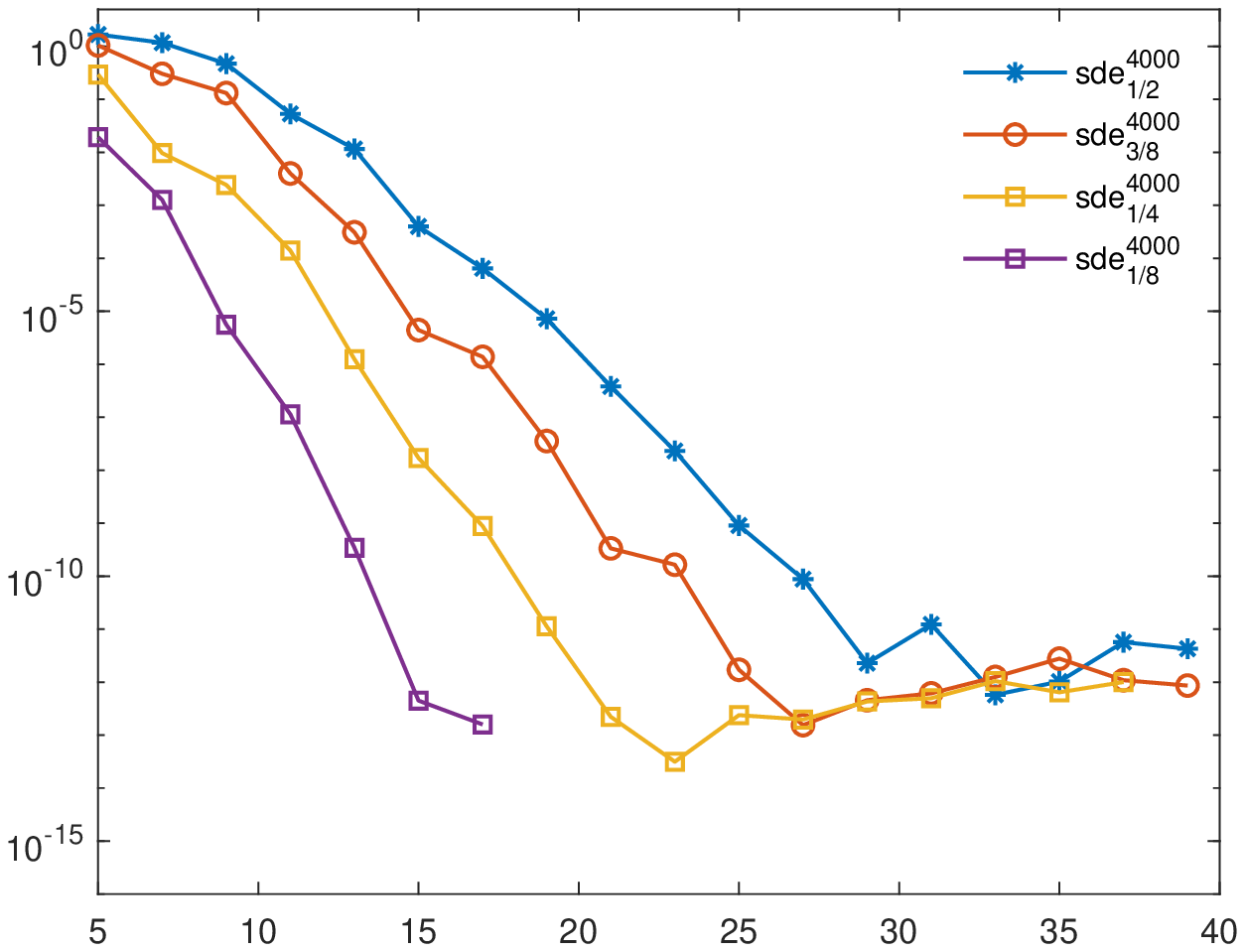}
}}
\caption{Relative errors \eqref{function_error}-\eqref{second_order_error}
for function $f_3$ by using $4000$ Halton points with $\overline{\mathbf{x}}%
=(0.5,0.5)$. }
\label{fig:Cosine_Halton_center}
\end{figure}

\begin{figure}[t]
{\small \centering 
\parbox{.32\linewidth}{\centering
    \includegraphics[width=1.1\linewidth]{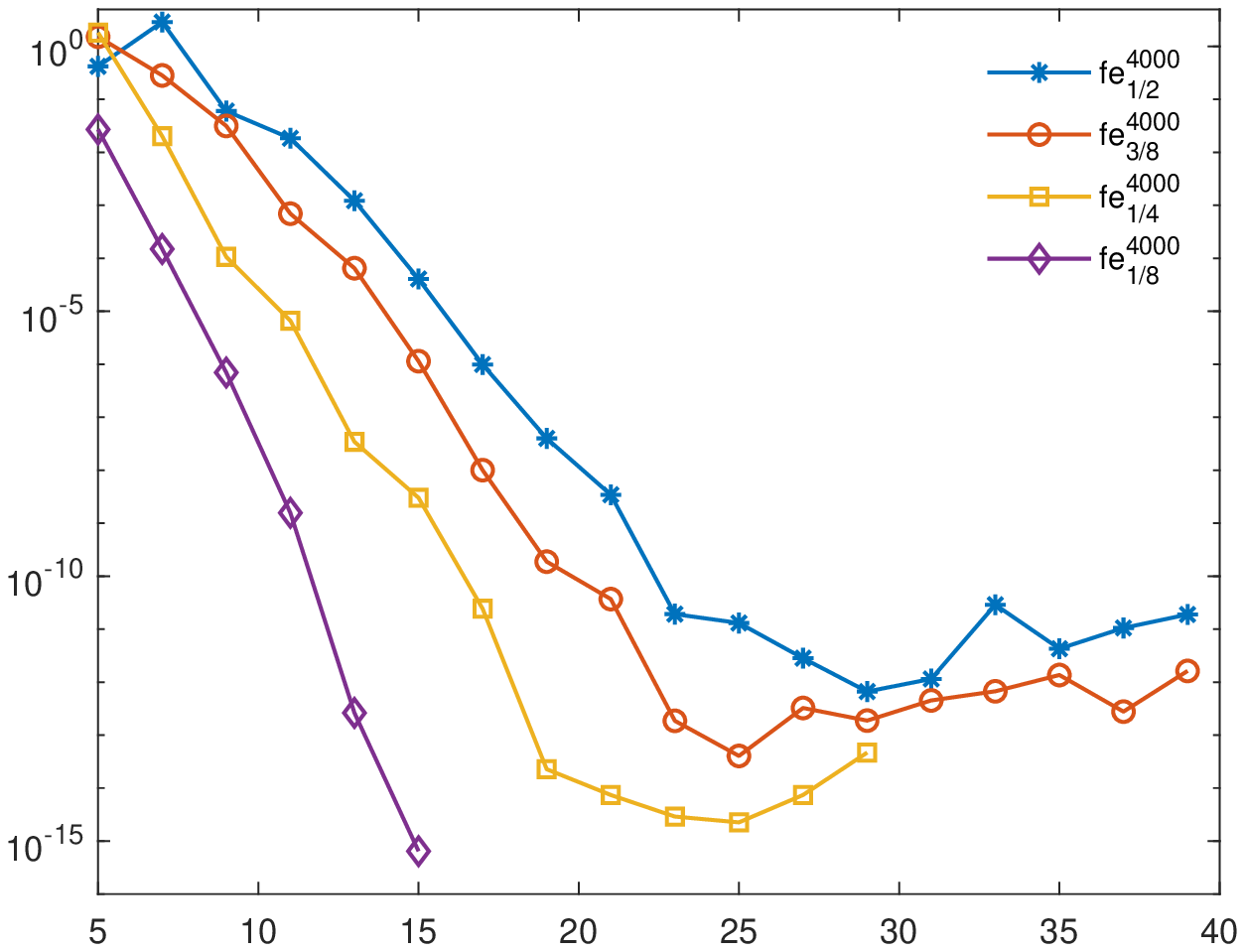}
} 
\parbox{.32\linewidth}{\centering
    \includegraphics[width=1.1\linewidth]{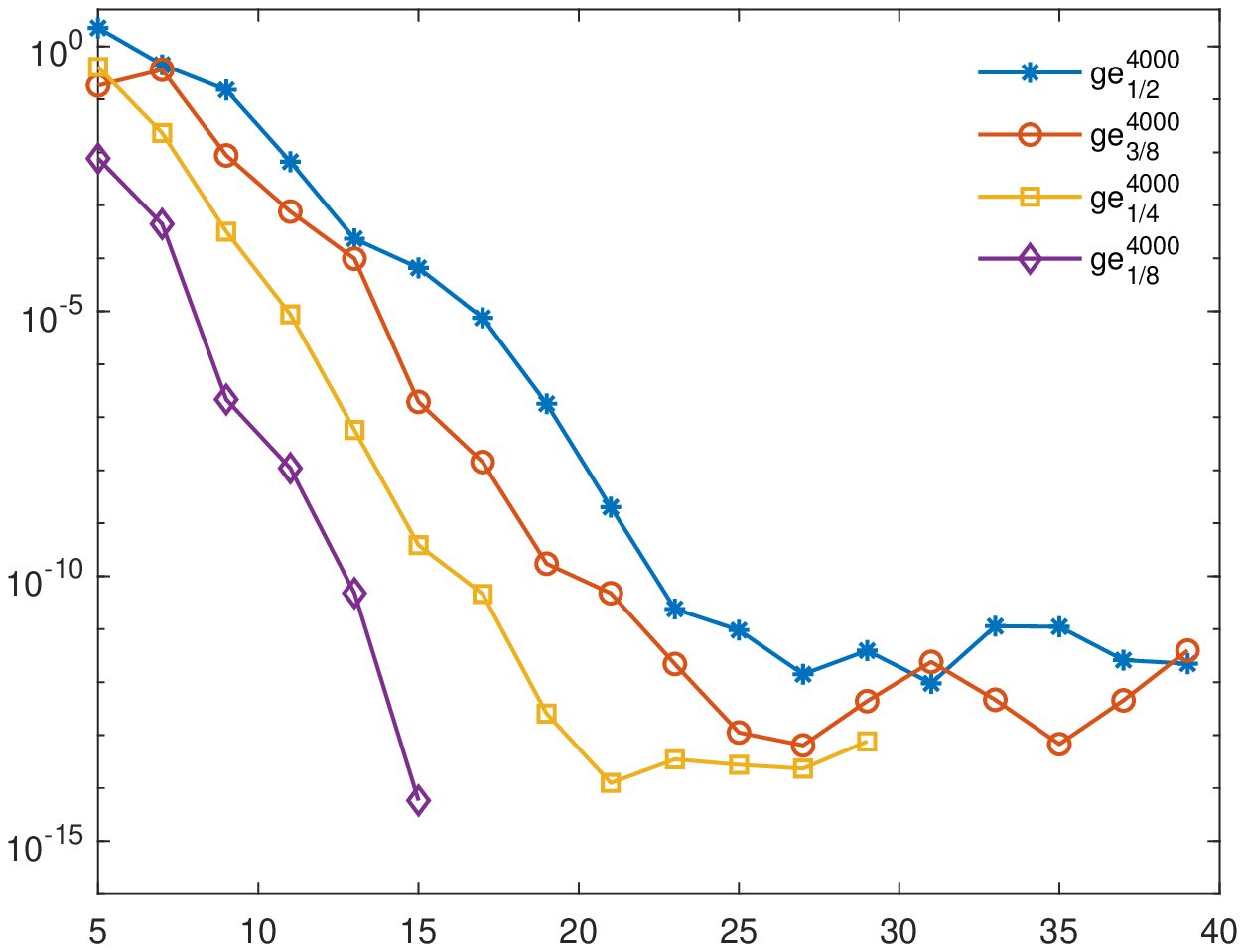}
} 
\parbox{.32\linewidth}{\centering
    \includegraphics[width=1.1\linewidth]{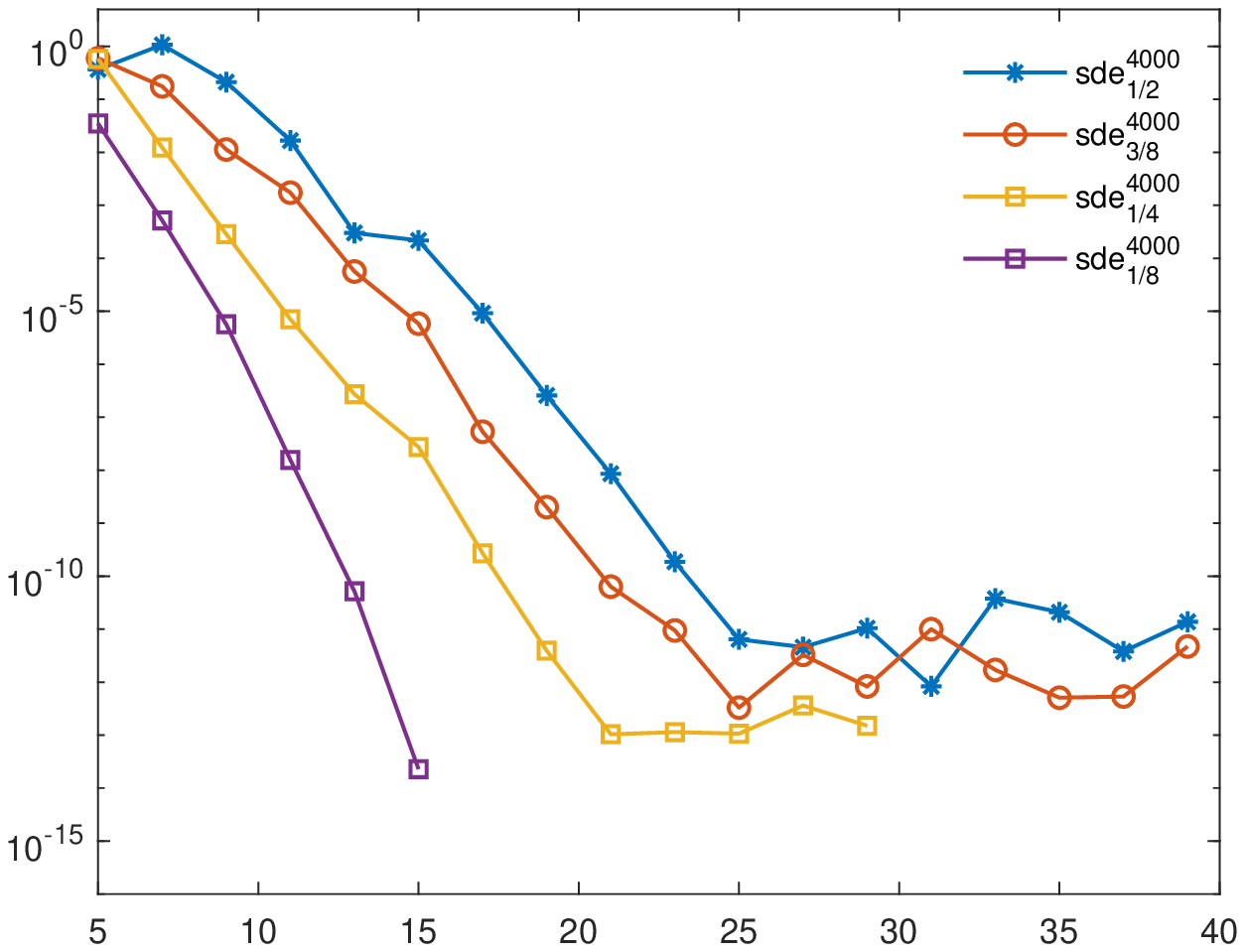}
}}
\caption{Relative errors \eqref{function_error}-\eqref{second_order_error}
for function $f_3$ by using $4000$ Halton points with $\overline{\mathbf{x}}%
=(0.95,0.5)$. }
\label{fig:Cosine_Halton_side}
\end{figure}

\begin{figure}[t]
{\small \centering 
\parbox{.32\linewidth}{\centering
    \includegraphics[width=1.1\linewidth]{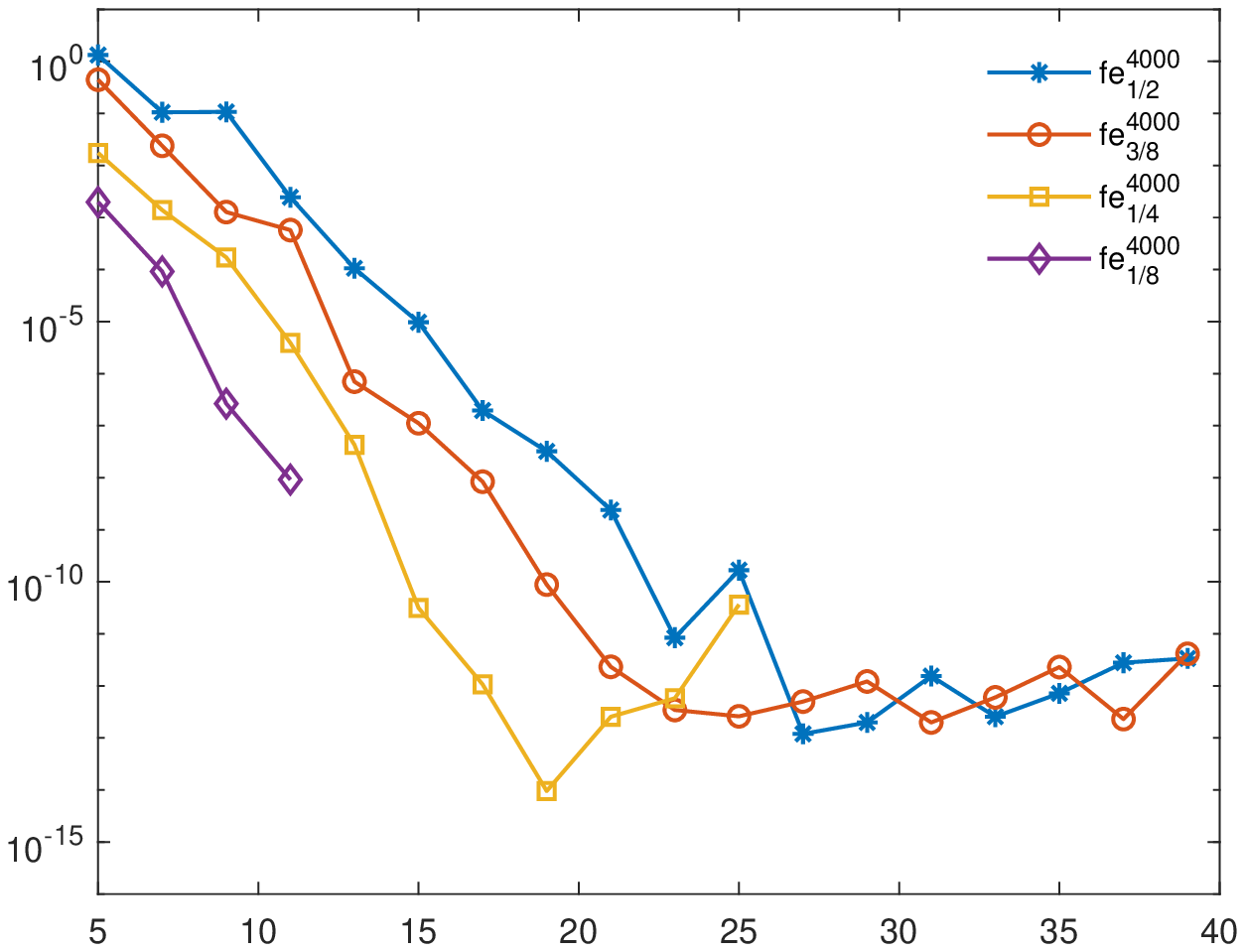}
} 
\parbox{.32\linewidth}{\centering
    \includegraphics[width=1.1\linewidth]{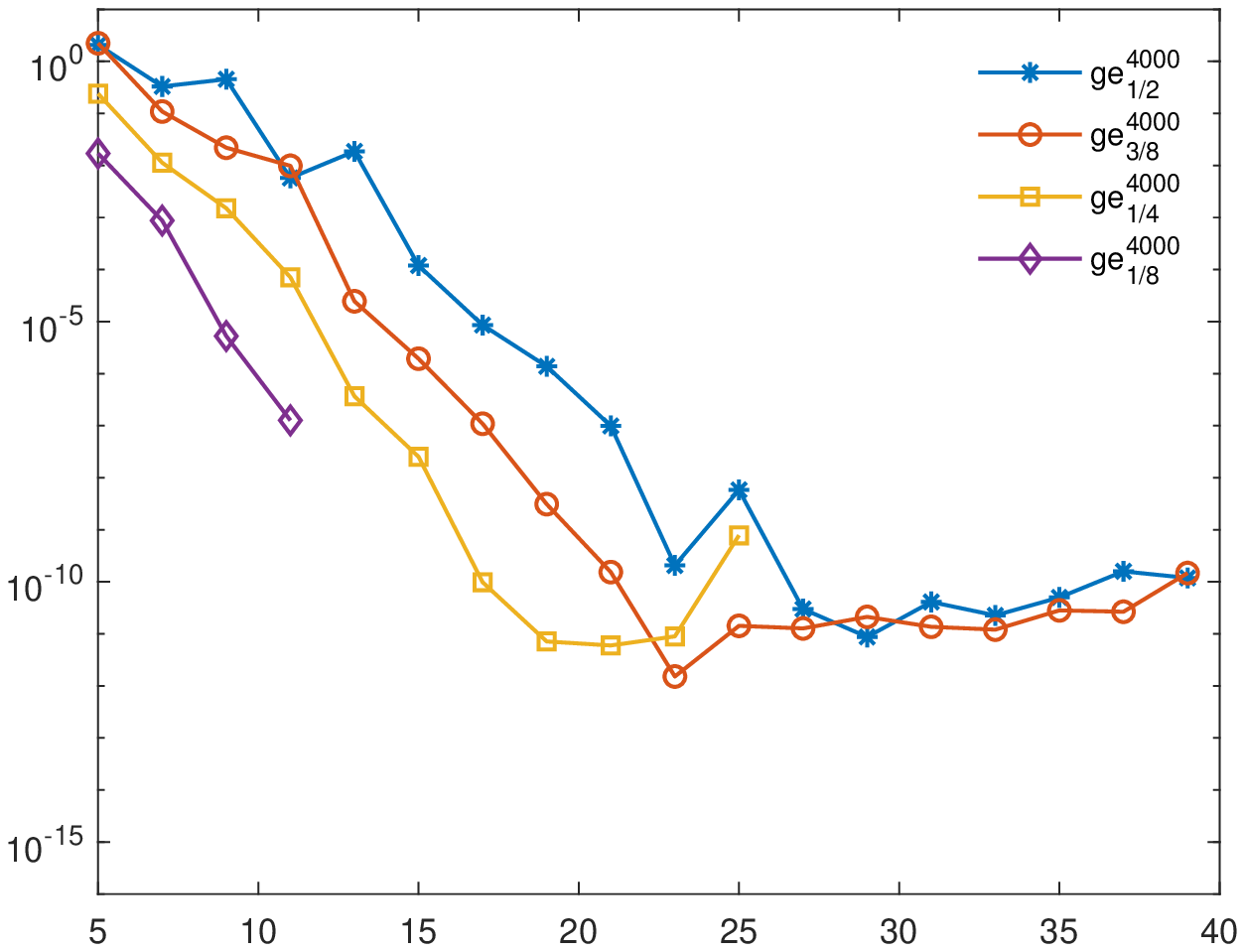}
} 
\parbox{.32\linewidth}{\centering
    \includegraphics[width=1.1\linewidth]{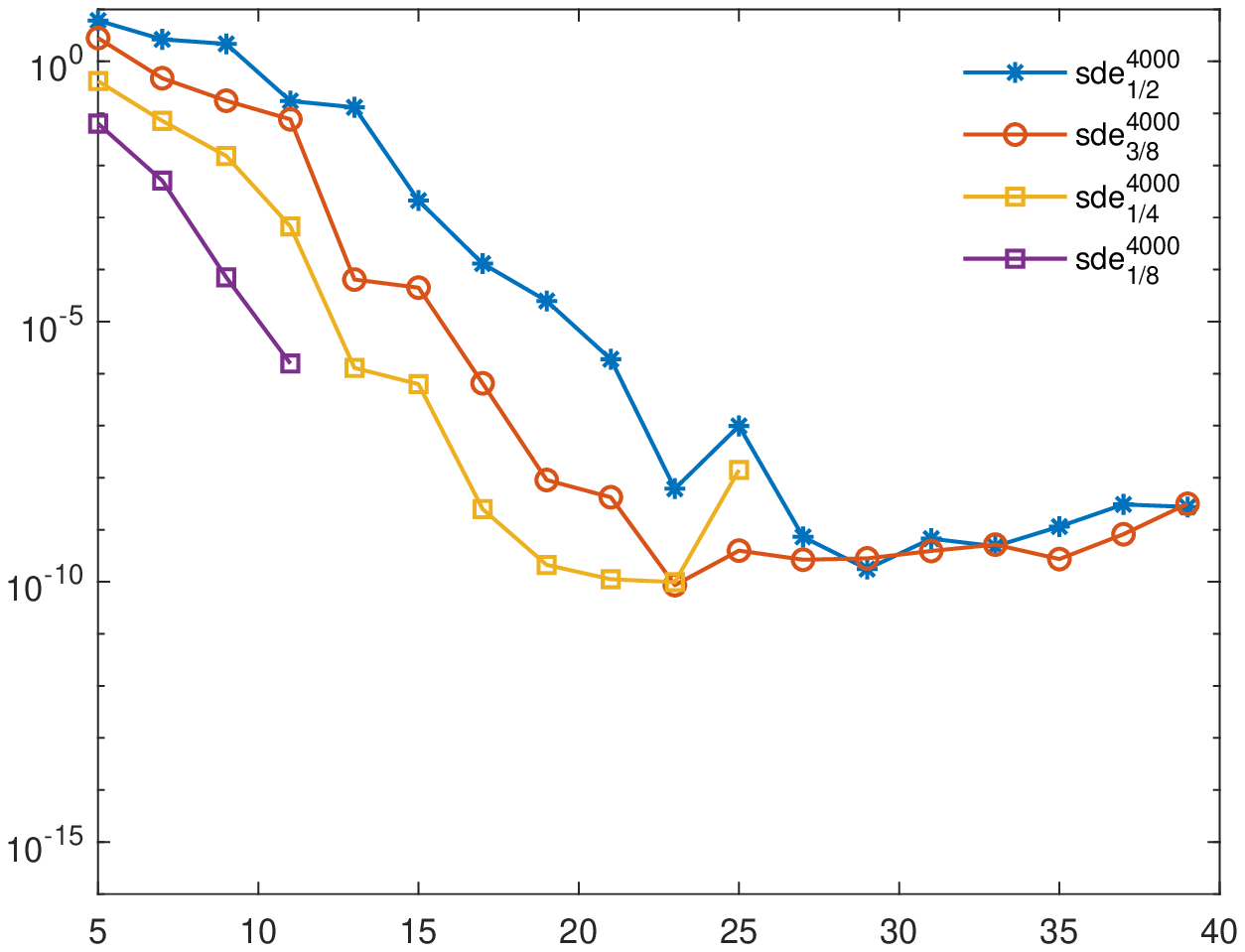}
}}
\caption{Relative errors \eqref{function_error}-\eqref{second_order_error}
for function $f_3$ by using $4000$ Halton points with $\overline{\mathbf{x}}%
=(1,0.5)$.}
\label{fig:Cosine_4000_Halton_exact_side}
\end{figure}
\begin{figure}[t]
{\small \centering 
\parbox{.32\linewidth}{\centering
    \includegraphics[width=1.1\linewidth]{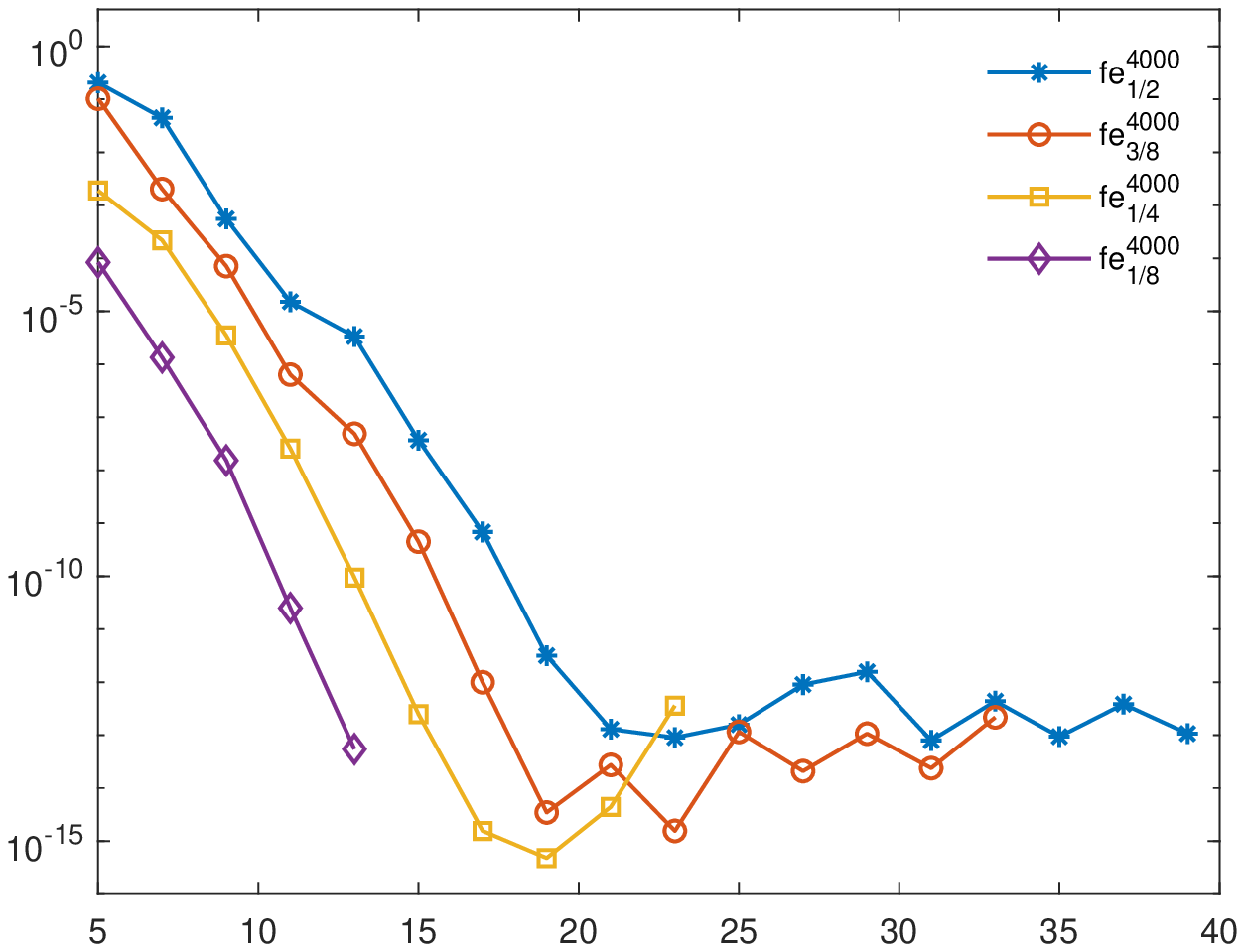}
} 
\parbox{.32\linewidth}{\centering
    \includegraphics[width=1.1\linewidth]{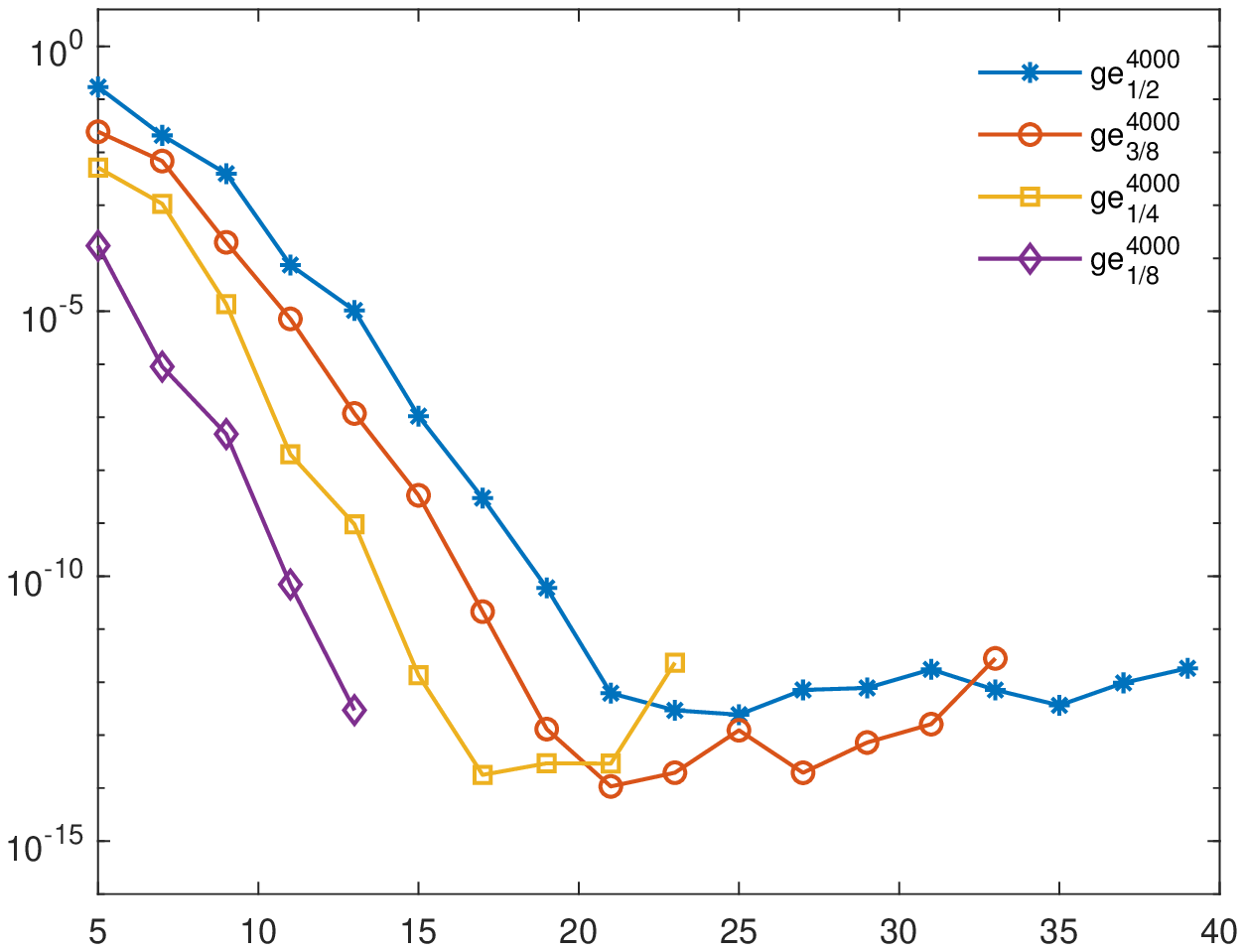}
} 
\parbox{.32\linewidth}{\centering
    \includegraphics[width=1.1\linewidth]{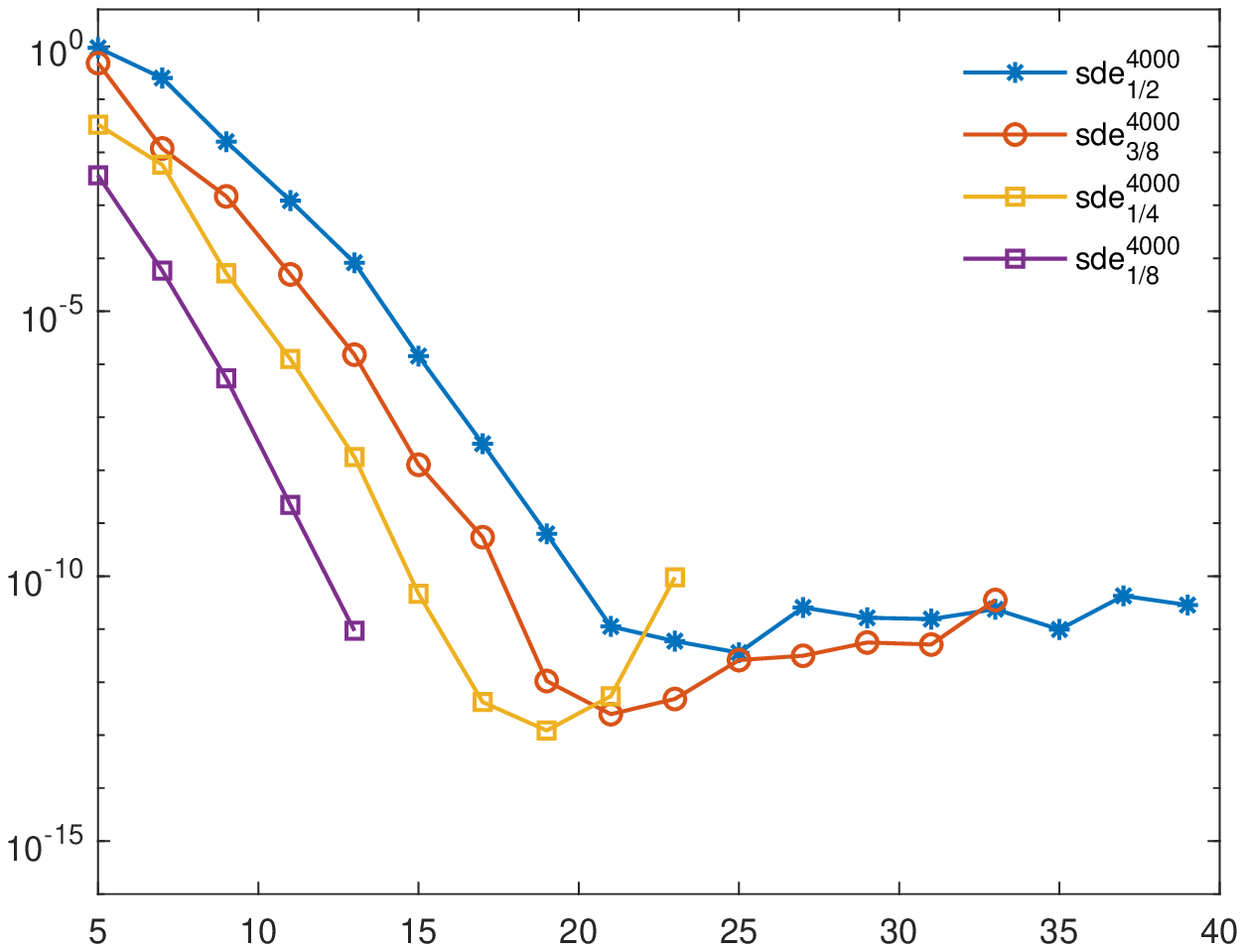}
}}
\caption{Relative errors \eqref{function_error}-\eqref{second_order_error}
for function $f_3$ by using $4000$ Halton points with $\overline{\mathbf{x}}%
=(0.95,0.95)$.}
\label{fig:Cosine_Halton_vertex}
\end{figure}

\begin{figure}[t]
{\small \centering 
\parbox{.32\linewidth}{\centering
    \includegraphics[width=1.1\linewidth]{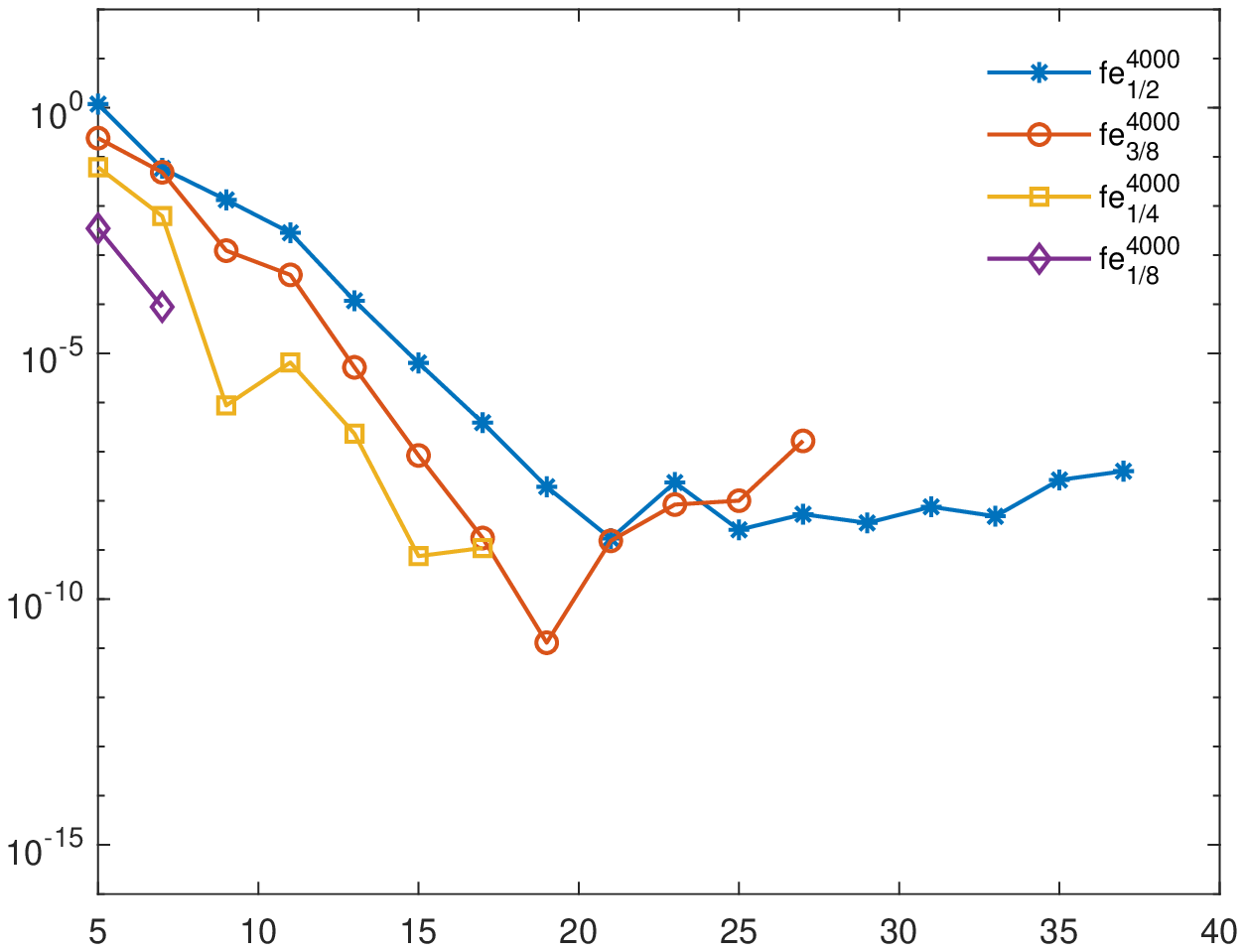}
} 
\parbox{.32\linewidth}{\centering
    \includegraphics[width=1.1\linewidth]{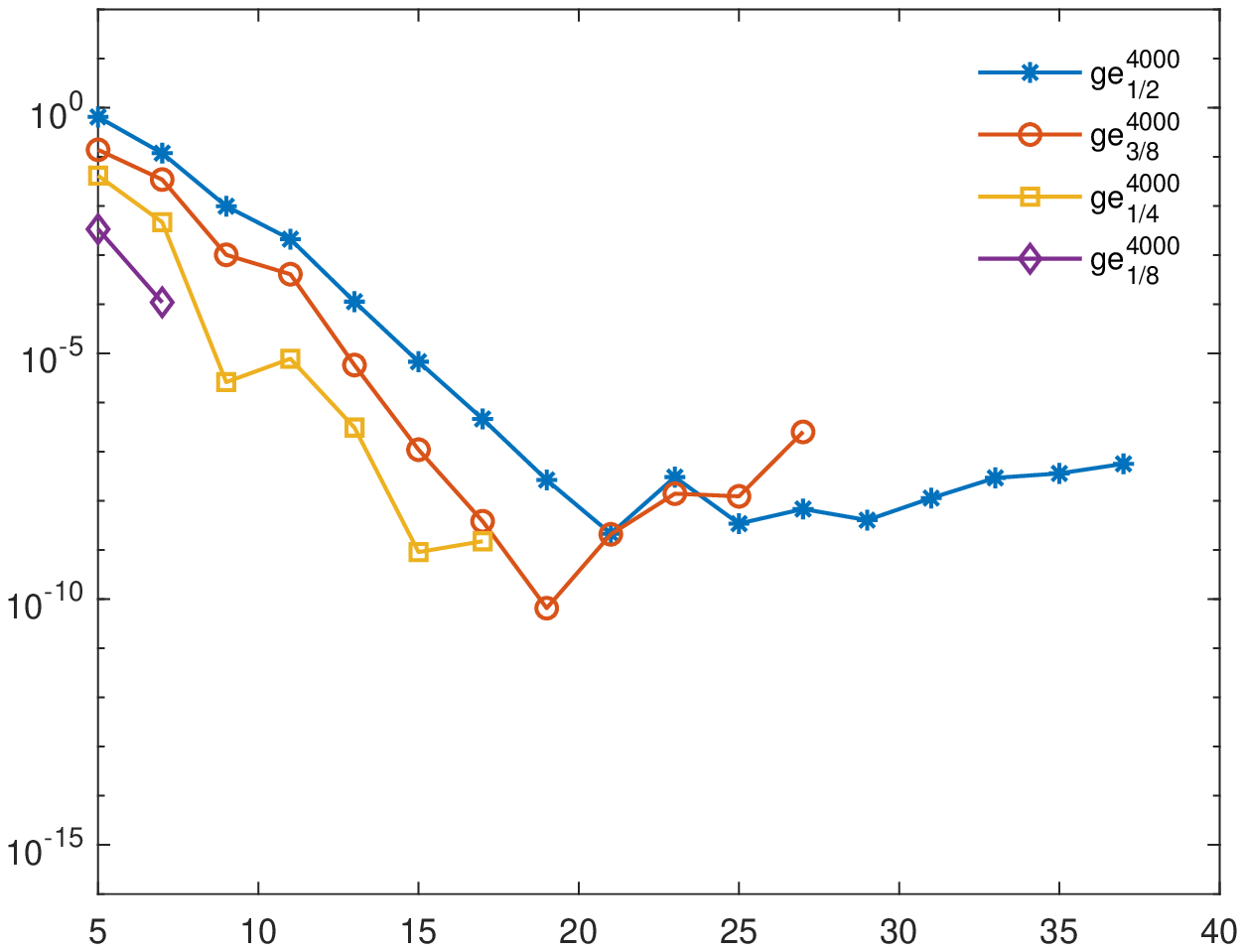}
} 
\parbox{.32\linewidth}{\centering
    \includegraphics[width=1.1\linewidth]{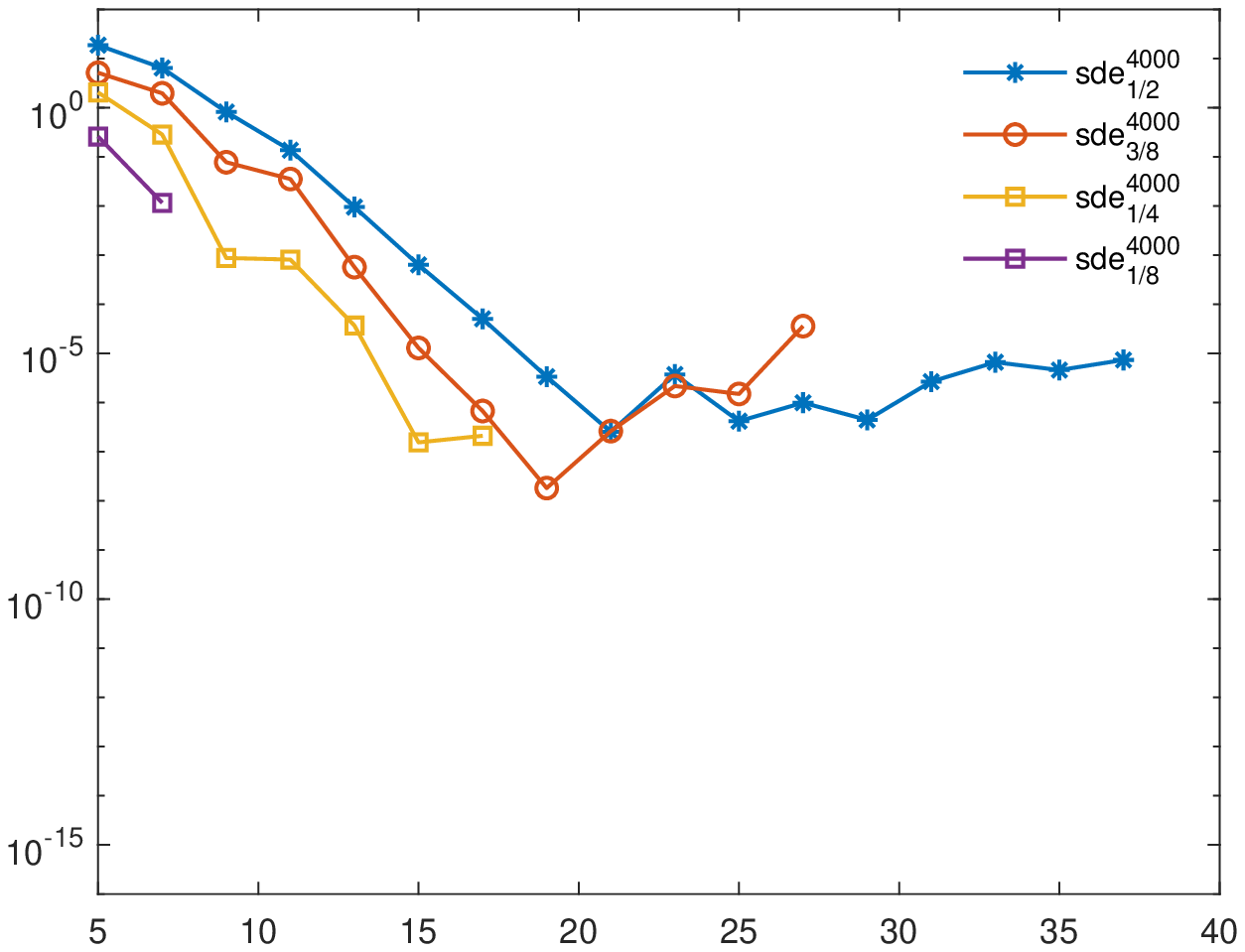}
}}
\caption{Relative errors \eqref{function_error}-\eqref{second_order_error}
for function $f_3$ by using $4000$ Halton points with $\overline{\mathbf{x}}%
=(1,1)$.}
\label{fig:Cosine_4000_Halton_exact_vertex}
\end{figure}

In the last experiment, for each test function $f_{k}$, $k=1,2,3$, we
include a random noise in the $m$ function values, namely 
\begin{equation}
\tilde{y}=y+U(-\varepsilon ,\varepsilon ),  \label{perturbation}
\end{equation}%
where $U(-\varepsilon ,\varepsilon )$ denotes the multivariate uniform
distribution in $[-\varepsilon ,\varepsilon ]^{m}$. In Figure \ref%
{fig:gradient_perturb} we display the relative error 
\begin{equation}
gep=\frac{\left\Vert \nabla {f}(\overline{\mathbf{x}})-\nabla \widetilde{{p}}%
(\overline{\mathbf{x}})\right\Vert _{2}}{\left\Vert \nabla {f}(\overline{%
\mathbf{x}})\right\Vert _{2}}  \label{gradient_error_perturbed}
\end{equation}
for the gradient at $%
\overline{\mathbf{x}}=\left( 0.5,0.5\right) $, computed using $1000$ Halton
points with exact function values \eqref{gradient_error} and perturbed function values (\ref{gradient_error_perturbed}) for $\varepsilon =10^{-6}$.
In Figure \ref{fig:gradient_perturb_polynomial} we display the
relative sensitivity in computing the gradient of the interpolating polynomial $p$
under the perturbation of the function values ($\varepsilon =10^{-6}$) 
\begin{equation}
gs=\frac{\left\Vert \nabla {p}(\overline{\mathbf{x}})-\nabla \widetilde{{p}}(%
\overline{\mathbf{x}})\right\Vert _{2}}{\left\Vert \nabla {f}(\overline{%
\mathbf{x}})\right\Vert _{2}},  \label{gradient_error_pert}
\end{equation}%
together with its estimate involving the stability constant of the gradient %
\eqref{stabconst} 
\begin{equation}
gse=\frac{\left\Vert \left( \varepsilon
h^{-1}\sum\limits_{i=1}^{m}\left\vert a_{(1,0),h}^{i}\right\vert
,\varepsilon h^{-1}\sum\limits_{i=1}^{m}\left\vert
a_{(0,1),h}^{i}\right\vert \right) \right\Vert _{2}}{\left\Vert \nabla {f}(%
\overline{\mathbf{x}})\right\Vert _{2}}.
\label{gradient_error_pert_estimate}
\end{equation}%
Notice that the relative errors $gep$ in Figure \ref{fig:gradient_perturb}
and sensitivity $gs$ in Figure \ref{fig:gradient_perturb_polynomial} are of the same
order of magnitude when the errors $ge$ become negligible with respect to $gs$. Moreover, $gse$ turns out to be a slight overestimate of the relative sensitivity $gs$. This fact, as we can see in Table \ref{tab:estimates_pertub} where $\overline{\mathbf{x}}=(0.5,0.5)$, is due to the relative small values of the stability constant that varies slowly while decreasing the radius $r$ or increasing the degree of the interpolating polynomial.

Finally, it is worth stressing that the stability constant \eqref{stabconst} is a function of $\overline{\mathbf{x}}$ and then it depends on the position of $\overline{\mathbf{x}}$ in the square. More precisely, for a fixed interpolation degree $d$, it slowly varies except for a neighborhood of the boundary where it increases rapidly, expecially near the vertices, as can be observed in Figure \ref{fig:stability_constant_plot}, where for clarity we restrict the stability constant to lines. On the other hand, it can be noticed that by increasing the degree, the stability constant tends to increase, much more rapidly at the boundary. Such a behavior, that explains the worsening of the accuracy at the
boundary (see Figures \ref{fig:Cosine_Halton_side}-\ref{fig:Cosine_4000_Halton_exact_side}) and at the vertex (see Figures \ref{fig:Cosine_Halton_vertex}-\ref{fig:Cosine_4000_Halton_exact_vertex}) of the
square, can be ascribed to the fact that the stability functions \eqref{stability_constant}
increase rapidly near the boundary of the local interpolation
domains $B_h(\overline{\mathbf{x}}) \cap \Omega$. This phenomenon, that
resembles the behavior at the boundary of Lebesgue functions of
univariate interpolation at equispaced
points \cite{trefethen2019approximation}, is worth of further investigation.

\begin{table}[tbp]
\begin{center}
{\scriptsize \ 
\begin{tabular}{c@{\hspace{.4cm}}c|@{\hspace{.4cm}}c|@{\hspace{.4cm}}c|@{\hspace{.4cm}}c|@{\hspace{.4cm}}c|@{\hspace{.4cm}}c|@{\hspace{.4cm}}c}
&  &  & $d=5$ & $d=10$ & $d=15$ & $d=20$ & $d=25$ \\ \hline
&  & $r=1/2$ & 2.31 & 2.43 & 6.69 & 2.41e+1 & 3.51e+1 \\ 
& $\left \vert \nu\right \vert =0$ & $r=3/8$ & 1.75 & 4.10 & 1.11e+1 & 
2.91e+1 & 3.03e+1 \\ 
&  & $r=1/4$ & 2.14 & 4.73 & 7.16 & - & - \\ 
&  & $r=1/8$ & 1.80 & - & - & - & - \\ \hline
&  & $r=1/2$ & 2.63e+1 & 7.26e+1 & 4.53e+2 & 9.06e+2 & 7.74e+2 \\ 
& $\left \vert \nu\right \vert =1$ & $r=3/8$ & 2.85e+1 & 1.64e+2 & 3.51e+2 & 
6.04e+2 & 9.55e+2 \\ 
&  & $r=1/4$ & 3.61e+1 & 1.67e+2 & 3.84e+2 & - & - \\ 
&  & $r=1/8$ & 1.27e+2 & - & - & - & - \\ \hline
&  & $r=1/2$ & 9.94e+1 & 1.41e+3 & 3.30e+3 & 1.82e+4 & 3.05e+4 \\ 
& $\left \vert \nu\right \vert =2$ & $r=3/8$ & 1.72e+2 & 2.80e+3 & 7.94e+3 & 
3.61e+4 & 5.15e+4 \\ 
&  & $r=1/4$ & 4.02e+2 & 4.54e+3 & 2.02e+4 & - & - \\ 
&  & $r=1/8$ & 1.73e+3 & - & - & - & - \\ \hline
\end{tabular}%
}
\end{center}
\caption{Numerical comparison among the mean value of the 
stability constant (\protect\ref{stabconst}) with $\left\vert \protect\nu %
\right\vert \leq 2$, for interpolation at a sequence of degrees $d$ on
Discrete Leja points extracted from the subset of $1000$ Halton points in $%
[0,1]^2$ contained in the ball $B_r(\overline{\mathbf{x}})$ centered at $%
\overline{\mathbf{x}}=(0.5,0.5)$.}
\label{tab:estimates_pertub}
\end{table}

\begin{figure}[t]
{\small \centering 
\parbox{.32\linewidth}{\centering
    \includegraphics[width=1.1\linewidth]{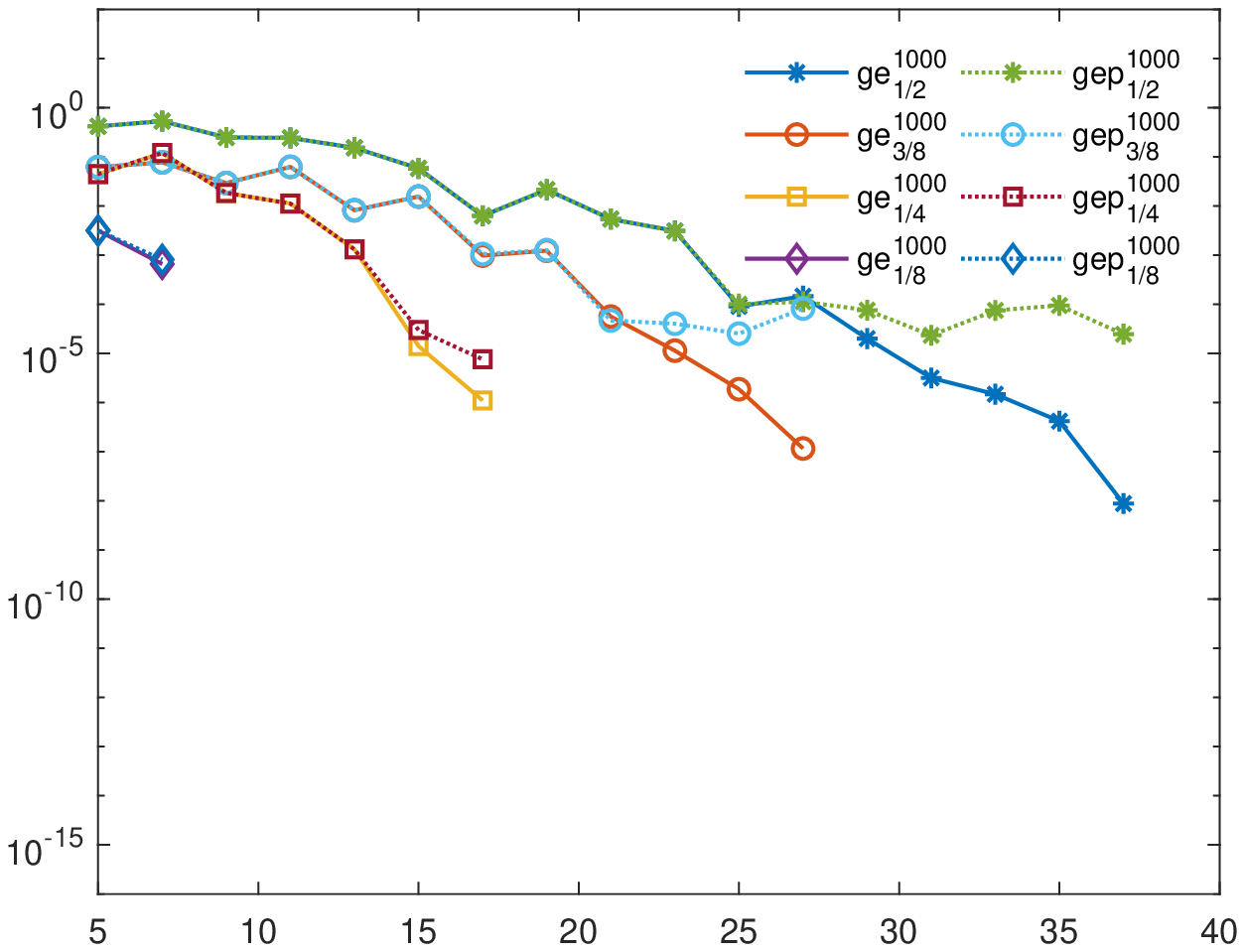}
} 
\parbox{.32\linewidth}{\centering
    \includegraphics[width=1.1\linewidth]{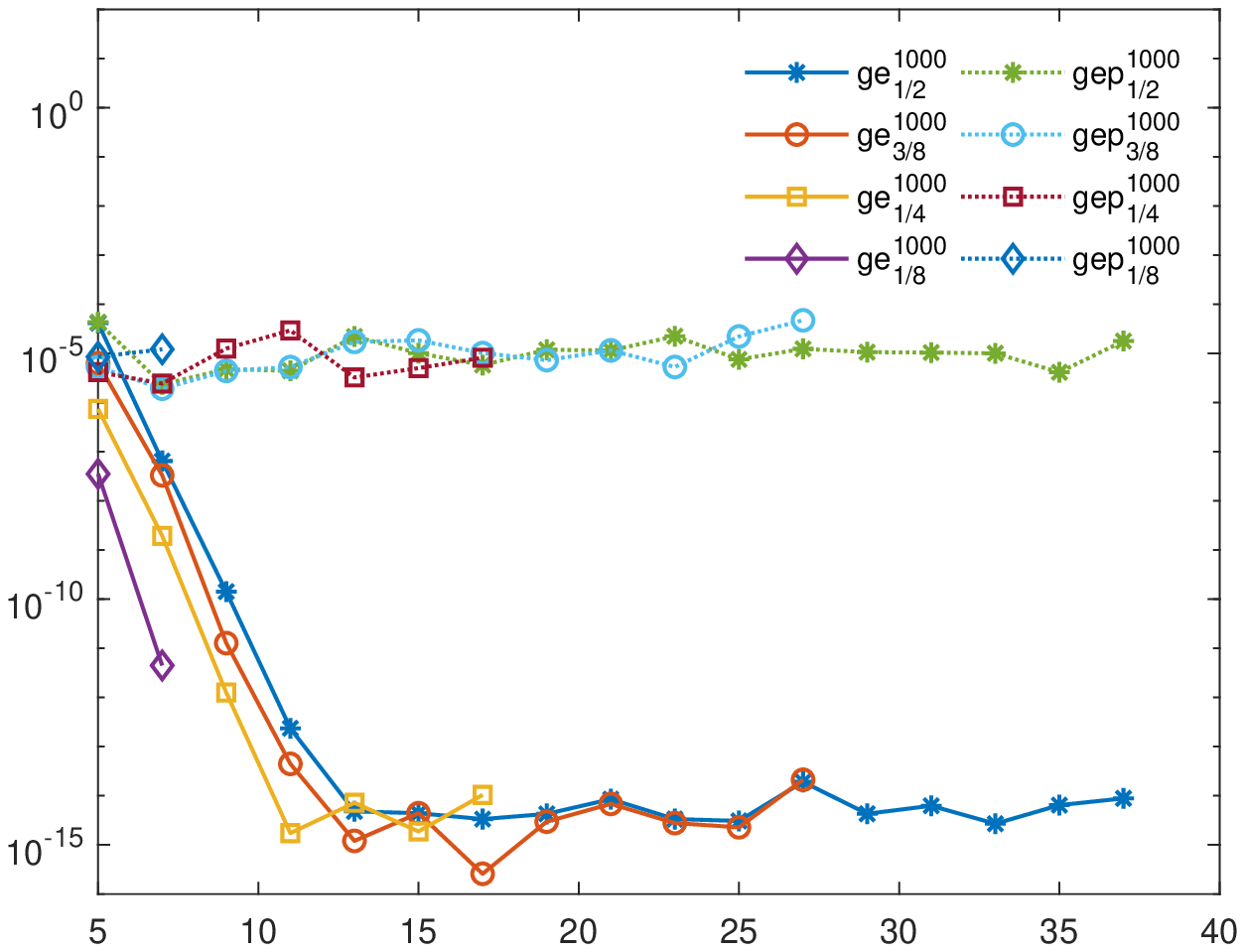}
} 
\parbox{.32\linewidth}{\centering
    \includegraphics[width=1.1\linewidth]{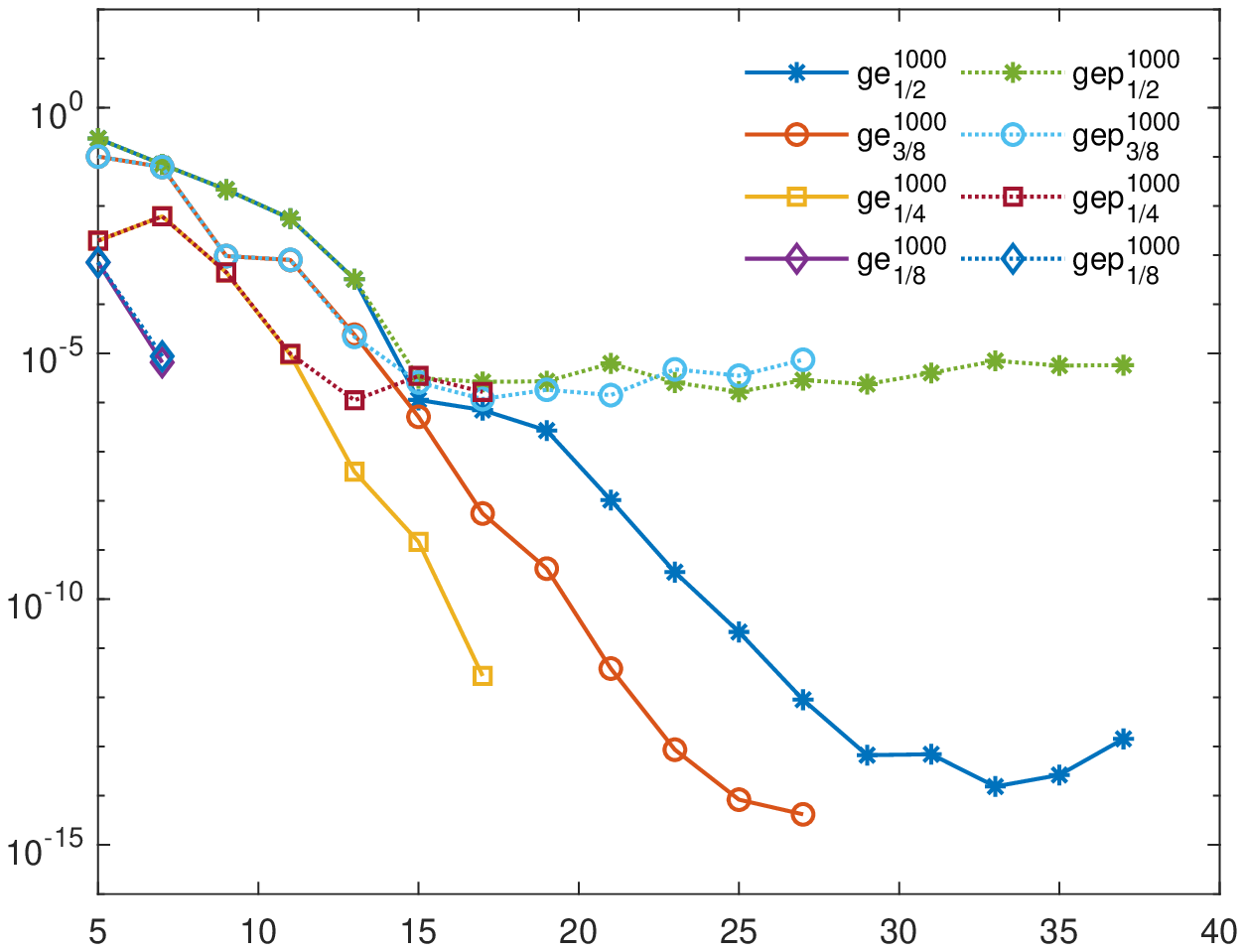}
}}
\caption{Relative error \eqref{gradient_error} for the gradient
computed with exact ($ge$) and perturbed ($gep$) function values 
\eqref{perturbation} for the functions $f_1$ (left), $f_2$ (center) and $f_3$
(right) by using $1000$ Halton points with $\overline{\mathbf{x}}=(0.5,0.5)$.%
}
\label{fig:gradient_perturb}
\end{figure}

\begin{figure}[t]
{\small \centering 
\parbox{.32\linewidth}{\centering
    \includegraphics[width=1.1\linewidth]{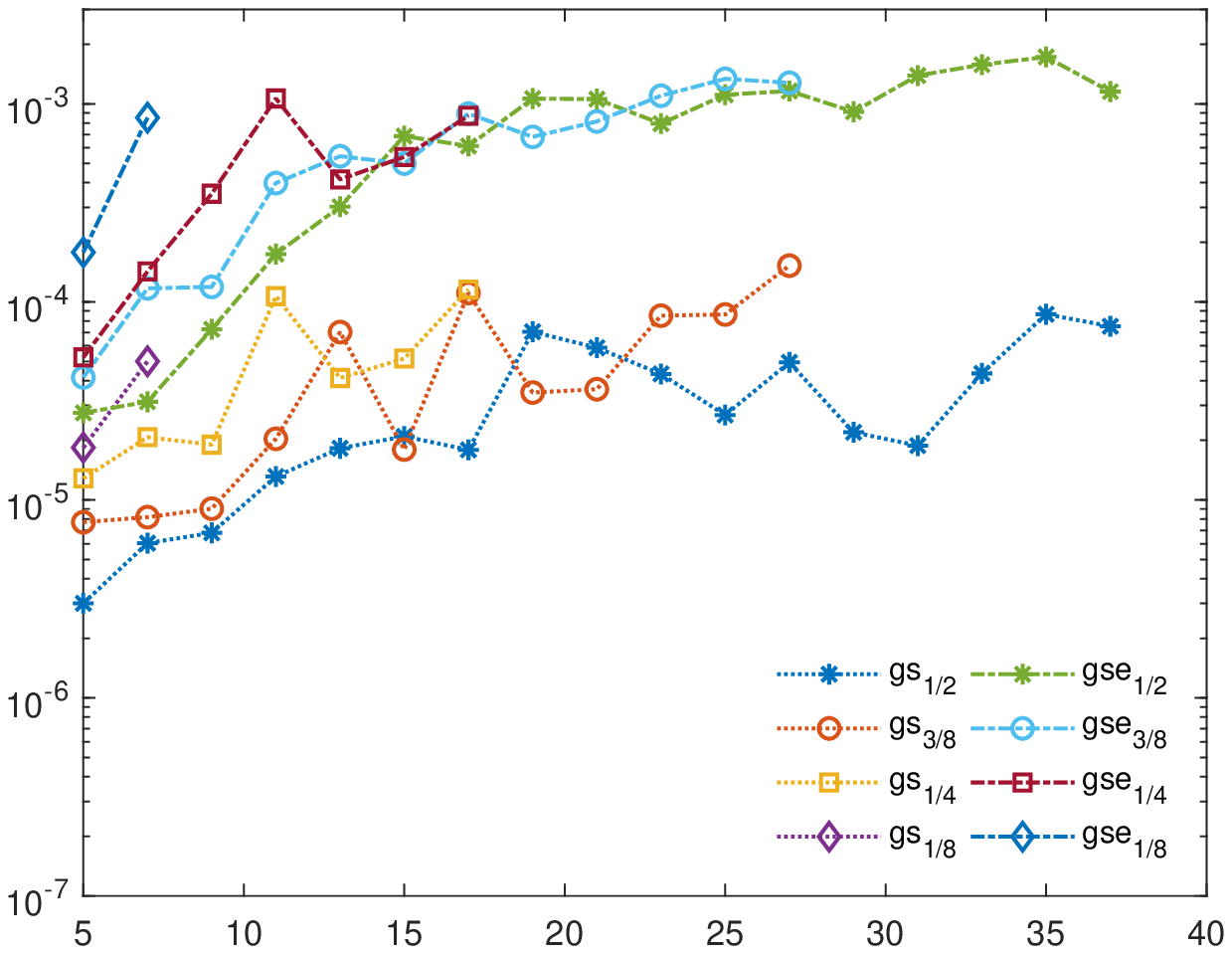}
} 
\parbox{.32\linewidth}{\centering
    \includegraphics[width=1.1\linewidth]{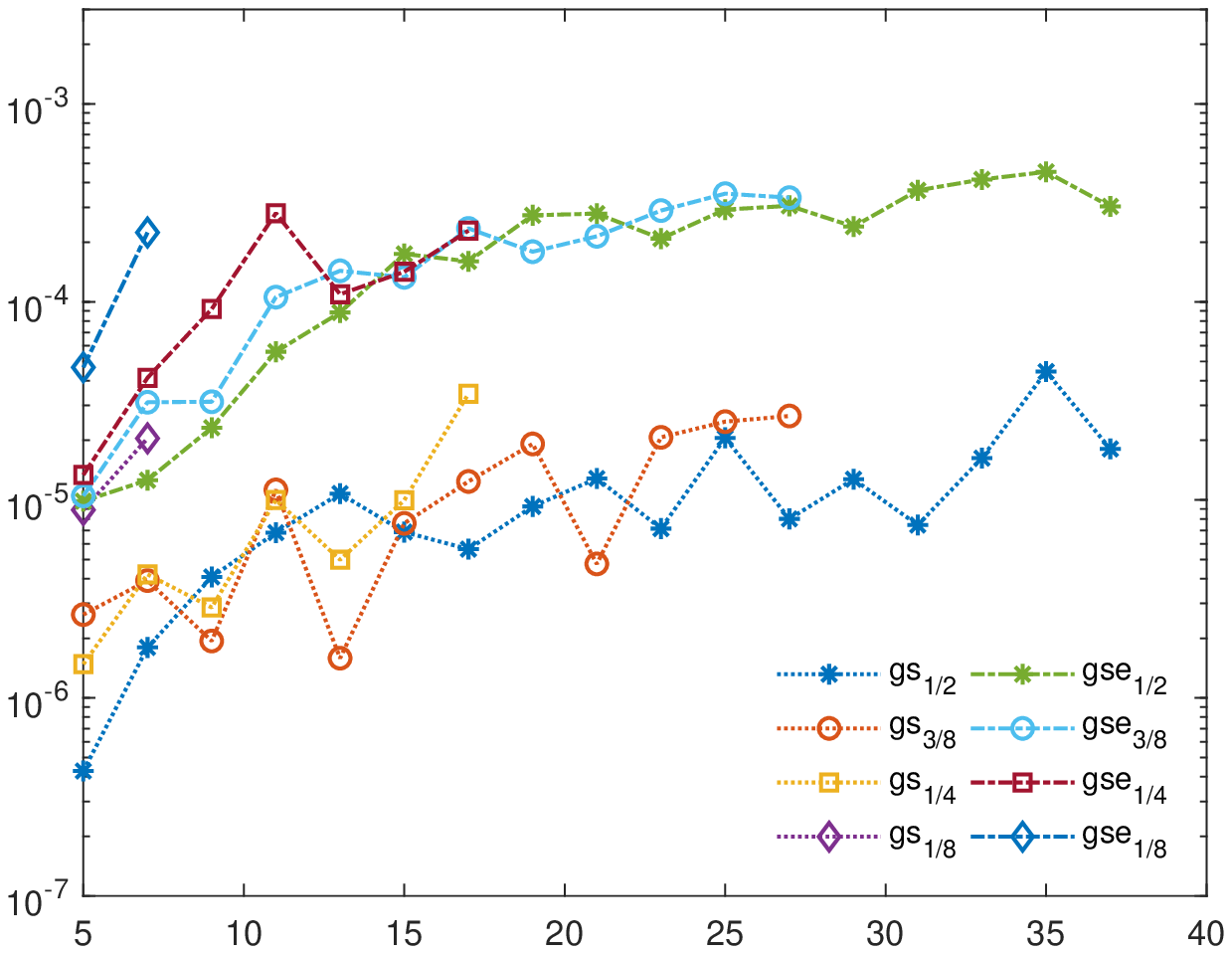}
} 
\parbox{.32\linewidth}{\centering
    \includegraphics[width=1.1\linewidth]{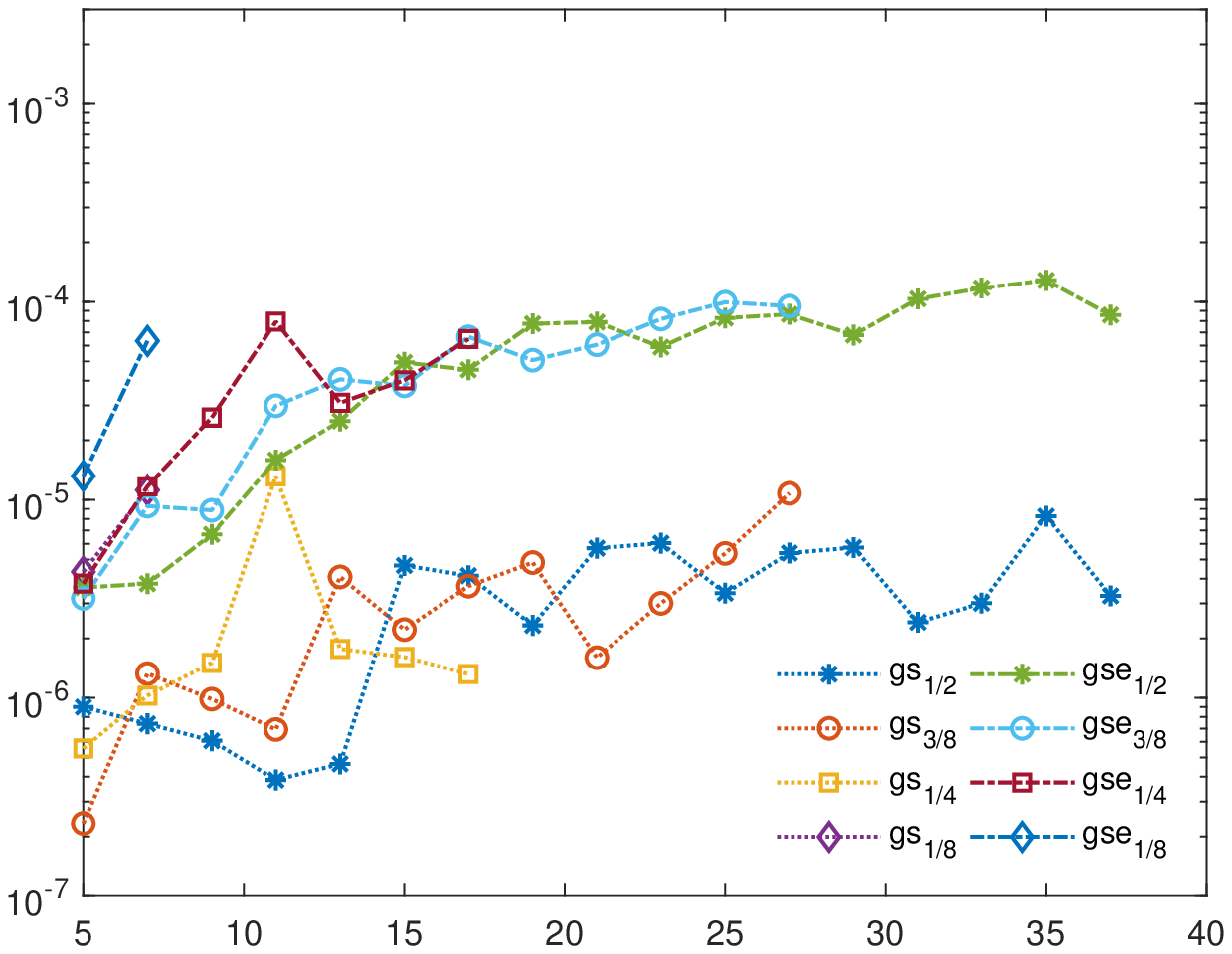}
}}
\caption{Relative sensitivity \eqref{gradient_error_pert} for the
gradient of $p$ computed with perturbed function values ($gs$) and its
estimate \eqref{gradient_error_pert_estimate} involving the stability
constant of the gradient ($gse$) for the functions $f_1$ (left), $f_2$
(center) and $f_3$ (right) by using $1000$ Halton points with $\overline{%
\mathbf{x}}=(0.5,0.5)$.}
\label{fig:gradient_perturb_polynomial}
\end{figure}


\begin{figure}[t]
{\small \centering 
\parbox{.32\linewidth}{\centering
    \includegraphics[width=1.1\linewidth]{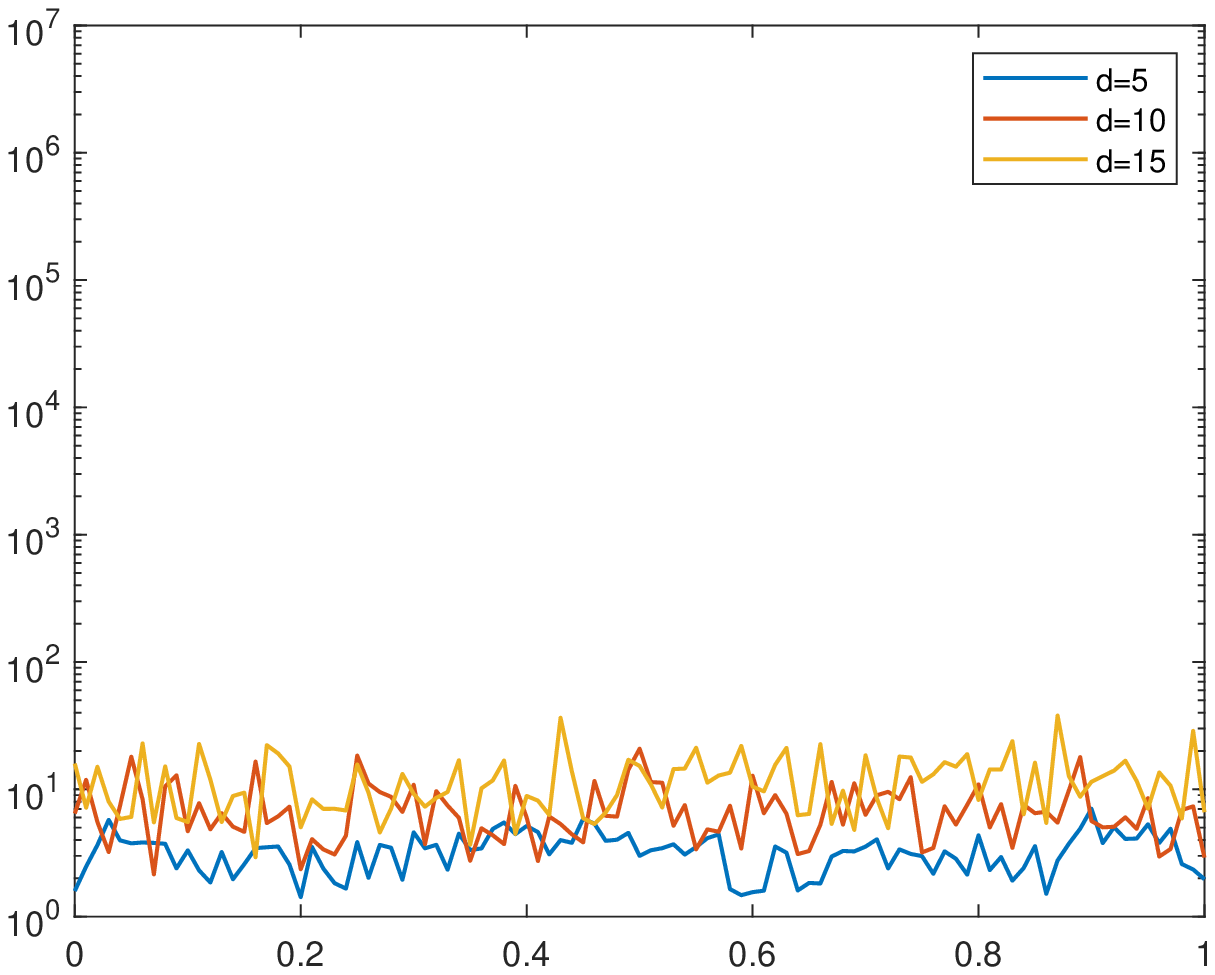}\\
    \includegraphics[width=1.1\linewidth]{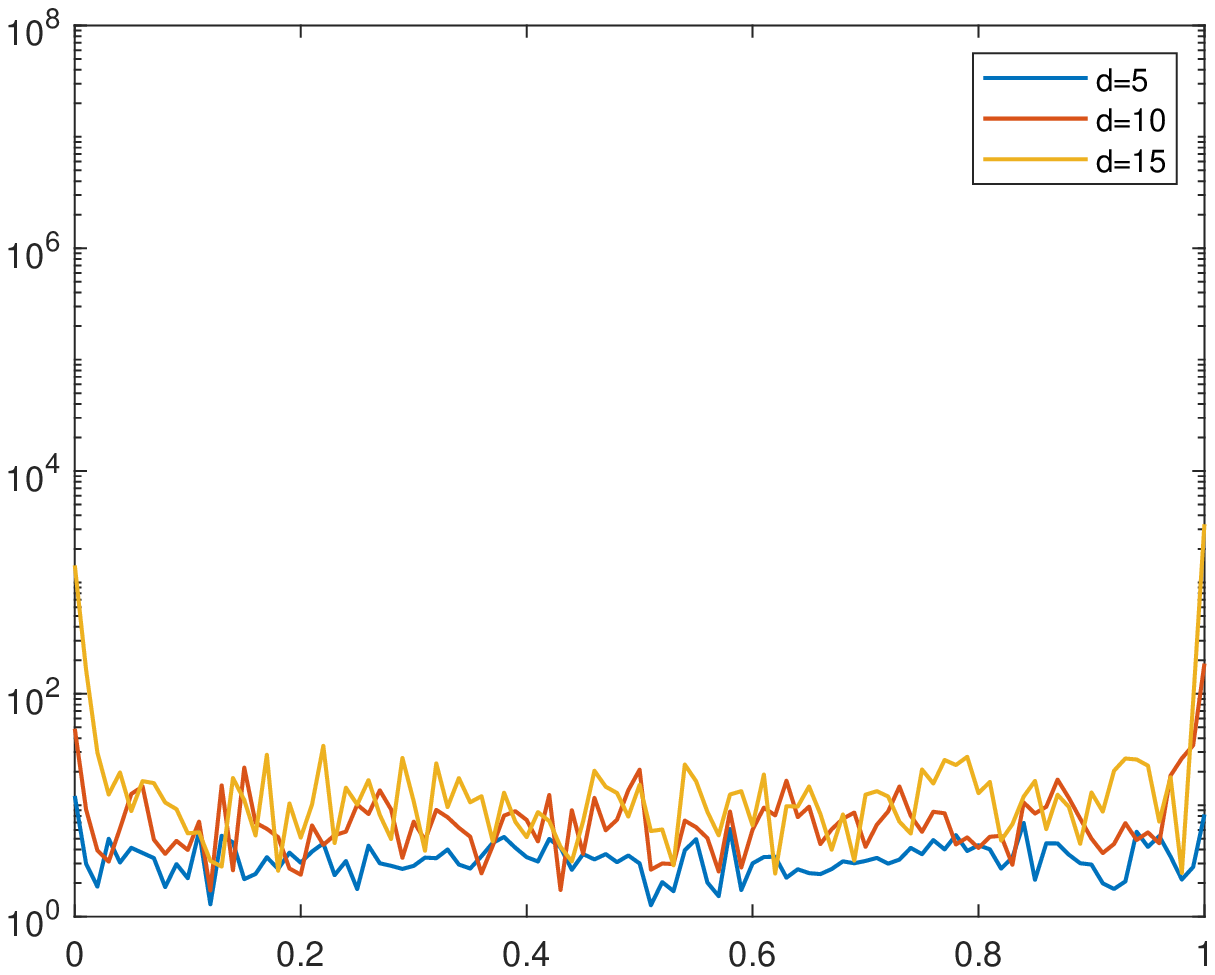}
} 
\parbox{.32\linewidth}{\centering
    \includegraphics[width=1.1\linewidth]{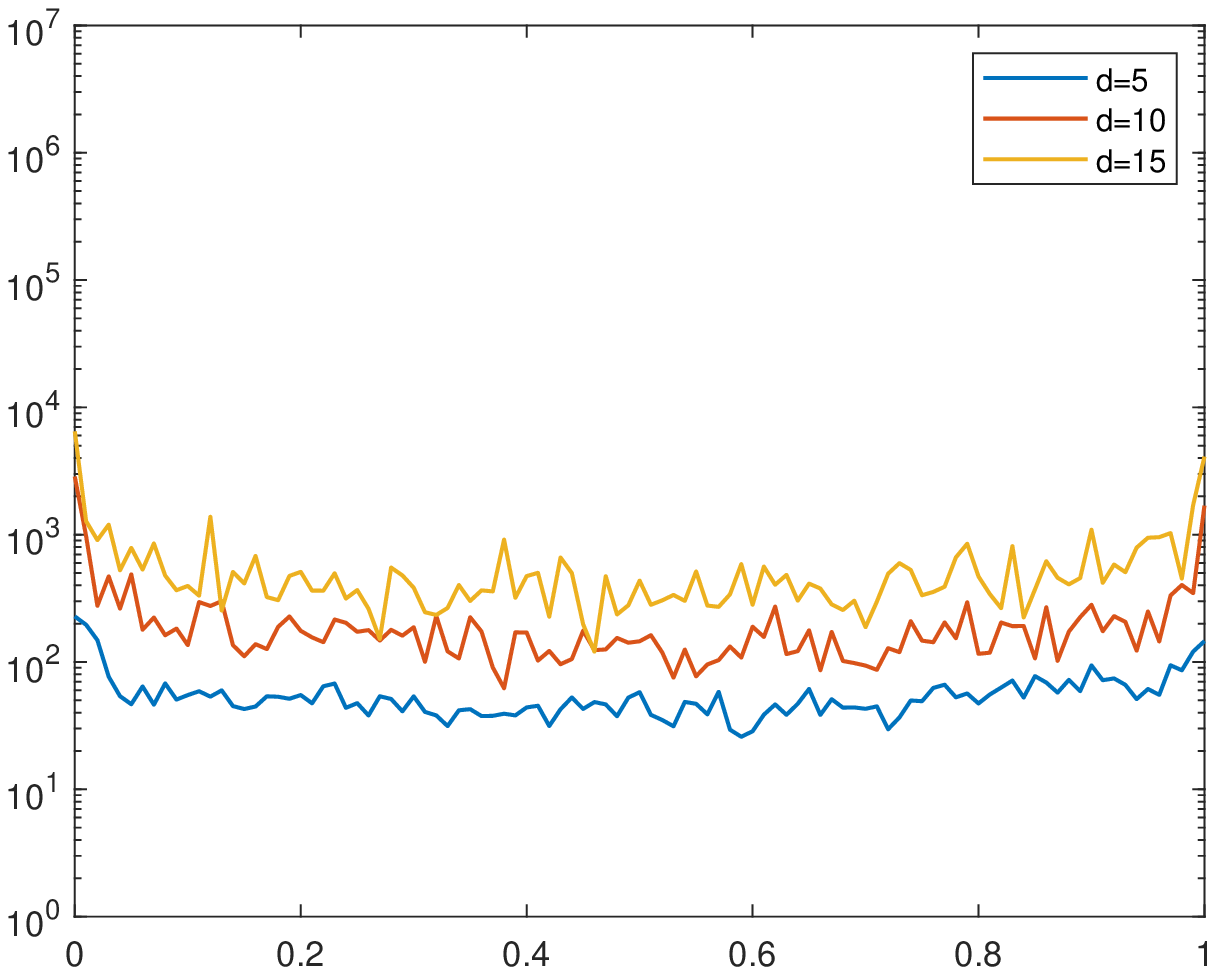}\\
     \includegraphics[width=1.1\linewidth]{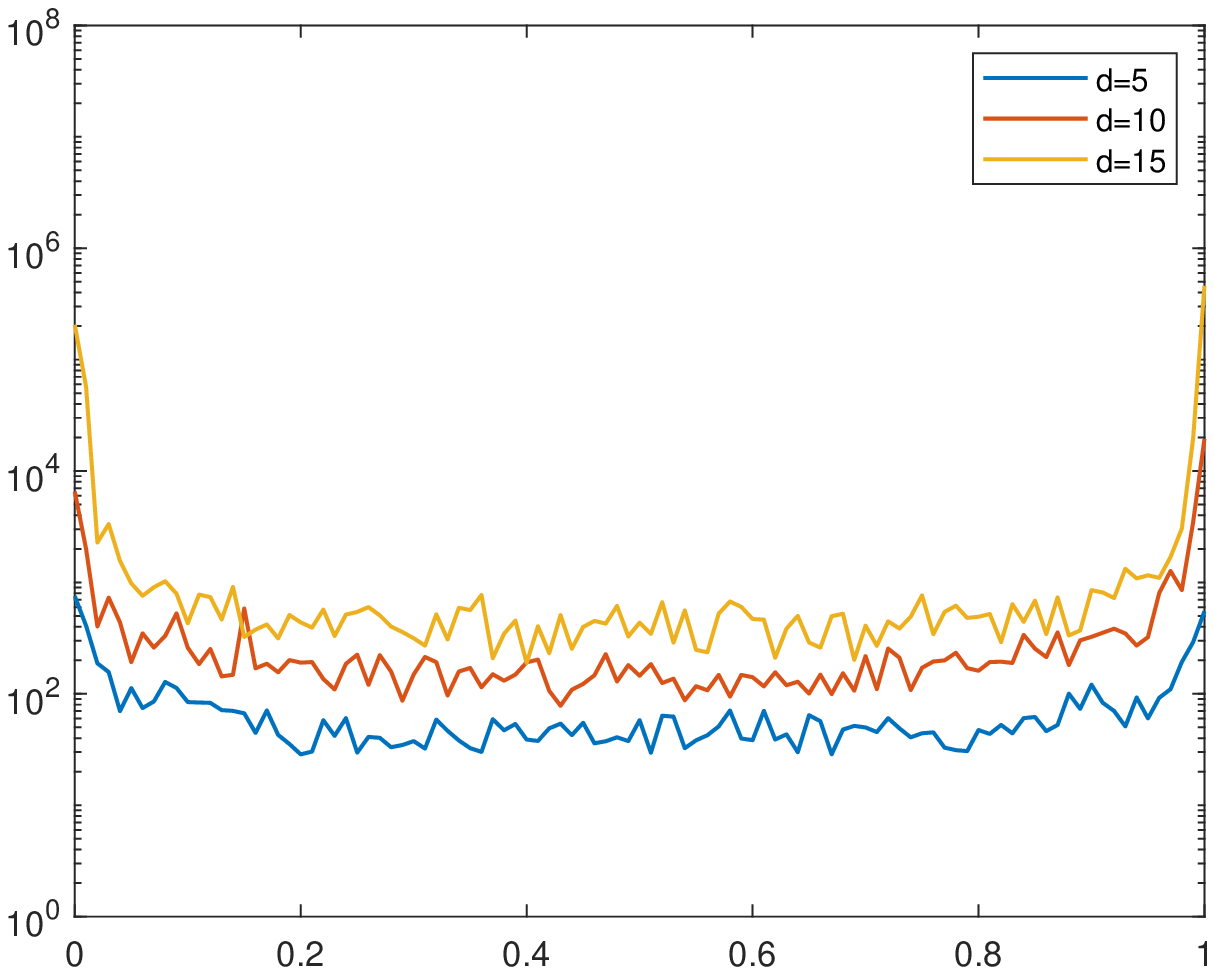}
} 
\parbox{.32\linewidth}{\centering
    \includegraphics[width=1.1\linewidth]{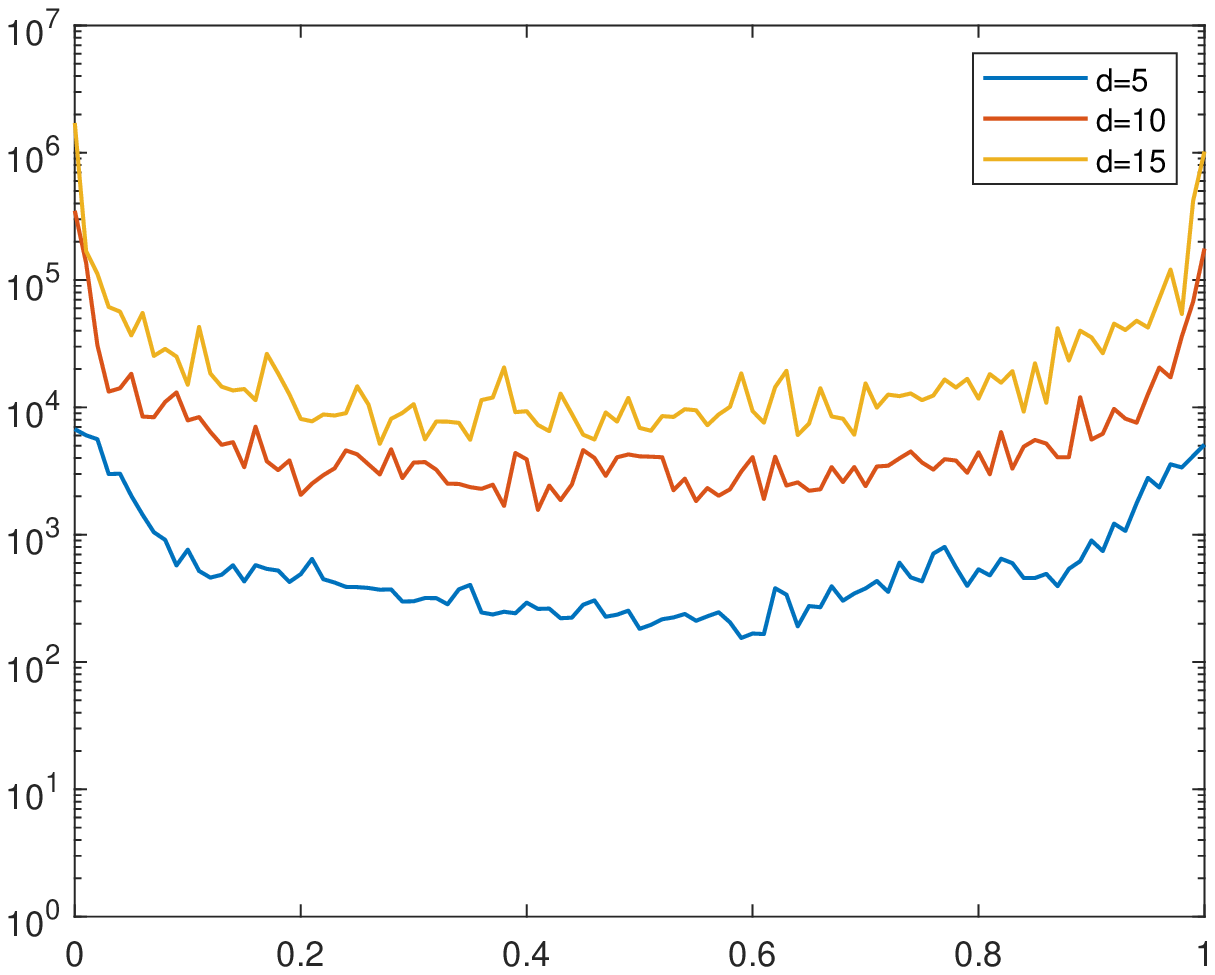}\\
    \includegraphics[width=1.1\linewidth]{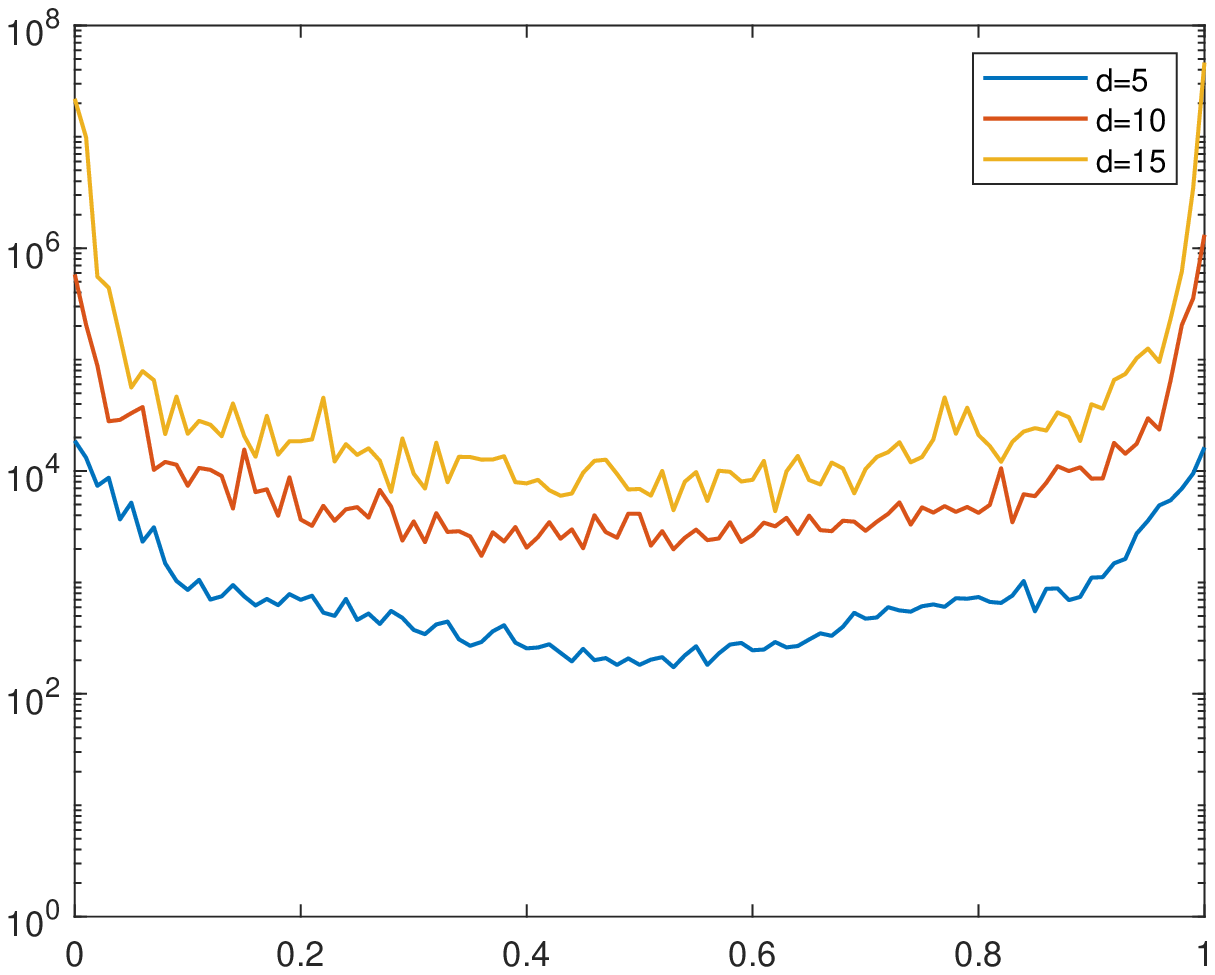}
}
}
\caption{The stability constant \eqref{stabconst}, for degrees $d=5,10,15$ and $\vert \nu \vert=0$ (left), $\vert \nu \vert=1$ (center) and $\vert \nu \vert=2$ (right) computed by using $4000$ Halton points with $r=1/2$ and $\overline{%
\mathbf{x}}$ being $101$ equispaced points on the horizontal line $y=0.5$ (top) and on the diagonal $y=x$ (bottom).}
\label{fig:stability_constant_plot}
\end{figure}

\section*{Acknowledgments}

This research has been achieved as part of RITA \textquotedblleft Research
ITalian network on Approximation'' and was supported by the GNCS-INdAM 2020 Projects ``Interpolation and smoothing: theoretical, computational and applied aspects with emphasis on image processing and data analysis'' and ``Multivariate approximation and functional equations for numerical modelling''. The third author's research was supported by the
National Center for Scientific and Technical Research (CNRST-Morocco) as
part of the Research Excellence Awards Program (No. 103UIT2019). 
The fourth author was partially supported by the DOR funds and the biennial
project BIRD 192932 of the University of Padova. \newline

\bibliographystyle{plain}
\bibliography{references}

\end{document}